\documentclass[12pt,a4paper]{amsart}
\usepackage{amssymb}
\usepackage[all,cmtip]{xy}
\usepackage{hyperref}

\calclayout

\newcommand{\+}{\protect\nobreakdash-}
\renewcommand{\:}{\colon}

\newcommand{\rarrow}{\longrightarrow}
\newcommand{\ot}{\otimes}

\newcommand{\bu}{{\text{\smaller\smaller$\scriptstyle\bullet$}}}
\newcommand{\oc}{\mathbin{\text{\smaller$\square$}}}
\newcommand{\lrarrow}{\mskip.5\thinmuskip\relbar\joinrel\relbar\joinrel
 \rightarrow\mskip.5\thinmuskip\relax} 
\newcommand{\llarrow}{\mskip.5\thinmuskip\leftarrow\joinrel\relbar
 \joinrel\relbar\mskip.5\thinmuskip\relax}

\DeclareMathOperator{\Hom}{Hom}
\DeclareMathOperator{\Ext}{Ext}
\DeclareMathOperator{\Tor}{Tor}
\DeclareMathOperator{\Cohom}{Cohom}
\DeclareMathOperator{\Tot}{Tot}
\DeclareMathOperator{\id}{id}
\DeclareMathOperator{\rd}{rd}
\DeclareMathOperator{\cd}{cd}

\newcommand{\Modl}{{\operatorname{\mathsf{--Mod}}}}

\newcommand{\Hot}{\mathsf{Hot}}
\newcommand{\sD}{\mathsf D}

\newcommand{\dproj}{{\operatorname{\mathsf{-proj}}}}
\newcommand{\dinj}{{\operatorname{\mathsf{-inj}}}}
\newcommand{\proj}{{\mathsf{proj}}}
\newcommand{\inj}{{\mathsf{inj}}}
\newcommand{\sgr}{{\mathsf{gr}}}
\newcommand{\cdg}{{\mathsf{cdg}}}
\newcommand{\co}{{\mathsf{co}}}
\newcommand{\ctr}{{\mathsf{ctr}}}

\newcommand{\acycl}{{\mathsf{acycl}}}

\newcommand{\Fil}{\mathsf{Fil}}
\newcommand{\Cof}{\mathsf{Cof}}
\newcommand{\Prod}{\mathsf{Prod}}
\newcommand{\Add}{\mathsf{Add}}

\newcommand{\A}{\mathcal A}
\newcommand{\F}{\mathcal F}
\newcommand{\G}{\mathcal G}
\newcommand{\C}{\mathcal C}
\newcommand{\D}{\mathcal D}
\renewcommand{\S}{\mathcal S}
\newcommand{\T}{\mathcal T}
\newcommand{\X}{\mathcal X}

\newcommand{\Z}{\mathbb Z}
\newcommand{\Q}{\mathbb Q}

\newcommand{\Section}[1]{\bigskip\section{#1}\medskip}

\theoremstyle{plain}
\newtheorem{thm}{Theorem}[section]

\newtheorem{lem}[thm]{Lemma}
\newtheorem{prop}[thm]{Proposition}
\newtheorem{cor}[thm]{Corollary}
\theoremstyle{definition}
\newtheorem{ex}[thm]{Example}
\newtheorem{exs}[thm]{Examples}
\newtheorem{rem}[thm]{Remark}

\begin{document}

\title{An explicit self-dual construction of complete cotorsion pairs
in the relative context}

\author{Leonid Positselski}

\address{Institute of Mathematics of the Czech Academy of Sciences \\
\v Zitn\'a~25, 115~67 Praha~1 (Czech Republic); and
\newline\indent Laboratory of Algebra and Number Theory \\
Institute for Information Transmission Problems \\
Moscow 127051 (Russia)}

\email{positselski@math.cas.cz}

\begin{abstract}
 Let $R\rarrow A$ be a homomorphism of associative rings, and let
$(\F,\C)$ be a hereditary complete cotorsion pair in $R\Modl$.
 Let $(\F_A,\C_A)$ be the cotorsion pair in $A\Modl$ in which
$\F_A$ is the class of all left $A$\+modules whose underlying
$R$\+modules belong to~$\F$.
 Assuming that the $\F$\+resolution dimension of every left $R$\+module
is finite and the class $\F$ is preserved by the coinduction functor
$\Hom_R(A,{-})$, we show that $\C_A$ is the class of all direct summands
of left $A$\+modules finitely (co)filtered by $A$\+modules coinduced from
$R$\+modules from~$\C$.
 Assuming that the class $\F$ is closed under countable products and
preserved by the functor $\Hom_R(A,{-})$, we prove that $\C_A$ is
the class of all direct summands of left $A$\+modules cofiltered by
$A$\+modules coinduced from $R$\+modules from~$\C$, with the decreasing
filtration indexed by the natural numbers.
 A combined result, based on the assumption that countable products
of modules from $\F$ have finite $\F$\+resolution dimension bounded
by~$k$, involves cofiltrations indexed by the ordinal $\omega+k$.
 The dual results also hold, provable by the same technique going
back to the author's monograph on semi-infinite homological
algebra~\cite{Psemi}.
 In addition, we discuss the $n$\+cotilting and $n$\+tilting cotorsion
pairs, for which we obtain better results using a suitable version of
a classical Bongartz--Ringel lemma.
 As an illustration of the main results of the paper, we consider
certain cotorsion pairs related to the contraderived and coderived
categories of curved DG\+modules.
\end{abstract}

\maketitle

\tableofcontents

\bigskip
\section*{Introduction}
\medskip

 Cotorsion pairs (or in the older terminology, ``cotorsion theories''),
introduced by Salce in~\cite{Sal}, became a standard tool of
the contemporary theory of rings and modules~\cite{GT}.
 The basic idea can be explained in a few words as follows.

 Given an associative ring $A$ and left $A$\+modules $L$ and $M$,
the groups $\Ext_A^n(L,M)$ can be computed either in terms of
a projective resolution of $L$, or using an injective coresolution
of~$M$.
 But what if we wish to use ``partially injective'' and ``partially
projective'' resolutions?
 We want to resolve $L$ by modules that are only somewhat projective,
and coresolve $M$ by modules that are only somewhat injective.
 Can we use such resolutions in order to compute $\Ext_A^n(L,M)$\,?

 As one can see, the answer is positive, provided that the chosen
classes of ``partially injective'' and ``partially projective'' modules
fit each other and one is willing to resolve \emph{both} $L$ and $M$
simultaneously.
 For example, one can choose a flat resolution $F_\bu$ for the module
$L$, and simultaneously choose a coresolution $C^\bu$ of the module
$M$ by so-called \emph{cotorsion $A$\+modules} (in the sense of
Enochs~\cite{En}).
 Then the total complex of the bicomplex $\Hom_A(F_\bu,C^\bu)$
computes $\Ext_A^*(L,M)$.

 Alternatively, let $R\subset A$ be a subring.
 We want to resolve $L$ by $A$\+modules that are \emph{projective as
$R$\+modules}.
 What kind of coresolution of $M$ do we need to use jointly with
such a resolution of $L$, in order to compute the $\Ext$ groups
over~$A$\,?

 The definition of a (\emph{hereditary}) \emph{cotorsion pair}
provides a general answer to such questions.
 A pair of classes of left $A$\+modules $\F$ and $\C\subset A\Modl$
is called a cotorsion pair if $\Ext^1_A(F,C)=0$ for all $F\in\F$
and $C\in\C$, and both the classes $\F$ and $\C$ are maximal with
respect to this property.
 A cotorsion pair $(\F,\C)$ is said to be hereditary if
$\Ext^n_A(F,C)=0$ for all $F\in\F$, \ $C\in\C$, and $n\ge1$.

 In particular, returning to the example above, a left $A$\+module $C$
is said to be (\emph{Enochs}) \emph{cotorsion}~\cite{En} if
$\Ext^1_A(F,C)=0$ for all flat left $A$\+modules $F$, or equivalently,
$\Ext^n_A(F,C)=0$ for all flat $F$ and $n\ge1$.

 More generally, one can consider projective objects, injective objects,
and cotorsion pairs in an abelian category~$\A$.
 In order to compute the groups $\Ext_\A^*$ using projective or
injective resolutions, one needs to have \emph{enough} projectives or
injectives, respectively.
 What does it mean that there are ``enough partially
projective/injective objects'' in a cotorsion pair $(\F,\C)$\,?
 The appropriate definition of this was suggested in~\cite{Sal}, and
it is a strong and unobvious condition.

 Given a cotorsion pair $(\F,\C)$ in $A\Modl$, one says that
\emph{there are enough projectives in $(\F,\C)$} if every left
$A$\+module $L$ is a quotient module of a module $F$ from $\F$
\emph{by a submodule $C'=\ker(F\to L)$ belonging to~$\C$}.
 Similarly, one says that \emph{there are enough injectives in
$(\F,\C)$} if every left $A$\+module $M$ is a submodule of a module
$C$ from $\C$ \emph{with the quotient module $F'=C/M$ belonging
to~$\F$}.
 The short exact sequences $0\rarrow C'\rarrow F\rarrow L\rarrow0$
and $0\rarrow M\rarrow C\rarrow F'\rarrow0$ are called
\emph{approximation sequences}.
 A cotorsion pair $(\F,\C)$ in $A\Modl$ has enough projectives if and
only if it has enough injectives; these assertions are known as
\emph{Salce lemmas}~\cite{Sal}.
 A cotorsion pair having enough projectives (equivalently,
enough injectives) is said to be \emph{complete}.
 In other words, a cotorsion pair $(\F,\C)$ is complete if
approximation sequences with respect to $(\F,\C)$ exist for all
left $A$\+modules.

 The assertion that the \emph{flat cotorsion pair} $(\F,\C)$, where
$\F$ is the class of flat left $A$\+modules and $\C$ is the class of
cotorsion left $A$\+modules, \emph{is complete} became known as
the \emph{flat cover conjecture}.
 It was proved (in two different ways) in the paper~\cite{BBE}.
 
 The most powerful (and the most commonly used) approach to
constructing complete cotorsion pairs known today was developed by
Eklof and Trlifaj~\cite{ET}.
 The Eklof--Trlifaj theorem claims that \emph{any cotorsion pair
generated by a set of modules is complete}.
 Here a cotorsion pair $(\F,\C)$ is said to be generated by a class
of modules $\S\subset A\Modl$ if $\C$ is the class of all left
$A$\+modules $C$ such that $\Ext^1_A(S,C)=0$ for all $S\in\S$.
 Subsequently it was realized that the technique of the Eklof--Trlifaj
construction is a particular case of the so-called \emph{small
object argument} in the homotopy theory or model category theory.
 In fact, a complete cotorsion pair can be thought of as a particular
case of a \emph{weak factorization system}~\cite{Ros,Hov}.

 On the dual side, it is known that \emph{any cotorsion pair
cogenerated by a class of pure-injective modules is
complete}~\cite[Theorem~6.19]{GT}.
 Further alternative approaches to proving completeness of cotorsion
pairs in some special cases are provided by the Bongartz--Ringel
lemma~\cite[Lemma~2.1]{Bon}, \cite[Lemma~4$'$]{Rin},
\cite[Lemma~6.15 and Proposition~6.44]{GT} and the Auslander--Buchweitz
construction~\cite{AB}.

 The aim of this paper is to offer another such alternative approach.
 It is an explicit self-dual construction applicable in the particular
case of cotorsion pairs lifted via the functor of restriction of
scalars $A\Modl\rarrow R\Modl$ with respect to a ring homomorphism
$R\rarrow A$.
 In the most typical situation, $R$ would be a subring in $A$.
 Notice that the small object argument is decidedly \emph{not}
self-dual.
 In fact, it is known to be consistent with ZFC${}+{}$GCH that
the dual version of the Eklof--Trlifaj theorem is not true~\cite{ES}.

 Still, most of the complete cotorsion pairs constructed in this paper
can be easily obtained from the small object argument.
 The main advantage of our approach is that it produces a quite
explicit description of the second class in the cotorsion pair.
 Sometimes this also follows from the Eklof--Trlifaj theorem; but in
other cases it does not.
 In the latter cases, our approach provides new knowledge.

 In the work of the present author, other results somewhat resembling
those of the present paper were obtained in the paper~\cite{PSl},
where descriptions of the right classes in the so-called \emph{strongly
flat} cotorsion pairs, and sometimes also in the flat cotorsion pair,
were provided for categories of modules over commutative rings.
 The constructions of approximation sequences in the present paper
go back to the author's monograph on semi-infinite homological
algebra~\cite{Psemi}.

 Semi-infinite homological algebra, as interpreted in
the book~\cite{Psemi}, is the study of module categories over algebraic
structures which have a mixture of algebra and coalgebra variables in
them.
 These include corings over rings (which means roughly ``coalgebras
over algebras'') and semialgebras over coalgebras (``algebras over
coalgebras''), as well as more complicated semialgebras over corings.

 Relative situations appearing naturally in this context, that is
a coring over a ring or a semialgebra over a coalgebra, tend to be
better behaved than a usual ring over a subring.
 Nevertheless, techniques originally developed in the semi-infinite
context can be transferred to the realm of ring theory.
 That is what we do in this paper.

 Section~\ref{prelim-secn} is an overview of preliminary material.
 The main results of the paper are presented in
Section~\ref{cofiltrations-secn}. 
 In that section, for various cotorsion pairs $(\F_A,\C_A)$ in
the category of $A$\+modules, we describe the right class $\C_A$
as the class of all modules cofiltered by modules of simpler nature.
 The latter means typically the $A$\+modules $\Hom_R(A,C)$ coinduced
from certain $R$\+modules $C$, using a ring homomorphism $R\rarrow A$.
 Here the class $\F_A$ consists of all $A$\+modules whose underlying
$R$\+module belongs to the left class $\F$ of a cotorsion pair
$(\F,\C)$ in $R\Modl$ (while the $R$\+modules $C$ above range over
the class~$\C$).
 Moreover, we show that it suffices to consider rather short
cofiltrations (or, in another language, decreasing filtrations):
depending on the assumptions, these are either finite (co)filtrations,
or cofiltrations indexed by the natural numbers, or indexed by
the ordinal $\omega+k$, where $k$~is a finite integer.

 The dual results are discussed in Section~\ref{filtrations-secn}.
 For various cotorsion pairs $(\F^A,\C^A)$ in $A\Modl$, we describe
the left class $\F^A$ as the class of all modules filtered by modules
of simpler nature.
 The latter means typically the $A$\+modules $A\ot_RF$ induced from
certain $R$\+modules $F$, using a ring homomorphism $R\rarrow A$.
 Here the class $\C^A$ consists of all $A$\+modules whose underlying
$R$\+module belongs to the right class $\C$ of a cotorsion pair
$(\F,\C)$ in $R\Modl$ (while the $R$\+modules $F$ above range over
the class~$\F$).
 The results of Section~\ref{filtrations-secn} are generally less
surprising, from the point of view of the contemporary module theory,
than those of Section~\ref{cofiltrations-secn}, in that a description
of the left class in terms of filtrations is provided, for many
cotorsion pairs, by the small object argument.
 Still, we obtain some new information, in the sense that
the filtrations which we construct are rather short (either finite,
or indexed by the natural numbers, or by the ordinal~$\omega+k$).

 We also discuss the $n$\+cotilting and $n$\+tilting cotorsion pairs
(see~\cite[Chapters~13--15]{GT}), for which it turns out that
the conventional techniques of the tilting theory allow to obtain better
results than our ``semi-infinite'' approach.
 In this connection we introduce a generalized ($n\ge1$) version of
the classical ($n=1$) Bongartz lemma~\cite[Lemma~6.15 and
Proposition~6.44]{GT}, or which is the same, an infinitely generated
version of a lemma of Ringel~\cite[Lemma~4$'$]{Rin}, and use it
to extend a recent result of \v Saroch and
Trlifaj~\cite[Example~2.3]{ST} to $n\ge2$.
 This material is presented in Sections~\ref{cotilting-subsecn}
and~\ref{tilting-subsecn}.

 As an illustration for the main results of the paper, we produce
certain cotorsion pairs in the abelian categories of curved
DG\+modules over some curved DG\+rings.
 These are hereditary, complete cotorsion pairs related to
the contraderived and coderived abelian model structures, as
constructed in~\cite[Section~1.3]{Bec}.
 The idea to consider these cotorsion pairs was suggested to the author
by J.~\v St\!'ov\'\i\v cek.
 We obtain almost no new results in this direction (some general
theorems about filtrations and cofiltrations indexed by countable
ordinals are notable exceptions).
 However, our approach allows us to obtain new proofs of the results
of the memoir~\cite{Pkoszul} concerning the contraderived and coderived
categories of CDG\+modules~\cite[Theorems~3.6, 3.7 and~3.8]{Pkoszul},
interpreting these essentially as a particular case of our results on
cotorsion pairs arising from ring homomorphisms.
 This is the material of Section~\ref{contra-co-derived-secn}.

\subsection*{Acknowledgement}
 I~wish to thank Roman Bezrukavnikov for motivating correspondence.
 I~am grateful to Jan \v St\!'ov\'\i\v cek, Jan Trlifaj, and
Silvana Bazzoni for illuminating discussions from which this work
benefited greatly.
 I~also want to thank Michal Hrbek for helpful conversations and
comments.
 I~wish to thank an anonymous referee for reading the paper
carefully and spotting a number of misprints; the suggestion
to include Proposition~\ref{small-object-argument-for-coherent}
is due to the referee.
 The author is supported by the GA\v CR project 20-13778S and
research plan RVO:~67985840.

\Section{Preliminaries}  \label{prelim-secn}

 All \emph{rings} and \emph{algebras} in this paper are presumed to be
associative and unital.
 All \emph{ring homomorphisms} take the unit to the unit, and all
\emph{modules} are unital.

 Let $A$ be a ring.
 We denote by $A\Modl$ the abelian category of left $A$\+modules.
 For any left $A$\+module $M$, we denote by $\Add(M)=\Add_A(M)\subset
A\Modl$ the class of all direct summands of direct sums $M^{(I)}$ of
copies of the $A$\+module $M$, indexed by arbitrary sets~$I$.
 Similarly, we let $\Prod(M)=\Prod_A(M)\subset A\Modl$ denote the class
of all direct summands of products $M^I$ of copies of
the $A$\+module~$M$.

 For any $A$\+module $M$, choose a projective resolution
$\dotsb\rarrow P_2\rarrow P_1\rarrow P_0\rarrow M\rarrow0$ and
an injective coresolution $0\rarrow M\rarrow J^0\rarrow J^1\rarrow J^2
\rarrow\dotsb$.
 For every $i\ge0$, denote by $\Omega^iM$ the cokernel of the morphism
$P_{i+1}\rarrow P_i$ and by $\Omega^{-i}M$ the kernel of the morphism
$J^i\rarrow J^{i+1}$.
 So, in particular, $\Omega^0M=M$, and our notation is consistent.
 The $A$\+modules $\Omega^iM$ are called the \emph{syzygy modules} of
the $A$\+module $M$, while the $A$\+modules $\Omega^{-i}M$ are called
the \emph{cosyzygy modules} of~$M$.

 To emphasize that $M$ is viewed as a (left) $A$\+module, we will
sometimes use the notation~${}_AM$.
 If $R\rarrow A$ is a ring homomorphism, then the underlying left
$R$\+module of $M$ will be sometimes denoted by~${}_RM$.

 Given a ring homomorphism $R\rarrow A$ and a left $R$\+module $L$,
the left $A$\+module $A\ot_RL$ is said to be \emph{induced from}
the left $R$\+module~$L$.
 The left $A$\+module $\Hom_R(A,L)$ is said to be \emph{coinduced from}
the left $R$\+module~$L$.

\subsection{$\Ext^1$-orthogonal classes}
 We say that two left $A$\+modules $F$ and $C$ are
\emph{$\Ext^1$\+orthogonal} if $\Ext_A^1(F,C)=0$.
 Two classes of left $A$\+modules $\F$ and $\C\subset A\Modl$ are
called \emph{$\Ext^1$\+orthogonal} if $\Ext_A^1(F,C)=0$ for all
$F\in\F$ and $C\in\C$.

 Given a class of left $A$\+modules $\F\subset A\Modl$, we denote
by $\F^{\perp_1}\subset A\Modl$ the class of all left $A$\+modules $X$
such that $\Ext_A^1(F,X)=0$ for all $F\in\F$.
 Similarly, given a class of left $A$\+modules $\C\subset A\Modl$,
we denote by ${}^{\perp_1}\C\subset A\Modl$ the class of all
left $A$\+modules $Y$ such that $\Ext_A^1(Y,C)=0$ for all $C\in\C$.

 Clearly, the classes $\F^{\perp_1}$ and ${}^{\perp_1}\C$ are
closed under extensions and direct summands in $A\Modl$.
 The class $\F^{\perp_1}$ contains all injective left $A$\+modules,
while the class ${}^{\perp_1}\C$ contains all projective left
$A$\+modules.

 A pair of classes of left $A$\+modules $(\F,\C)$ is said to be
a \emph{cotorsion pair} if $\C=\F^{\perp_1}$ and $\F={}^{\perp_1}\C$.
 In other words, $(\F,\C)$ is called a cotorsion pair if both $\F$
and $\C$ are the maximal classes with the property of being
$\Ext^1$\+orthogonal to each other.

 For any class of left $A$\+modules $\S\subset A\Modl$, the pair of
classes $\F={}^{\perp_1}(\S^{\perp_1})$ and $\C=\S^{\perp_1}$ is
a cotorsion pair in $A\Modl$.
 We will say that the cotorsion pair $(\F,\C)$ is \emph{generated
by}~$\S$.
 The class~$\F$ is also said to be generated by~$\S$.

 Dually, for any class of left $A$\+modules $\T\subset A\Modl$,
the pair of classes $\F={}^{\perp_1}\T$ and
$\C=({}^{\perp_1}\T)^{\perp_1}$ is a cotorsion pair in $A\Modl$.
 We will say that the cotorsion pair $(\F,\C)$ is \emph{cogenerated
by}~$\T$.
 The class $\C$ is also said to be cogenerated by~$\T$.

 The following variation of the above notation will be also useful.
 Given a class of left $A$\+modules $\F$ and an integer $j\ge0$,
we denote by $\F^{\perp_{>j}}\subset A\Modl$ the class of all
left $A$\+modules $X$ such that $\Ext_A^n(F,X)=0$ for all $F\in\F$
and $n>j$.
 Similarly, given a class of left $A$\+modules $\C$, we denote by
${}^{\perp_{>j}}\C\subset A\Modl$ the class of all left $A$\+modules
$Y$ such that $\Ext_A^n(Y,C)=0$ for all $C\in\C$ and $n>j$.

\subsection{Approximation sequences}
 Let $\F$ and $\C\subset A\Modl$ be two $\Ext^1$\+orthogonal classes
of left $A$\+modules.
 We will say that $\F$ and $\C$ \emph{admit approximation sequences}
if, for every left $A$\+module $M$, there exist short exact
sequences of left $A$\+modules
\begin{gather}
0\lrarrow C'\lrarrow F\lrarrow M\lrarrow 0,
\label{sp-precover-seq} \\
0\lrarrow M\lrarrow C\lrarrow F'\lrarrow 0\hphantom{,}
\label{sp-preenvelope-seq}
\end{gather}
with $F$, $F'\in\F$ and $C$, $C'\in\C$.

 An approximation sequence~\eqref{sp-precover-seq} is called
a \emph{special precover sequence}, and the surjective morphism
$F\rarrow M$ is called a \emph{special precover}.
 An approximation sequence~\eqref{sp-preenvelope-seq} is called
a \emph{special preenvelope sequence}, and the injective morphism
$M\rarrow C$ is called a \emph{special preenvelope}.

\begin{lem}[Salce~\cite{Sal}] \label{salce-lemma}
 Let $(\F,\C)$ be an\/ $\Ext^1$\+orthogonal pair of classes of modules,
both of them closed under extensions in $A\Modl$.
 Assume that every left $A$\+module is a quotient module of
a module from $\F$ and a submodule of a module from~$\C$.
 Then a special precover sequence~\eqref{sp-precover-seq} exists
for every left $A$\+module $M$ \emph{if and only if} a special
preenvelope sequence~\eqref{sp-preenvelope-seq} exists for every
left $A$\+module~$M$.
\end{lem}

\begin{proof}
 Let us prove the ``if''.
 Let $M$ be a left $A$\+module, and let $E\in\F$ be a module for
which there exists a surjective $A$\+module morphism $E\rarrow M$.
 Let $N$ be the kernel of this morphism; so we have a short
exact sequence $0\rarrow N\rarrow E\rarrow M\rarrow0$.
 Let $0\rarrow N\rarrow C\rarrow F\rarrow 0$ be a special
preenvelope sequence for the left $A$\+module~$N$, i.~e., $C\in\C$
and $F\in\F$.
 Denote by $H$ the pushout (that is, in other words, the fibered
coproduct) of the pair of morphisms $N\rarrow E$ and $N\rarrow C$.
 So $H$ is the cokernel of the diagonal morphism $N\rarrow E\oplus C$.
 Then there are short exact sequences $0\rarrow E\rarrow H\rarrow F
\rarrow0$ and $0\rarrow C\rarrow H\rarrow M\rarrow0$.
 Now the former sequence shows that $H\in\F$, and the latter one is
the desired special precover sequence for the $A$\+module~$M$.
 The proof of the ``only if'' implication is dual.
\end{proof}

 Let $(\F,\C)$ be a cotorsion pair in $A\Modl$.
 Then it is clear from Lemma~\ref{salce-lemma} that the pair $(\F,\C)$
admits special precover sequences if and only if it admits special
preenvelope sequences.
 In this case, the cotorsion pair $(\F,\C)$ is said to be
\emph{complete}.

 Given a class of modules $\A\subset A\Modl$, denote by $\A^\oplus
\subset A\Modl$ the class of all direct summands of modules from~$\A$.

\begin{lem} \label{direct-summand-lemma}
 Let $(\F,\C)$ be an\/ $\Ext^1$\+orthogonal pair of classes of left
$A$\+modules admitting approximation sequences.
 Then $(\F^\oplus,\C^\oplus)$ is a complete cotorsion pair in
$A\Modl$.
\end{lem}

\begin{proof}
 Since $(\F,\C)$ is an\/ $\Ext^1$\+orthogonal pair of classes of
modules admitting approximation sequences, it follows immediately
that the pair of classes $\F^\oplus$ and $\C^\oplus$ has the same
properties.
 So it only remains to show that $\F^{\perp_1}\subset\C^\oplus$
and ${}^{\perp_1}\C\subset\F^\oplus$.
 Indeed, let $M$ be a left $A$\+module belonging to $\F^{\perp_1}$.
 By assumption, there exists a short exact sequence of left
$A$\+modules $0\rarrow M\rarrow C\rarrow F'\rarrow0$ with
$C\in\C$ and $F'\in\F$.
 Since $\Ext^1_A(F',M)=0$, it follows that $M$ is a direct summand
of~$C$.
\end{proof}

\subsection{Filtrations and cofiltrations}
 We consider ordinal-indexed smooth increasing filtrations (called
for brevity simply ``filtrations'') and ordinal-indexed smooth
decreasing filtrations (called ``cofiltrations'').
 In the main results of this paper, we will mostly deal with 
(co)filtrations by rather small ordinals, such as the ordinal of
natural numbers~$\omega$; but here we discuss the general case.

 Let $\alpha$~be an ordinal and $M$ be an $A$\+module.
 An \emph{$\alpha$\+filtration} on $M$ is a collection of submodules
$F_iM\subset M$ indexed by the ordinals $0\le i\le\alpha$ such that
\begin{itemize}
\item $F_0M=0$, \ $F_\alpha M=M$, and $F_j M\subset F_iM$ for all
$0\le j\le i\le\alpha$;
\item $F_iM=\bigcup_{j<i}F_jM$ for all limit ordinals $i\le\alpha$.
\end{itemize}
 An $A$\+module $M$ with an $\alpha$\+filtration $F$ is said to
be \emph{filtered} (or \emph{$\alpha$\+filtered}) by the $A$\+modules
$F_{i+1}M/F_iM$, \ $0\le i<\alpha$.

 Given a class of $A$\+modules $\S\subset A\Modl$, an $A$\+module $M$
is said to be \emph{$\alpha$\+filtered by $\S$} if $M$ admits
an $\alpha$\+filtration $F$ such that the successive quotient
module $F_{i+1}M/F_iM$ is isomorphic to a module from $\S$ for every
$0\le i<\alpha$.
 An $A$\+module is said to be \emph{filtered by~$\S$} if it is
$\alpha$\+filtered by $\S$ for some ordinal~$\alpha$.

 The class of all $A$\+modules filtered by $\S$ is denoted by
$\Fil(\S)\subset A\Modl$, and the class of all $A$\+modules
$\alpha$\+filtered by $\S$ is denoted by $\Fil_\alpha(\S)\subset
\Fil(\S)$.
 It is convenient to assume that $0\in\S$, guaranteeing that
$\Fil_\alpha(\S)\subset\Fil_\beta(\S)$ whenever $\alpha\le\beta$.

 Let $\alpha$ and $\beta$ be two ordinals.
 We denote, as usually, by $\alpha\cdot\beta=\bigsqcup_\beta\alpha$
the ordinal product of $\alpha$ and~$\beta$.
 This means the ordinal which is order isomorphic to the well-ordered
set of pairs $\{(i,j)\mid 0\le i<\alpha,\>0\le j<\beta\}$ with
the lexicographical order, $(i',j')<(i'',j'')$ if either $j'<j''$, or
$j'=j''$ and $i'< i''$.
 
\begin{lem}
 For any class of $A$\+modules $\S\subset A\Modl$, one has \par
\textup{(a)} $\Fil_\beta(\Fil_\alpha(\S))=
\Fil_{\alpha\cdot\beta}(\S)$; \par
\textup{(b)} $\Fil_\alpha(\S^\oplus)\subset
\Fil_\alpha(\S)^\oplus$. \qed
\end{lem}

 The following result is known as the Eklof lemma.

\begin{lem} \label{eklof-lemma}
 For any class of left $A$\+modules $\S$, one has\/
$\Fil(\S)^{\perp_1}=\S^{\perp_1}$.
\end{lem}

\begin{proof}
 This is~\cite[Lemma~1]{ET} or~\cite[Lemma~6.2]{GT}.
\end{proof}

 The next result is called the Eklof--Trlifaj theorem.

\begin{thm} \label{eklof-trlifaj-theorem}
 Let $\S$ be a set (rather than a class) of left $A$\+modules,
and let $(\F,\C)$ be the cotorsion pair in $A\Modl$ generated by~$\S$.
 Then \par
\textup{(a)} $(\F,\C)$ is a complete cotorsion pair; \par
\textup{(b)} the class $\F$ can be described as
$\F=\Fil(\S\cup\{A\})^\oplus$, where $A$ denotes the free left
$A$\+module with one generator.
\end{thm}

\begin{proof}
 Part~(a) is~\cite[Theorem~10]{ET} or~\cite[Theorem~6.11]{GT},
and part~(b) is~\cite[Corollary~6.13 or~6.14]{GT}.
 Essentially, one proves by an explicit construction (a particular
case of the small object argument) that the pair of classes
$\Fil(\S)$ and~$\S^{\perp_1}\subset A\Modl$ admits special preenvelope
sequences, and then by Lemma~\ref{salce-lemma} it follows that
the pair of classes $\Fil(\S\cup\{A\})$ and $\S^{\perp_1}$ admits
special precover sequences.
 The two classes $\Fil(\S\cup\{A\})$ and $\S^{\perp_1}$ are
$\Ext^1$\+orthogonal by Lemma~\ref{eklof-lemma}.
 By Lemma~\ref{direct-summand-lemma}, one can conclude that the two
classes $\Fil(\S\cup\{A\})^\oplus$ and $\S^{\perp_1}$ form
a complete cotorsion pair.
 By the definition, we have $\C=\S^{\perp_1}$, and it follows that
$\F=\Fil(\S\cup\{A\})^\oplus$.
\end{proof}

 Let $\alpha$~be an ordinal and $N$ be a left $A$\+module.
 An \emph{$\alpha$\+cofiltration} on $N$ is a collection of
left $A$\+modules $G_iN$ indexed by the ordinals $0\le i\le\alpha$
and left $A$\+module morphisms $G_iN\rarrow G_jN$ defined for
all $0\le j < i\le\alpha$ such that
\begin{itemize}
\item the triangle diagram $G_iN\rarrow G_jN\rarrow G_kN$ is
commutative for all triples of indices $0\le k<j<i\le\alpha$;
\item $G_0N=0$ and $G_\alpha N=N$;
\item the induced map into the projective limit
$G_iN\rarrow\varprojlim_{j<i}G_jN$ is an isomorphism for all
limit ordinals $i\le\alpha$;
\item the map $G_{i+1}N\rarrow G_iN$ is surjective for all
$0\le i<\alpha$.
\end{itemize}
 It follows from the above list of conditions that the map
$G_iN\rarrow G_jN$ is surjective for all $0\le j<i\le\alpha$.
 An $A$\+module $N$ with an $\alpha$\+cofiltration $G$ is said to
be \emph{cofiltered} (or \emph{$\alpha$\+cofiltered}) by
the $A$\+modules $\ker(G_{i+1}N\to G_iN)$.

 Given a class of $A$\+modules $\T\subset A\Modl$, an $A$\+module
$N$ is said to be \emph{$\alpha$\+cofiltered by $\T$} if $N$
admits an $\alpha$\+cofiltration $G$ such that the $A$\+module
$\ker(G_{i+1}N\to G_iN)$ is isomorphic to an $A$\+module from $\T$
for all $0\le i<\alpha$.
 An $A$\+module is said to be \emph{cofiltered by $\T$} if it is
$\alpha$\+cofiltered by $\T$ for some ordinal~$\alpha$.

 The class of all $A$\+modules cofiltered by $\T$ is denoted by
$\Cof(\T)\subset A\Modl$, and the class of all $A$\+modules
$\alpha$\+cofiltered by $\T$ is denoted by $\Cof_\alpha(\T)\subset
\Cof(\T)$.
 It is convenient to assume that $0\in\T$, so that
$\Cof_\alpha(\T)\subset\Cof_\beta(\T)$ whenever $\alpha\le\beta$.

\begin{lem} \label{iterated-cof-lemma}
 For any class of $A$\+modules $\T\subset A\Modl$ and any two
ordinals $\alpha$ and~$\beta$, one has \par
\textup{(a)} $\Cof_\beta(\Cof_\alpha(\T))=
\Cof_{\alpha\cdot\beta}(\T)$; \par
\textup{(b)} $\Cof_\alpha(\T^\oplus)\subset\Cof_\alpha(\T)^\oplus$.
\end{lem}

\begin{proof}
 Part~(b) is obvious.
 The proof of part~(a) is left to the reader.
\end{proof}

 The following assertion is known as the Lukas lemma or ``the dual
Eklof lemma''.

\begin{lem} \label{dual-eklof-lemma}
 For any class of left $A$\+modules $\T$, one has
${}^{\perp_1}\Cof(\T)={}^{\perp_1}\T$.
\end{lem}

\begin{proof}
 This is~\cite[Proposition~18]{ET} or~\cite[Lemma~6.37]{GT}.
\end{proof}

 The dual version of the small object argument does not work in
module categories, because most modules are not cosmall.
 In fact, it is consistent with ZFC that the dual version of
Theorem~\ref{eklof-trlifaj-theorem}(a) is not true.

 Specifically, let $A=\Z$ be the ring of integers, so $A\Modl$ is
the category of abelian groups.
 Let $\T=\{\Z\}$ be the set consisting of one infinite cyclic
abelian group only; and let $\Q$ denote the additive group of
rational numbers.
 Let $(\mathcal W,\mathcal W^{\perp_1})$ be the cotorsion pair
in $\Z\Modl$ cogenerated by~$\T$; the class $\mathcal W ={}^{\perp_1}\T$
is known as the class of all \emph{Whitehead groups}.
 According to~\cite[Theorem~0.4]{ES}, it is consistent with
ZFC${}+{}$GCH that the group $\Q$ has no $\mathcal W$\+precover.
 (See also the discussion in~\cite[Lemma~2.1 and Example~2.2]{ST}.)

\subsection{Homological formulas}
 Let $R\rarrow A$ be a homomorphism of associative rings.
 Then every left or right $A$\+module has an underlying $R$\+module
structure.
 In particular, $A$ itself acquires the structure of
an $R$\+$R$\+bimodule.

\begin{lem} \label{homological-formulas-lemma}
\textup{(a)} Let $L$ be a left $R$\+module and $D$ be a left
$A$\+module, and let $n\ge0$ be an integer.
 Assume that\/ $\Tor^R_i(A,L)=0$ for all\/ $0<i\le n$.
 Then there is a natural isomorphism of abelian groups\/
$\Ext_A^i(A\ot_RL,\>D)\simeq\Ext_R^i(L,D)$ for every\/ $0\le i\le n$.
\par
\textup{(b)} Let $B$ be a left $A$\+module and $M$ be a left
$R$\+module, and let $n\ge0$ be an integer.
 Assume that\/ $\Ext_R^i(A,M)=0$ for all\/ $0<i\le n$.
 Then there is a natural isomorphism of abelian groups\/
$\Ext_A^i(B,\Hom_R(A,M))\simeq\Ext_R^i(B,M)$ for every\/ $0\le i\le n$.
\end{lem}

\begin{proof}
 We will prove part~(b); the proof of part~(a) is similar.
 Notice that, for any injective left $R$\+module $I$, the left
$A$\+module $\Hom_R(A,I)$ is injective (because the functor
$\Hom_R(A,{-})\:R\Modl\rarrow A\Modl$ is right adjoint to
the forgetful functor $A\Modl\rarrow R\Modl$, which is exact).
 Let $I^\bu$ be an injective coresolution of the left $R$\+module~$M$.
 Then the sequence of left $A$\+modules $0\rarrow\Hom_R(A,M)\rarrow
\Hom_R(A,I^0)\rarrow\dotsb\rarrow\Hom_R(A,I^{n+1})$ is exact, since
$\Ext_R^i(A,M)=0$ for all $0<i\le n$.
 Extending this sequence to an injective coresolution
$\Hom_R(A,I^0)\rarrow\dotsb\rarrow\Hom_R(A,I^{n+1})\rarrow J^{n+2}
\rarrow J^{n+3}\rarrow\dotsb$ of the left $A$\+module
$\Hom_R(A,M)$ and computing the groups $\Ext_A^i(B,\Hom_R(A,M))$
in terms of this coresolution, we obtain the desired natural
isomorphisms.
\end{proof}

\subsection{Resolution dimension}
 Let $A$ be a ring and $\F\subset A\Modl$ be a class of left
$A$\+modules.
 We will say that the class $\F$ is \emph{resolving} if the following
conditions hold:
\begin{enumerate}
\renewcommand{\theenumi}{\roman{enumi}}
\item $\F$ is closed under extensions in $A\Modl$;
\item $\F$ is closed under the kernels of surjective morphisms in
$A\Modl$;
\item every left $A$\+module is a quotient module of a module from~$\F$.
\end{enumerate}
 Notice that, if $\F$ is closed under direct summands, then
condition~(iii) can be equivalently rephrased by saying that all
the projective left $A$\+modules belong to~$\F$.

 Let $k\ge0$ be an integer.
 We say that a left $A$\+module $M$ has \emph{$\F$\+resolution
dimension~$\le k$} if there exists an exact sequence of
left $A$\+modules $0\rarrow F_k\rarrow F_{k-1}\rarrow\dotsb\rarrow F_1
\rarrow F_0\rarrow M\rarrow0$ with $F_i\in\F$ for all $0\le i\le k$.

 Dually, a class of modules $\C\subset A\Modl$ is said to be
\emph{coresolving} if the following conditions hold:
\begin{enumerate}
\renewcommand{\theenumi}{\roman{enumi}*}
\item $\C$ is closed under extensions in $A\Modl$;
\item $\C$ is closed under the cokernels of injective morphisms in
$A\Modl$;
\item every left $A$\+module is a submodule of a module from~$\C$.
\end{enumerate}
 If $\C$ is closed under direct summands, then condition~(iii*) is
equivalent to the condition that all the injective left $A$\+modules
belong to~$\C$.

 We say that a left $A$\+module $N$ has \emph{$\C$\+coresolution
dimension~$\le k$} if there exists an exact sequence of
left $A$\+modules $0\rarrow N\rarrow C^0\rarrow C^1\rarrow\dotsb
\rarrow C^{k-1}\rarrow C^k\rarrow0$ with $C^i\in\C$ for all
$0\le i\le k$.

\begin{lem} \label{co-resolution-dimension-lemma}
\textup{(a)} Let $\F\subset A\Modl$ be a resolving class, and
let $M$ be a left $A$\+module of $\F$\+resolution dimension~$\le k$.
 Let\/ $0\rarrow G_k\rarrow G_{k-1}\rarrow\dotsb\rarrow G_1\rarrow
G_0\rarrow M\rarrow0$ be an exact sequence of left $A$\+modules.
 Assume that $G_i\in\F$ for all\/ $0\le i<k$.
 Then $G_k\in\F$. \par
\textup{(b)} Let $\C\subset A\Modl$ be a coresolving class, and
let $N$ be a left $A$\+module of $\C$\+coresolution dimension~$\le k$.
 Let\/ $0\rarrow N\rarrow D^0\rarrow D^1\rarrow\dotsb\rarrow D^{k-1}
\rarrow D^k\rarrow 0$ be an exact sequence of left $A$\+modules.
 Assume that $D^i\in\C$ for all\/ $0\le i<k$.
 Then $D^k\in\C$.
\end{lem}

\begin{proof}
 This is~\cite[Proposition~2.3(1)]{Sto}
or~\cite[Corollary~A.5.2]{Pcosh}.
 (The resolving and coresolving classes are assumed to be closed under
direct summands in~\cite{Sto}, but this assumption can be dropped.)
\end{proof}

\begin{lem} \label{co-resolution-dimension-defined-classes}
\textup{(a)} For any resolving class\/ $\F\subset A\Modl$ and any
integer $l\ge0$, the class $\F(l)$ of all left $A$\+modules of
$\F$\+resolution dimension~$\le l$ is resolving as well. \par
\textup{(b)} For any coresolving class\/ $\C\subset A\Modl$ and
any integer $l\ge0$, the class $\C(l)$ of all left $A$\+modules
of $\C$\+coresolution dimension~$\le l$ is coresolving as well.
\end{lem}

\begin{proof}
 This is~\cite[Proposition~2.3(2)]{Sto}
or~\cite[Lemma~A.5.4]{Pcosh}.
\end{proof}

\begin{lem} \label{hereditary-cotorsion-pair-lemma}
 Let $(\F,\C)$ be a cotorsion pair in $A\Modl$.
 Then the following conditions are equivalent:
\begin{enumerate}
\item the class $\F$ is resolving (i.~e., $\F$ is closed under
the kernels of surjective morphisms in $A\Modl$);
\item the class $\C$ is coresolving (i.~e., $\C$ is closed under
the cokernels of injective morphisms in $A\Modl$);
\item $\Ext_A^2(F,C)=0$ for all $F\in\F$ and $C\in\C$;
\item $\Ext_A^n(F,C)=0$ for all $F\in\F$, \ $C\in\C$,
and $n\ge1$ (i.~e., $\C=\F^{\perp_{>0}}$ and $\F={}^{\perp_{>0}}\C$).
\end{enumerate}
\end{lem}

\begin{proof}
 This lemma is well-known; see~\cite[Theorem~1.2.10]{GR}
or~\cite[Lemma~5.24]{GT}.
 The argument is straightforward, based on the long exact sequences
of Ext groups for a short exact sequence of modules.
 One proves the equivalences
(1)\,$\Longleftrightarrow$\,(3)\,$\Longleftrightarrow$\,(2)
and then deduces~(4) from either (1) or~(2).
\end{proof}

 A cotorsion pair $(\F,\C)$ in $A\Modl$ is said to be
\emph{hereditary} if it satisfies the equivalent conditions of
Lemma~\ref{hereditary-cotorsion-pair-lemma}.

\Section{Cofiltrations by Coinduced Modules}
\label{cofiltrations-secn}

\subsection{Posing the problem}
 Let $R\rarrow A$ be a homomorphism of associative rings,
and let $\F$ be a class of left $R$\+modules.
 Mostly we will assume $\F$ to be the left part of a cotorsion
pair $(\F,\C)$ in $R\Modl$.

 Denote by $\F_A$ the class of all left $A$\+modules whose
underlying $R$\+modules belong to~$\F$.
 Does there exist a cotorsion pair $(\F_A,\C_A)$ in $A\Modl$\,?

 Obviously, if the answer to this question is positive, then
the class $\C_A$ can be recovered as $\C_A=\F_A^{\perp_1}$.
 But can one describe the class $\C_A$ more explicitly?

 We start with the following easy lemma, which provides
a necessary condition.

\begin{lem} \label{necessary-A-in-F}
 Assume that $\F_A$ is the left part of a cotorsion pair
$(\F_A,\C_A)$ in $A\Modl$.
 Then the left $R$\+module $A$ belongs to~$\F$.
\end{lem}

\begin{proof}
 For any cotorsion pair $(\F_A,\C_A)$ in $A\Modl$, all projective
left $A$\+modules belong to~$\F_A$.
 So, in the situation at hand, the underlying left $R$\+modules
of all projective left $A$\+modules must belong to~$\F$.
\end{proof}

 The next lemma shows that this condition is also sufficient to
get a cotorsion pair $(\F_A,\C_A)$.
 Given a class of left $R$\+modules $\T$, we denote by $\Hom_R(A,\T)$
the class of all left $A$\+modules of the form $\Hom_R(A,T)$ with
$T\in\T$.

\begin{lem} \label{orthogonal-to-coinduced-lemma}
 Let $(\F,\C)$ be a cotorsion pair in $R\Modl$ cogenerated by
a class of left $R$\+modules~$\T$.
 Assume that the left $R$\+module $A$ belongs to~$\F$.
 Then we have \par
\textup{(a)} $\F_A={}^{\perp_1}\!\Hom_R(A,\C)=
{}^{\perp_1}\!\Hom_R(A,\T)$; \par
\textup{(b)} $(\F_A,\F_A^{\perp_1})$ is a cotorsion pair in
$A\Modl$; \par
\textup{(c)} $\Cof(\Hom_R(A,\T))^\oplus\subset
\Cof(\Hom_R(A,\C))^{\oplus}\subset\F_A^{\perp_1}$.
\end{lem}

\begin{proof}
 Part~(a): by assumptions, we have $\F={}^{\perp_1}\T$ and
$\Ext^1_R(A,T)=0$ for all $T\in\T$.
 By Lemma~\ref{homological-formulas-lemma}(b) (for $n=1$), it follows
that a left $A$\+module $F$ belongs to ${}^{\perp_1}\!\Hom_R(A,\T)$
if and only if the underlying left $R$\+module of $F$ belongs to
${}^{\perp_1}\T$.
 In particular, this is applicable to $\T=\C$.

 Part~(b): in view of part~(a), $(\F_A,\F_A^{\perp_1})$ is
the cotorsion pair in $A\Modl$ cogenerated by the class
$\Hom_R(A,\T)$ or $\Hom_R(A,\C)$.

 Part~(c) follows from part~(a) and Lemma~\ref{dual-eklof-lemma}.
\end{proof}

 So we have answered our first question, but we want to know more.
 Can one guarantee that the cotorsion pair $(\F_A,\C_A)$ is complete?

\begin{prop} \label{deconstructible-prop}
 Let $(\F,\C)$ be a (complete) cotorsion pair in $R\Modl$ generated
by a set of left $R$\+modules~$\S$, and let $\F_A$ be the class of
all left $A$\+modules whose underlying left $R$\+modules belong
to~$\F$.
 Assume that the left $R$\+module $A$ belongs to~$\F$.
 Then there exists a complete cotorsion pair $(\F_A,\C_A)$ in
$A\Modl$ generated by a certain set of left $A$\+modules~$\S_A$.
\end{prop}

\begin{proof}
 A class of left $R$\+modules $\F$ is said to be \emph{deconstructible}
if there exists a set of left $R$\+modules $\S$ such that
$\F=\Fil(\S)$.
 Any class of modules of the form $\F=\Fil(\S)^\oplus$ is
deconstructible, that is, for any set $\S\subset R\Modl$ there exists
a set $\S'\subset R\Modl$ such that $\Fil(\S)^\oplus=\Fil(\S')\subset
R\Modl$ \cite[Lemma~7.12]{GT}.
 Furthermore, it follows from the Hill lemma~\cite[Theorem~7.10]{GT}
that the class $\F_A\subset A\Modl$ is deconstructible for every
deconstructible class $\F\subset R\Modl$.
 So there exists a set of left $A$\+modules $\S_A$ such that
$\F_A=\Fil(\S_A)$.
 In fact, if~$\kappa$ is an uncountably infinite regular cardinal
such that the cardinalites of $R$ and $A$ are smaller than~$\kappa$ and
all the modules in $\S$ are $<\kappa$\+presented, then one can use
the set of (representatives of the isomorphism classes of) all
the $<\kappa$\+presented modules in $\F_A$ in the role of~$\S_A$.
 Finally, if a deconstructible class $\F_A=\Fil(\S_A)\subset A\Modl$ is
closed under direct summands and $A\in\F_A$, then
$(\F_A,\F_A^{\perp_1})$ is a complete cotorsion pair in $A\Modl$
generated by the set of left $A$\+modules $\S_A$ by
Lemma~\ref{eklof-lemma} and Theorem~\ref{eklof-trlifaj-theorem}.
\end{proof}

 After these observations, which follow from the general theory of
cotorsion pairs in module categories, essentially the only remaining
question is the one about an explicit description of the class
$\C_A=\F_A^{\perp_1}$.
 In the rest of Section~\ref{cofiltrations-secn}, our aim is to show
that, under certain assumptions, the inclusions in
Lemma~\ref{orthogonal-to-coinduced-lemma}(c) become equalities, that is,
most importantly, $\C_A=\Cof(\Hom_R(A,\C))^\oplus$.

 In fact, depending on specific assumptions, we will be able to
prove that $\C_A=\Cof_\beta(\Hom_R(A,\C))^\oplus$ for certain rather
small ordinals~$\beta$.
 Our assumptions are going to be rather restrictive; but we will
\emph{not} assume the cotorsion pair $(\F,\C)$ to be generated by a set
(as in Proposition~\ref{deconstructible-prop}).

 Concerning the second inclusion in
Lemma~\ref{orthogonal-to-coinduced-lemma}(c), all we can say is
the following.

\begin{lem}
 Let $\T$ be a class of left $R$\+modules such that
$A\in{}^{\perp_1}\T$, and let $\alpha$~be an ordinal.
 Then \par
\textup{(a)} $\Hom_R(A,\Cof_\alpha(\T))\subset
\Cof_\alpha(\Hom_R(A,\T))$; \par
\textup{(b)} $\Hom_R(A,\Cof_\alpha(\T)^\oplus)\subset
\Cof_\alpha(\Hom_R(A,\T))^\oplus$. \par
\noindent In particular, if $\C=\Cof(\T)^\oplus$, then\/
$\Cof(\Hom_R(A,\C))^\oplus=\Cof(\Hom_R(A,\T))^\oplus$.
\end{lem}

\begin{proof}
 Part~(a) holds, because the functor $\Hom_R(A,{-})\:R\Modl\rarrow
A\Modl$ preserves inverse limits, as well as short exact sequences of
modules belonging to $\{A\}^{\perp_1}\subset R\Modl$.
 Part~(b) follows immediately from~(a).

 The last assertion follows from~(b) in view of 
Lemma~\ref{iterated-cof-lemma}.
 Indeed, we have $\Cof(\Hom_R(A,\C))^\oplus=
\Cof(\Hom_R(A,\Cof(\T)^\oplus))^\oplus\subset
\Cof(\Cof(\Hom_R(A,\T))^\oplus)^\oplus=\Cof(\Hom_R(A,\T))^\oplus$.
\end{proof}

 For a class of examples of complete cotorsion pairs like in
Proposition~\ref{deconstructible-prop} arising in connection with
$n$\+cotilting modules, see
Lemma~\ref{cotilting-class-change-of-ring}
and Proposition~\ref{bazzoni-prop} below.

\subsection{Finite filtrations by coinduced modules}
\label{finite-by-coinduced-subsecn}
 Let $R\rarrow A$ be a ring homomorphism.
 Suppose that we are given an $\Ext^1$\+orthogonal pair of
classes of left $R$\+modules $\F$ and $\C\subset R\Modl$, and
denote by $\F_A\subset A\Modl$ the class of all left $A$\+modules
whose underlying left $R$\+modules belong to~$\F$.
 
 For any left $R$\+module $M$, one can consider the left $A$\+module
$\Hom_R(A,M)$.
 Sometimes we also consider the underlying left $R$\+module of
the left $A$\+module $\Hom_R(A,M)$.
 That is what we do when formulating the following condition, which
will be a key technical assumption in much of the rest of
Section~\ref{cofiltrations-secn}:
\begin{itemize}
\item[($\dagger\dagger$)] for any left $R$\+module $F\in\F$,
the left $R$\+module $\Hom_R(A,F)$ also belongs to~$\F$.
\end{itemize}

 The specific assumption on which the results of this
Section~\ref{finite-by-coinduced-subsecn} are based is that
all left $R$\+modules have finite $\F$\+resolution dimension.

\begin{lem} \label{coinduction-preserves-resolution-dimension}
 Assume that the\/ $\Ext^1$\+orthogonal pair of classes of
left $R$\+modules $(\F,\C)$ admits special precover
sequences~\eqref{sp-precover-seq}.
 Assume further that the left $R$\+module $A$ belongs to~$\F$,
the condition~\textup{($\dagger\dagger$)} holds, and the class $\F$
is resolving in $R\Modl$.
 Let $M$ be a left $R$\+module of $\F$\+resolution dimension~$\le l$.
 Then the $\F$\+resolution dimension of the left $R$\+module\/
$\Hom_R(A,M)$ also does not exceed~$l$.
\end{lem}

\begin{proof}
 Let $0\rarrow C_1\rarrow F_0\rarrow M\rarrow0$ be a special
precover sequence~\eqref{sp-precover-seq} for the left
$R$\+module~$M$; so $C_1\in\C$ and $F_0\in\F$.
 Consider a special precover sequence $0\rarrow C_2\rarrow F_1
\rarrow C_1\rarrow0$ for the left $R$\+module $C_1$, etc.
 Proceeding in this way, we construct an exact sequence of
left $R$\+modules $0\rarrow C_l\rarrow F_{l-1}\rarrow F_{l-2}
\rarrow\dotsb\rarrow F_1\rarrow F_0\rarrow M\rarrow0$, in
which $F_i\in\F$ for all $0\le i\le l-1$, \ $C_l\in\C$, and
the image $C_i$ of the morphism $F_i\rarrow F_{i-1}$ belongs
to $\C$ for all $1\le i\le l-1$.
 Since the $\F$\+resolution dimension of $M$ does not exceed~$l$
by assumption, by Lemma~\ref{co-resolution-dimension-lemma}(a) it
follows that $C_l\in\F$.
 Since $A\in\F\subset{}^{\perp_1}\C$, our exact sequence remains
exact after applying the functor $\Hom_R(A,{-})$.
 The resulting exact sequence is the desired resolution of length~$l$
of the left $R$\+module $\Hom_R(A,M)$ by modules from~$\F$.
\end{proof}

\begin{prop} \label{finite-resol-dim-approximation-sequences}
 Assume that the\/ $\Ext^1$\+orthogonal pair of classes of
left $R$\+modules $(\F,\C)$ admits approximation
sequences~\textup{(\ref{sp-precover-seq}\+-\ref{sp-preenvelope-seq})}.
 Assume that the left $R$\+module $A$ belongs to~$\F$, and
that the condition~\textup{($\dagger\dagger$)} holds.
 Assume further that the class $\F$ is resolving in $R\Modl$ and
the $\F$\+resolution dimension of any left $R$\+module does not
exceed a finite integer~$k\ge0$.
 Then the\/ $\Ext^1$\+orthogonal pair of classes of left $A$\+modules
$\F_A$ and\/ $\Cof_{k+1}(\Hom_R(A,\C))$ admits approximation sequences
as well.
 Here the integer~$k+1$ is considered as a finite ordinal.
\end{prop}

\begin{proof}
 The pair of classes $\F_A$ and $\Cof(\Hom_R(A,\C))\subset A\Modl$ is
$\Ext^1$\+orthogonal by Lemma~\ref{orthogonal-to-coinduced-lemma}(c).
 Let us show by explicit construction that the pair of classes $\F_A$
and $\Cof_k(\Hom_R(A,\C))$ admits special precover sequences.
 The construction below goes back to~\cite[Lemma~1.1.3]{Psemi}.

 Let $M$ be a left $A$\+module.
 Then there is a natural (adjunction) morphism of left $A$\+modules
$\nu_M\:M\rarrow\Hom_R(A,M)$ defined by the formula
$\nu_M(m)(a)=am\in M$ for every $m\in M$ and $a\in A$.
 The map~$\nu_M$ is always injective.
 Moreover, viewed as a morphism of left $R$\+modules, $\nu_M$~is
a split monomorphism.
 Indeed, the evaluation-at-unit map $\phi_M\:\Hom_R(A,M)\rarrow M$
taking a function $f\in\Hom_R(A,M)$ to its value $\phi_M(f)=f(1)\in M$
is a left $R$\+module morphism for which the composition
$\phi_M\circ\nu_M$ is the identity map, $\phi_M\circ\nu_M=\id_M$.

 Consider the underlying left $R$\+module of $M$, and choose
a special precover sequence $0\rarrow C'(M)\rarrow F(M)\rarrow M
\rarrow 0$ in $R\Modl$ with $C'(M)\in\C$ and $F(M)\in\F$.
 Then we have $\Ext^1_R(A,C'(M))=0$, so the morphism of left
$A$\+modules $\Hom_R(A,F(M))\rarrow\Hom_R(A,M)$ coinduced from
the surjective left $R$\+module map $F(M)\rarrow M$ is surjective.
 Denote by $Q(M)$ the pullback (or in other words, the fibered
product) of the pair of left $A$\+module morphisms $M\rarrow
\Hom_R(A,M)$ and $\Hom_R(A,F(M))\rarrow\Hom_R(A,M)$.

 We have a commutative diagram of left $A$\+module morphisms, in
which the four short sequences are exact:
\begin{equation} \label{precover-construction-diagram}
 \xymatrix{
  & 0 & 0 \\
  0 \ar[r] & M \ar[u]\ar[r] & \Hom_R(A,M) \ar[u]\ar[r]
  & \Hom_R(A,M)/M \ar[r] & 0 \\ \\
  0 \ar[r] & Q(M) \ar[uu]\ar[r] & \Hom_R(A,F(M)) \ar[uu]\ar[r]
  & \Hom_R(A,M)/M \ar@{=}[uu] \ar[r] & 0 \\ \\
  & \Hom_R(A,C'(M)) \ar[uu]\ar@{=}[r] & \Hom_R(A,C'(M)) \ar[uu] \\
  & 0 \ar[u] & 0 \ar[u]
 }
\end{equation}

 Introduce the notation $\rd_\F N$ for the $\F$\+resolution dimension
of a left $R$\+module~$N$.
 We will apply the same notation to $A$\+modules, presuming that
the $\F$\+resolution dimension of the underlying $R$\+module is taken.

 Next we observe that, whenever $0<\rd_\F M<\infty$,
the $\F$\+resolution dimension of the underlying left $R$\+module
of the left $A$\+module $Q(M)$ is strictly smaller than
the $\F$\+resolution dimension of the underlying $R$\+module of
the $A$\+module~$M$, i.~e., $\rd_\F Q(M)<\rd_\F(M)$.
 Indeed, the short exact sequence of left $A$\+modules $0\rarrow M
\rarrow \Hom_R(A,M)\rarrow \Hom_R(A,M)/M\rarrow0$ splits over $R$,
or in other words, the underlying left $R$\+module of $\Hom_R(A,M)/M$
can be presented as the kernel of the surjective left $R$\+module
morphism $\phi_M\:\allowbreak\Hom_R(A,M)\rarrow M$.
 By Lemmas~\ref{coinduction-preserves-resolution-dimension}
and~\ref{co-resolution-dimension-defined-classes}(a), we have
$\rd_\F\Hom_R(A,M)/M\le\rd_\F M$.
 Since $\Hom_R(A,F(M))\in\F$, it follows from the short exact
sequence $0\rarrow Q(M)\rarrow\Hom_R(A,F(M))\rarrow \Hom_R(A,M)/M
\rarrow0$ that $\rd_\F Q(M)<\rd_\F(M$).

 It remains to iterate our construction, producing a sequence of
surjective morphisms of left $A$\+modules
$$
 M\llarrow Q(M)\llarrow Q(Q(M))\llarrow Q^3(M)\llarrow\dotsb
 \llarrow Q^k(M).
$$
 Since $\rd_\F M\le k$ by assumption, it follows from the above
argument that $\rd_\F Q^k(M)\le0$, that is $Q^k(M)\in\F_A$.

 The kernel of the surjective morphism $Q^k(M)\rarrow M$ is
cofiltered by the kernels of the surjective $A$\+module morphisms
$Q(M)\rarrow M$, \ $Q^2(M)\rarrow Q(M)$, \dots,
$Q^k(M)\rarrow Q^{k-1}(M)$.
 These are the left $A$\+modules $\Hom_R(A,C'(M))$, \ 
$\Hom_R(A,C'(Q(M)))$, \ $\Hom_R(A,C'(Q^2(M)))$, \dots,
$\Hom_R(A,C'(Q^{k-1}(M)))$.
 We have constructed the desired special precover sequence for
the pair of classes $\F_A$ and $\Cof_k(\Hom_R(A,\C))$.

 Finally, any left $R$\+module $N$ is a submodule of an $R$\+module
$C(N)\in\C$, since a special preenvelope sequence with respect to
$(\F,\C)$ exists for $N$ by assumption.
 If $N$ is a left $A$\+module, then the map~$\nu_N$ provides
an embedding of $N$ into the left $A$\+module $\Hom_R(A,N)$,
which is a submodule of the left $A$\+module $\Hom_R(A,C(N))$.
 Thus $N$ is an $A$\+submodule of $\Hom_R(A,C(N))$.
 Following the proof of (the ``only if'' implication in)
Lemma~\ref{salce-lemma}, we conclude that the pair of classes $\F_A$
and $\Cof_{k+1}(\Hom_R(A,\C))$ admits special preenvelope sequences.
\end{proof}

\begin{thm} \label{finite-resol-dim-theorem}
 Let $(\F,\C)$ be a hereditary complete cotorsion pair in $R\Modl$.
 Assume that the left $R$\+module $A$ belongs to~$\F$, and
that the condition~\textup{($\dagger\dagger$)} holds.
 Assume further that the $\F$\+resolution dimension of any left
$R$\+module does not exceed a finite integer~$k\ge0$.
 Then the pair of classes $\F_A$ and
$\C_A=\Cof_{k+1}(\Hom_R(A,\C))^\oplus$ is a hereditary complete
cotorsion pair in $A\Modl$.
\end{thm}

\begin{proof}
 The class $\F_A$ is closed under direct summands and the kernels
of surjective morphisms, since the class $\F$~is.
 Thus the assertion of the theorem follows from
Proposition~\ref{finite-resol-dim-approximation-sequences}
in view of Lemma~\ref{direct-summand-lemma}.
\end{proof}

\begin{cor} \label{finite-resol-dim-cor}
 For any associative ring homomorphism $R\rarrow A$ and any
hereditary complete cotorsion pair $(\F,\C)$ in $R\Modl$ satisfying
the assumptions of Theorem~\ref{finite-resol-dim-theorem}, one has
$\F_A^{\perp_1}=\Cof_{k+1}(\Hom_R(A,\C))^\oplus$.
 In particular, it follows that\/ $\Cof(\Hom_R(A,\C))^\oplus=
\Cof_{k+1}(\Hom_R(A,\C))^\oplus$.  \hbadness=1200
\end{cor}

\begin{proof}
 The first assertion is a part of
Theorem~\ref{finite-resol-dim-theorem}.
 The second assertion follows from the first one together with
Lemma~\ref{orthogonal-to-coinduced-lemma}(c).
\end{proof}

\begin{rem} \label{corings-and-comodules-remark}
 The condition~($\dagger\dagger$) appears to be rather restrictive.
 In fact, the construction of
Proposition~\ref{finite-resol-dim-approximation-sequences} originates
from the theory of comodules over corings, as
in~\cite[Lemma~1.1.3]{Psemi}, where the natural analogue
of this condition feels much less restrictive, particularly when
$\F$ is simply the class of all projective left $R$\+modules.
 So one can say that the ring $A$ in this
Section~\ref{finite-by-coinduced-subsecn} really ``wants'' to be
a coring $C$ over the ring~$R$, and the left $A$\+modules ``want''
to be left $C$\+comodules.
 Then the coinduction functor, which was $\Hom_R(A,{-})$
in the condition~($\dagger\dagger$), takes the form of the tensor
product functor $C\ot_R{-}$.
 This one is much more likely to take projective left $R$\+modules to
projective left $R$\+modules (it suffices that $C$ be
a projective left $R$\+module).
 To make a ring $A$ behave rather like a coring, one can assume it to
be ``small'' relative to $R$ in some sense.
 The following example is inspired by the analogy with corings and
comodules.
\end{rem}

\begin{ex} \label{corings-and-comodules-example}
 Let $\F=R\Modl_\proj$ be the class of all projective left $R$\+modules.
 Then $\C=R\Modl$ is the class of all left $R$\+modules, and
$\F_A=A\Modl_{R\dproj}$ is the class of all left $A$\+modules
whose underlying $R$\+modules are projective.
 In the terminology of~\cite[Sections~4.1 and~4.3]{BP}
and~\cite[Section~5]{Pfp}, the left $A$\+modules from the related class
$\C_A=\F_A^{\perp_1}$ would be called \emph{weakly injective relative
to~$R$} or \emph{weakly $A/R$\+injective}.

 For $\F=R\Modl_\proj$, the necessary condition of
Lemma~\ref{necessary-A-in-F} means that $A$ must be a projective left
$R$\+module.
 Assume that $A$ is a finitely generated projective left $R$\+module;
then the functor $\Hom_R(A,{-})$ preserves infinite direct sums.
 Assume further that the left $R$\+module $\Hom_R(A,R)$ is projective.
 Then it follows that the functor $\Hom_R(A,{-})$ preserves the class
$\F$ of all projective left $R$\+modules.
 Thus the condition~($\dagger\dagger$) is satisfied.
 
 The results of Section~\ref{finite-by-coinduced-subsecn} tell us
that, whenever the left homological dimension of the ring $R$ is
a finite number~$k$ and the assumptions in the previous paragraph hold,
the $\Ext^1$\+orthogonal pair of classes of left $A$\+modules
$A\Modl_{R\dproj}$ and $\Cof_{k+1}(\Hom_R(A,R\Modl))$ admits
approximation sequences.
 Consequently, the pair of classes $\F_A=A\Modl_{R\dproj}$ and
$\C_A=\Cof_{k+1}(\Hom_R(A,R\Modl))^\oplus$ is a hereditary complete
cotorsion pair in $A\Modl$.
 In particular, we have
$$
 (A\Modl_{R\dproj})^{\perp_1}=\Cof_{k+1}(\Hom_R(A,R\Modl))^\oplus,
$$
and therefore
$\Cof(\Hom_R(A,R\Modl))^\oplus=\Cof_{k+1}(\Hom_R(A,R\Modl))^\oplus$.
 So the weakly $A/R$\+injective left $A$\+modules are precisely
the direct summands of the $A$\+modules admitting a finite
$(k+1)$\+step filtration by $A$\+modules coinduced from left
$R$\+modules.

 The reader can find a discussion of the related results for
corings and comodules (of which this example is a particular case)
in~\cite[Lemma~3.11(a)]{Prev}.
\end{ex}

 For a class of examples to Theorem~\ref{finite-resol-dim-theorem}
arising in connection with $n$\+cotilting modules, see
Example~\ref{cotilting-examples}(1) below.
 For a class of examples to the same theorem arising from curved
DG\+rings, see Proposition~\ref{finite-homol-dim-contraacyclic}.

 One problem with the condition~($\dagger\dagger$) is that it mentions
the underived $\Hom_R(A,F)$.
 The groups $\Ext_R^i(A,F)$ with $i>0$ are lurking around, but they
are ignored in the formulation of the condition.
 Yet there is no reason to expect these Ext groups to vanish for all
modules $F\in\F$.

 Thus it may be useful to generalize~($\dagger\dagger)$ by restricting
it to a subclass of the class $\F$ consisting of modules for which
the functor $\Hom_R(A,{-})$ is better behaved.
 One can do so by considering the following condition:
\begin{itemize}
\item[($\widetilde{\dagger\dagger}$)] 
there exists a coresolving class $\D\subset R\Modl$ such that
$\C\subset\D$, the underlying left $R$\+modules of all
the left $A$\+modules from $\C_A=\F_A^{\perp_1}$ belong
to~$\D$, and the left $R$\+module $\Hom_R(A,F)$ belongs to $\F$
for every left $R$\+module $F\in\F\cap\D$.
\end{itemize}

 Taking $\D=R\Modl$, one recovers~$(\dagger\dagger)$ as a particular
case of~($\widetilde{\dagger\dagger}$).

\begin{thm} \label{finite-resol-dim-tilde}
 Let $(\F,\C)$ be a hereditary complete cotorsion pair in $R\Modl$.
 Assume that the left $R$\+module $A$ belongs to $\F$, and that
the condition~\textup{($\widetilde{\dagger\dagger}$)} holds.
 Assume further that the $\F$\+resolution dimension of any
left $R$\+module does not exceed a finite integer $k\ge0$.
 Then the class $\C_A=\F_A^{\perp_1}\subset A\Modl$ can be described
as $\C_A=\Cof_{k+1}(\Hom_R(A,\C))^\oplus$.
 In particular, we have\/ $\Cof(\Hom_R(A,\C))^\oplus=
\Cof_{k+1}(\Hom_R(A,\C))^\oplus$.
\end{thm}

\begin{proof}
 We are following the proof of Corollary~\ref{finite-resol-dim-cor}
step by step and observing that the assumptions of the present
theorem are sufficient for the validity of the argument.
 Essentially, the point is that the key constructions are performed
within the class $\D\subset R\Modl$ and the class of all left
$A$\+modules whose underlying left $R$\+modules belong to~$\D$.

 The inclusion $\Cof(\Hom_R(A,\C))^\oplus\subset\C_A$ holds by
Lemma~\ref{orthogonal-to-coinduced-lemma}(c).
 Given a left $A$\+module $N\in\C_A$, we will show that
$N\in\Cof_{k+1}(\Hom_R(A,\C))^\oplus$.

 Arguing as in the last paragraph of the proof of
Proposition~\ref{finite-resol-dim-approximation-sequences},
the left $R$\+module $N$ is a submodule of an $R$\+module $C(N)\in\C$,
and therefore the left $A$\+module $N$ is a $A$\+submodule of
the left $A$\+module $\Hom_R(A,C(N))$.
 Denote the quotient $A$\+module by $M=\Hom_R(A,C(N))/N$.
 By~($\widetilde{\dagger\dagger}$), we have ${}_RN\in\D$ and
$\Hom_R(A,C(N))\in\D$, hence the underlying left $R$\+module of $M$
also belongs to~$\D$.

 Now we construct the diagram~\eqref{precover-construction-diagram}
for the left $A$\+module~$M$.
 In the special precover sequence $0\rarrow C'(M)\rarrow F(M)\rarrow
M\rarrow0$, we have $C'(M)\in\C\subset\D$ and ${}_RM\in\D$, hence
$F(M)\in\D$.
 According to~($\widetilde{\dagger\dagger}$), it follows that
$\Hom_R(A,F(M))\in\F$.
 Also by~($\widetilde{\dagger\dagger}$), we have
$\Hom_R(A,C'(M))\in\Hom_R(A,\C)\subset\D$, so it follows from the short
exact sequence $0\rarrow\Hom_R(A,C'(M))\rarrow Q(M)\rarrow M\rarrow0$
that the underlying left $R$\+module of the left $A$\+module $Q(M)$
belongs to~$\D$.

 Iterating the construction and following the proof of
Proposition~\ref{finite-resol-dim-approximation-sequences}, we obtain
a surjective morphism of left $A$\+modules $Q^k(M)\rarrow M$ with
$Q^k(M)\in\F_A$ and the kernel belonging to $\Cof_k(\Hom_R(A,\C))$.
 Following the proof of (the ``only if'' implication in)
Lemma~\ref{salce-lemma}, we produce an injective $A$\+module morphism
from $N$ into an $A$\+module belonging to $\Cof_{k+1}(\Hom_R(A,\C))$
with the cokernel isomorphic to~$Q^k(M)$.
 As $\Ext_A^1(Q^k(M),N)=0$ by assumption, we can conclude that
$N\in\Cof_{k+1}(\Hom_R(A,\C))^\oplus$.
\end{proof}

 For a class of examples to Theorem~\ref{finite-resol-dim-tilde}
arising in connection with $n$\+cotilting modules, see
Example~\ref{cotilting-examples}(2).

\subsection{Cotilting cotorsion pairs and dual Bongartz--Ringel lemma}
\label{cotilting-subsecn}
 In this section we digress to discuss an important class of examples
in which a suitable version of the Bongartz--Ringel
lemma~\cite[Lemma~2.1]{Bon}, \cite[Lemma~4$'$]{Rin},
\cite[Proposition~6.44]{GT} leads to a better result than the techniques
of Section~\ref{finite-by-coinduced-subsecn}.

 Let $U$ be a left $R$\+module and $n\ge0$ be an integer.
 The $R$\+module $U$ is said to be
\emph{$n$\+cotilting}~\cite[Section~2]{AC}, \cite[Definition~15.1]{GT}
if the following three conditions hold:
\begin{enumerate}
\renewcommand{\theenumi}{C\arabic{enumi}}
\item the injective dimension of the left $R$\+module $U$ does
not exceed~$n$;
\item $\Ext^i_R(U^\kappa,U)=0$ for all integers $i>0$ and
all cardinals~$\kappa$;
\item there exists an exact sequence of left $R$\+modules
$0\rarrow U_n\rarrow U_{n-1}\rarrow\dotsb\rarrow U_1\rarrow U_0
\rarrow J\rarrow0$, where $J$ is an injective cogenerator of
$R\Modl$ and $U_i\in\Prod_R(U)$ for all $0\le i\le n$.
\end{enumerate}

 The \emph{$n$\+cotilting class} induced by $U$ in $R\Modl$ is
the class of left $R$\+modules $\F={}^{\perp_{>0}}U=
\{F\in R\Modl\mid\Ext^i_R(F,U)=0 \ \forall i>0\}$.
 The cotorsion pair $(\F,\C)$ with $\C=\F^{\perp_1}\subset R\Modl$ is
hereditary and complete~\cite[Proposition~3.3]{AC}; it is called
the \emph{$n$\+cotilting cotorsion pair} induced by $U$ in $R\Modl$.

\begin{prop} \label{when-n-cotilting}
 Let $R\rarrow A$ be a homomorphism of associative rings and $U$ be
an $n$\+cotilting left $R$\+module.
 Assume that the underlying left $R$\+module of $A$ belongs to $\F$,
that is ${}_RA\in\F$.
 Then \par
\textup{(a)} the left $A$\+module\/ $\Hom_R(A,U)$ satisfies
the conditions~\textup{(C1)} and~\textup{(C3)}; \par
\textup{(b)} the left $A$\+module\/ $\Hom_R(A,U)$
satisfies~\textup{(C2)} if and only if its underlying
left $R$\+module belongs to~$\F$.
\end{prop}

\begin{proof}
 Part~(a): by assumption, we have $\Ext_R^i(A,U)=0$ for all $i>0$.
 Hence applying the functor $\Hom_R(A,{-})$ to an injective
resolution $0\rarrow U\rarrow J^0\rarrow\dotsb\rarrow J^n\rarrow0$
of the left $R$\+module $U$ produces an injective resolution
$0\rarrow\Hom_R(A,U)\rarrow\Hom_R(A,J^0)\rarrow\dotsb\rarrow
\Hom_R(A,J^n)\rarrow0$ of the left $A$\+module $\Hom_R(A,U)$.
 Similarly, applying the functor $\Hom_R(A,{-})$ to an exact
sequence in~(C3) produces an exact sequence of left $A$\+modules
$0\rarrow\Hom_R(A,U_n)\rarrow\dotsb\rarrow\Hom_R(A,U_0)\rarrow
\Hom_R(A,J)\rarrow0$, in which $\Hom_R(A,U_i)\in\Prod_A(\Hom_R(A,U))$
for all $0\le i\le n$ and $\Hom_R(A,J)$ is an injective cogenerator
of $A\Modl$.

 Part~(b): put $U'=\Hom_R(A,U)$.
 By Lemma~\ref{homological-formulas-lemma}(b), we have
$\Ext^i_A(U'{}^\kappa,U')\simeq\Ext^i_R(U'{}^\kappa,U)$ for all
$i\ge0$, since $\Ext_R^i(A,U)=0$ for $i>0$.
 It follows that the left $A$\+module $U'$ is $n$\+cotilting if and only
if the left $R$\+module $U'{}^\kappa$ belongs to $\F\subset R\Modl$
for every cardinal~$\kappa$.
 Since the $n$\+cotilting class $\F$ is closed under infinite
products in $R\Modl$ \cite[Proposition~15.5(a)]{GT}, it suffices
that ${}_RU'\in\F$.
\end{proof}

\begin{lem} \label{cotilting-class-change-of-ring}
 Let $R\rarrow A$ be a homomorphism of associative rings and $U$ be
an $n$\+cotilting left $R$\+module.
 Let $(\F,\C)$ be the $n$\+cotilting cotorsion pair induced by $U$ in
$R\Modl$.
 Assume that the underlying left $R$\+module of $A$ belongs to $\F$,
that is ${}_RA\in\F$.
 Assume further that the left $A$\+module\/ $\Hom_R(A,U)$ is
$n$\+cotilting.
 Then the $n$\+cotilting cotorsion pair induced by\/ $\Hom_R(A,U)$
in $A\Modl$ has the form $(\F_A,\C_A)$ in our notation.
 In other words, the $n$\+cotilting class induced by\/ $\Hom_R(A,U)$
in $A\Modl$ consists precisely of all the left $A$\+modules whose
underlying left $R$\+modules belong to the $n$\+cotilting class $\F$
induced by $U$ in $R\Modl$.
\end{lem}

\begin{proof}
 Indeed, for any left $A$\+module $F$ we have $\Ext^i_A(F,\Hom_R(A,U))
\simeq\Ext^i_R(F,U)$ for all $i\ge0$ by
Lemma~\ref{homological-formulas-lemma}(b), since $\Ext^i_R(A,U)=0$
for all $i>0$.
\end{proof}

\begin{prop} \label{bazzoni-prop}
 Let $R$ be a commutative ring and $A$ be an associative $R$\+algebra.
 Let $U$ be an $n$\+cotilting $R$\+module and $(\F,\C)$ be
the $n$\+cotilting cotorsion pair induced by $U$ in $R\Modl$.
 Assume that the underlying $R$\+module of $A$ belongs to~$\F$.
 Then the left $A$\+module\/ $\Hom_R(A,U)$ is $n$\+cotilting.
\end{prop}

\begin{proof}
 According to Proposition~\ref{when-n-cotilting}, it suffices to show
that the $R$\+module $U'=\Hom_R(A,U)$ belongs to~$\F$.
 The following argument was suggested to the author by S.~Bazzoni.
 By~\cite[Lemma~3.2]{Baz} or~\cite[Proposition~15.5(a)]{GT},
the cotilting class $\F$ can be described as the class of
all $R$\+modules admitting a coresolution by products of copies of~$U$.
 Let $\dotsb\rarrow P_2\rarrow P_1\rarrow P_0\rarrow A\rarrow0$ be
a free resolution of the $R$\+module~$A$.
 Then $0\rarrow\Hom_R(A,U)\rarrow\Hom_R(P_0,U)\rarrow\Hom_R(P_1,U)
\rarrow\Hom_R(P_2,U)\rarrow\dotsb$ is a coresolution of the $R$\+module
$\Hom_R(A,U)$ by products of copies of~$U$.
 Thus ${}_RU'\in\F$, as desired.
\end{proof}

 A discussion of the particular case of the above proposition and
lemma in which the ring $R$ is Noetherian and $A=R_{\mathfrak m}$ is
the localization of $R$ at the maximal ideal $\mathfrak m\subset R$
can be found in~\cite[Lemma~2.1]{TS}.

\begin{exs} \label{cotilting-examples}
 Let $A$ be an associative algebra over a commutative ring $R$, and
let $U$ be an $n$\+cotilting $R$\+module.
 Let $(\F,\C)$ be the $n$\+cotilting cotorsion pair induced by $U$ in
$R\Modl$.
 Assume that the $R$\+module $A$ belongs to~$\F$.
 Then $\Hom_R(A,U)$ is an $n$\+cotilting left $A$\+module
by Proposition~\ref{bazzoni-prop}, and
the induced $n$\+cotilting cotorsion pair in $A\Modl$ has the form
$(\F_A,\C_A)$ in our notation by
Lemma~\ref{cotilting-class-change-of-ring}.

 (1) In the following particular cases
Theorem~\ref{finite-resol-dim-theorem} is applicable.
 Assume that either $A$ is a projective $R$\+module, or $n\le 2$.
 Then the condition~($\dagger\dagger$) holds.
 
 Indeed, when $A$ is a projective $R$\+module, it suffices to observe
that $\F$ is closed under infinite products.
 When $n\le 2$, consider an $R$\+module $F\in\F$.
 Choose a projective presentation $P_1\rarrow P_0\rarrow A\rarrow 0$
for the $R$\+module~$A$.
 Then we have a left exact sequence of $R$\+modules $0\rarrow
\Hom_R(A,F)\rarrow\Hom_R(P_0,F)\rarrow\Hom_R(P_1,F)$ with
$\Hom_R(P_i,F)\in\F$ for $i=0$,~$1$.
 Denoting by $L$ the cokernel of the morphism $\Hom_R(P_0,F)\rarrow
\Hom_R(P_1,F)$, we have $\Ext^i_R(\Hom_R(A,F),U)=
\Ext^{i+2}_R(L,U)=0$ for all $i>0$, as desired.

 Finally, the $\F$\+resolution dimension of any $R$\+module does not
exceed~$n$ (since the injective dimension of the $R$\+module~$U$
is~$\le n$).
 According to Corollary~\ref{finite-resol-dim-cor}, we can conclude
that $\F_A^{\perp_1}=\C_A=\Cof_{n+1}(\Hom_R(A,\C))^\oplus$.

 (2) This is a generalization of~(1) that can be obtained
using Theorem~\ref{finite-resol-dim-tilde}.
 We are assuming that $A$ is an associative $R$\+algebra, $U$ is
an $n$\+cotilting $R$\+module, and ${}_RA\in\F$.
 Assume further that $\Ext^i_R(A,\Hom_R(A,U))=0$ for all $i>0$.
 (In particular, this holds whenever $A$ is a flat $R$\+module, as
the $R$\+module $U$ is pure-injective by~\cite[Theorem~15.7]{GT}.)
 Then we claim that ($\widetilde{\dagger\dagger}$)~is satisfied.

 Let $\D=A^\perp\subset R\Modl$ be the class of all $R$\+modules $D$
such that $\Ext^i_R(A,D)=0$ for all $i>0$.
 Then we have $\C\subset\D$, since ${}_RA\in\F$.
 Furthermore, all the left $A$\+modules in $\C_A$ have finite
resolutions by direct summands of products of copies of $\Hom_R(A,U)$
\cite[Proposition~15.5(b)]{GT}; hence $\Ext^i_R(A,C)=0$
for all $C\in\C_A$ and $i>0$.

 In order to check the condition~($\widetilde{\dagger\dagger}$), it
remains to show that $\Hom_R(A,F)\in\F$ for any $R$\+module
$F\in\F\cap\D$.
 Indeed, let us choose a projective resolution $\dotsb\rarrow P_2
\rarrow P_1\rarrow P_0\rarrow A\rarrow0$ for the $R$\+module~$A$.
 Then we have an exact sequence of $R$\+modules $0\rarrow\Hom_R(A,F)
\rarrow\Hom_R(P_0,F)\rarrow\Hom_R(P_1,F)\rarrow\Hom_R(P_2,F)\rarrow
\dotsb$ with $\Hom_R(P_i,F)\in\F$ for all $i\ge0$.
 Denoting by $L$ the image of the morphism $\Hom_R(P_{n-1},F)
\rarrow\Hom_R(P_n,F)$, we have $\Ext_R^i(\Hom_R(A,F),U)=
\Ext_R^{i+n}(L,U)=0$ for all $i>0$, as desired.

 By Theorem~\ref{finite-resol-dim-tilde}, we can infer that
$\F_A^{\perp_1}=\C_A=\Cof_{n+1}(\Hom_R(A,\C))^\oplus$.
\end{exs}

 Now we will explain how a stronger and more general version of
the results of Examples~\ref{cotilting-examples} can be obtained
with an approach based on a suitable version of the Bongartz--Ringel
lemma, in the spirit of~\cite[Example~2.3]{ST}.

\begin{thm}[dual Bongartz--Ringel lemma] \label{dual-n-bongartz}
 Let $A$ be an associative ring, $n\ge0$ be an integer, and
$\T=\{S_0,S_1,\dotsc,S_n\}$ be a collection of $n+1$ left $A$\+modules.
 Assume that $S_0$ is an injective cogenerator of $A\Modl$ and\/
$\Ext_A^1(S_j^\kappa,S_i)=0$ for all\/ $0\le i\le j\le n$ and all
cardinals~$\kappa$.
 Let $(\F,\C)=({}^{\perp_1}\T,({}^{\perp_1}\T)^{\perp_1})$ be
the cotorsion pair cogenerated by the set $\T$ in $A\Modl$.
 Then \par
\textup{(a)} $(\F,\C)$ is a complete cotorsion pair; \par
\textup{(b)} the class $\C\subset A\Modl$ can be described as the class
of all direct summands of $(n+1)$\+cofiltered left $A$\+modules $D$
with a cofiltration $D=G_{n+1}D\twoheadrightarrow G_nD\twoheadrightarrow
\dotsb\twoheadrightarrow G_1D\twoheadrightarrow G_0D=0$ such that\/
$\ker(G_{i+1}D\to G_iD)\in\Prod(S_i)$ for every\/ $0\le i\le n$.
 In particular, we have
$\C=\Cof_{n+1}(\bigcup_{i=0}^n\Prod(S_i))^\oplus$.
\end{thm}

\begin{proof}
 This is an $n\ge1$ generalization of the classical dual Bongartz
lemma~\cite[Proposition~6.44]{GT}, which corresponds to the case $n=1$.
 At the same time, this is an infinitely generated dual version of
a result of Ringel, who considered the $n\ge1$ case for finitely
generated modules over Artinian algebras~\cite[Lemma~4$'$]{Rin}.

 Denote by $\D$ the class of all left $A$\+modules admitting
an $(n+1)$\+cofiltration $G$ with the successive quotient modules as
described in part~(b).
 Then it is clear that one has $\Ext_A^1(F,D)=0$ for all $F\in\F$ and
$D\in\D$.
 Moreover, $\Ext_A^1(F,C)=0$ for all $F\in\F$ and
$C\in\Cof(\bigcup_{i=0}^n\Prod(S_i))^\oplus$.
 In order to prove the theorem, it remains to show that the pair of
classes $\F$ and $\D\subset A\Modl$ admits approximation sequences
(cf.\ Lemma~\ref{direct-summand-lemma}).

 Let us first show that the pair of classes $\F$ and $\D$ admits
special precover sequences.
 Let $M$ be a left $A$\+module; put $G_1F=M$.
 Denote by $I_1$ the underlying set (or any generating subset) of
the abelian group $\Ext^1_A(M,S_1)$.
 Then we have a canonical element in the abelian group
$\Ext^1_A(M,S_1)^{I_1}=\Ext^1_A(M,S_1^{I_1})$.
 Let $G_2F$ denote the middle term of the related short exact sequence
of left $A$\+modules $0\rarrow S_1^{I_1}\rarrow G_2F\rarrow G_1F=M
\rarrow0$.
 Notice that $\Ext^1_A(G_2F,S_1)=0$ by construction (in view of
the assumption that $\Ext^1_A(S_1^{I_1},S_1)=0$).

 Denote by $I_2$ the underlying set (or any generating subset) of
the abelian group $\Ext^1_A(G_2F,S_2)$.
 Then we have a canonical element in the abelian group
$\Ext^1_A(G_2F,S_2)^{I_2}=\Ext^1_A(G_2F,S_2^{I_2})$.
 Let $G_3F$ denote the middle term of the related short exact sequence
of left $A$\+modules $0\rarrow S_2^{I_2}\rarrow G_3F\rarrow G_2F
\rarrow0$.
 Notice that $\Ext^1_A(G_3F,S_1)=0$ since $\Ext^1_A(G_2F,S_1)=0
=\Ext^1_A(S_2^{I_2},S_1)$, and $\Ext^1_A(G_3F,S_2)=0$ by construction
(as $\Ext^1_A(S_2^{I_2},S_2)=0$).

 Proceeding in this way until all the modules $S_1$,~\dots, $S_n$ have
been taken into account, we construct a sequence of sets
$I_1$, $I_2$,~\dots, $I_n$ and a left $A$\+module $F$ with
an $(n+1)$\+step cofiltration $F=G_{n+1}F\twoheadrightarrow G_nF
\twoheadrightarrow\dotsb\twoheadrightarrow G_2F\twoheadrightarrow G_1F=M
\twoheadrightarrow G_0F=0$ such that $\ker(G_{i+1}F\to G_iF)\simeq
S_i^{I_i}$ for $1\le i\le n$.
 Furthermore, we have $\Ext^1_A(G_jF,S_i)=0$ for all $0\le i<j\le n+1$,
so in particular $\Ext^1_A(F,S_i)=0$ for all $0\le i\le n$.
 Denoting by $D'$ the kernel of the surjective morphism $F=G_{n+1}F
\rarrow G_1F=M$, we obtain the desired special precover sequence
$0\rarrow D'\rarrow F\rarrow M\rarrow0$ with $D'\in\D$ and $F\in\F$.
 Here the left $A$\+module $D'$ is endowed with a cofiltration $G$ as in
part~(b), with the additional property that $G_1D'=0$.

 To produce a special preenvelope sequence $0\rarrow N\rarrow D\rarrow
F'\rarrow0$ (with $D\in\D$ and $F'\in\F$) for a left $A$\+module $N$,
it now remains to choose a set $I_0$ such that $N$ is a submodule in
$S_0^{I_0}$ and use the construction from the proof of (the ``only if''
implication in) Lemma~\ref{salce-lemma}.
\end{proof}

\begin{rem} \label{deste-counterex-remark}
 For any $n$\+cotilting $R$\+module $U$, the induced $n$\+cotilting
cotorsion pair $(\F,\C)$ is obviously cogenerated by the cosyzygy
modules $U$, $\Omega^{-1}U$,~\dots, $\Omega^{-n+1}U$ of
the $R$\+module~$U$.
 So the first na\"\i ve idea of an application of
Theorem~\ref{dual-n-bongartz} to the cotilting cotorsion pairs
would be to consider the sequence of cosyzygy modules $S_n=U$,
\,$S_{n-1}=\Omega^{-1}U$,~\dots, $S_1=\Omega^{-n+1}U$.

 In fact, one has $\Ext_R^1(U^X,\Omega^{-i}U)\simeq
\Ext_R^{i+1}(U^X,U)=0$ for all $i\ge0$ and all sets~$X$.
 However, it may well happen that $\Ext_R^1(\Omega^{-1}U,\Omega^{-1}U)
\ne0$.
 Let $0\rarrow U\rarrow J^0\rarrow\Omega^{-1}U\rarrow0$ be a short exact
sequence of $R$\+modules with an injective $R$\+module~$J^0$;
then one has $\Ext_R^1(\Omega^{-1}U,\Omega^{-1}U)\simeq
\Ext_R^2(\Omega^{-1}U,U)\simeq\Ext_R^2(J^0,U)$, and there is no apparent
reason for this Ext group to vanish.

 The relevant counterexample was constructed by D'Este~\cite{Est}.
 One starts with the observation that, for any finite-dimensional
algebra $A$ over a field~$k$ such that $A$ has finite homological
dimension~$n$, the free $A$\+module $A$ is $n$\+cotilting.
 Furthermore, for any field~$k$, there is an acyclic quiver algebra $A$
of homological dimension~$2$, with $4$~vertices, $4$~edges, and
$2$~relations, such that $\Ext^1_A(\Omega^{-1}A,\Omega^{-1}A)\ne0$
(for the minimal cosyzygy module $\Omega^{-1}A$ of the free
$A$\+module~$A$) \cite[Theorem~5]{Est}.

 Therefore, the dual Bongartz--Ringel lemma is \emph{not} applicable
to the sequence of cosyzygy modules $\Omega^{-j}U$ of an $n$\+cotilting
$R$\+module $U$ (generally speaking), and the na\"\i ve idea does not
work.
 That is why an approach based on the next lemma is needed instead.

 On the other hand, for any $n$\+cotilting module $U$ over
a commutative Noetherian ring $R$, one can choose an injective
coresolution of $U$ in such a way that the related sequence of
cosyzygy modules $U$, $\Omega^{-1}U$,~\dots, $\Omega^{-n+1}U$,
\,$\Omega^{-n}U=J$ would satisfy the assumptions of
Theorem~\ref{dual-n-bongartz} (see~\cite[Corollary~3.17]{APST}).
\end{rem}

\begin{lem} \label{derived-cotilting-classes}
 Let $R$ be an associative ring and $U$ be an $n$\+cotilting
left $R$\+module.
 Then, for every\/ $0\le j\le n$, there exists an $(n-j)$\+cotilting
left $R$\+module $U_j$ such that the class $\F_j={}^{\perp_{>j}}U$ is
the cotilting class induced by $U_j$ in $R\Modl$.
 In particular, one can (and we will) take $U_0=U$, while
$J=U_n$ is an injective cogenerator of $R\Modl$.
\end{lem}

\begin{proof}
 The proof of this, classical by now, result is
based on~\cite[Lemma~3.4]{Baz}.
 One can combine the assertions of~\cite[Theorem~4.2]{Baz2}
and~\cite[Theorem~13]{Sto0} with~\cite[Theorem~4.2]{AC}.
 Alternatively, see~\cite[Proposition~15.13]{GT}.
\end{proof}

\begin{lem} \label{cogenerated-by-cotilting-modules}
 In the notation of Lemma~\ref{derived-cotilting-classes},
the $n$\+cotilting class $\F={}^{\perp_{>0}}U$ can be described as
$\F={}^{\perp_1}\{U_0,U_1,\dotsc,U_n\}=\{F\in R\Modl \mid
\Ext^1_R(F,U_j)=0 \ \forall\,0\le j\le n\}$.
\end{lem}

\begin{proof}
 For any $i$, $j\ge0$ we have
${}^{\perp_{>i}}U_j={}^{\perp_{>i+j}}U$,
since a left $R$\+module $F$ belongs to ${}^{\perp_{>i}}U_j$
if and only if the left $R$\+module $\Omega^iF$ belongs to
${}^{\perp_{>0}}U_j$, which means that $\Omega^iF$ belongs to
${}^{\perp_{>j}}U$, which holds if and only if $F$ belongs to
${}^{\perp_{>i+j}}U$.
 In particular, it follows that ${}^{\perp_{>1}}U_j={}^{\perp_{>j+1}}U
={}^{\perp_{>0}}U_{j+1}$.

 Now ${}^{\perp_1}U_n=R\Modl$ and ${}^{\perp_1}U_{n-1}=
{}^{\perp_{>0}}U_{n-1}={}^{\perp_{>1}}U_{n-2}$.
 Proceeding by decreasing induction in $0\le j\le n$, one proves that
${}^{\perp_1}\{U_j,U_{j+1},\dotsc,U_n\}={}^{\perp_{>0}}U_j$, since
${}^{\perp_1}U_j\cap{}^{\perp_1}\{U_{j+1},U_{j+2},\dotsc,U_n\}=
{}^{\perp_1}U_j\cap{}^{\perp_{>0}}U_{j+1}=
{}^{\perp_1}U_j\cap{}^{\perp_{>1}}U_j={}^{\perp_{>0}}U_j$.
\end{proof}

 The next theorem is a generalization of~\cite[Example~2.3]{ST}
(which corresponds to the case of $n=1$).

\begin{thm} \label{right-cotilting-class-described}
 Let $R$ be an associative ring and $U$ be an $n$\+cotilting
left $R$\+module.
 Let $(\F,\C)$ be the $n$\+cotilting cotorsion pair induced by $U$ in
$R\Modl$.
 Then the class $\C$ can be described as the class of all
direct summands of $(n+1)$\+cofiltered left $R$\+modules $D$ with
a cofiltration $D=G_{n+1}D\twoheadrightarrow G_nD\twoheadrightarrow
\dotsb\twoheadrightarrow G_1D\twoheadrightarrow G_0D=0$ such that,
in the notation of Lemma~\ref{derived-cotilting-classes},
\,$G_1D\in\Prod(J)$, \ $\ker(G_{i+1}D\to G_iD)\in\Prod(U_{n-i})$
for every\/ $0\le i\le n$, and\/ $\ker(G_{n+1}D\to G_nD)\in\Prod(U)$.
\end{thm}

\begin{proof}
 By Lemma~\ref{cogenerated-by-cotilting-modules}, the $n$\+cotilting
cotorsion pair $(\F,\C)$ is cogenerated by the set of $n+1$~modules
$S_0=J$, $S_1=U_{n-1}$, $S_2=U_{n-2}$,~\dots, $S_n=U$.
 Furthermore, one has $\Ext^m_R(U_i^\kappa,U_j)=0$ for all integers
$0\le i\le j\le n$, \,$m>0$ and all cardinals~$\kappa$, since
$U_i^\kappa\in{}^{\perp_{>0}}U_i\subset{}^{\perp_{>0}}U_j$.
 Thus Theorem~\ref{dual-n-bongartz}(b) is applicable.
\end{proof}

 The result of the following corollary generalizes those of
Examples~\ref{cotilting-examples}.

\begin{cor} \label{finite-filtration-by-coinduced-via-dual-n-bongartz}
 Let $R\rarrow A$ be a homomorphism of associative rings and $U$ be
an $n$\+cotilting left $R$\+module.
 Let $(\F,\C)$ be the $n$\+cotilting cotorsion pair induced by $U$ in
$R\Modl$.
 Assume that the underlying left $R$\+module of $A$ belongs to $\F$,
that is ${}_RA\in\F$.
 Assume further that the left $A$\+module\/ $\Hom_R(A,U_j)$ is
$(n-j)$\+cotilting for every\/ $0\le j\le n$.
 (In particular, by Proposition~\ref{bazzoni-prop}, this holds whenever
$R$ is commutative and $A$ is an $R$\+algebra.)
 Let $(\F_A,\C_A)$ be the $n$\+cotilting cotorsion pair induced by
$U'=\Hom_R(A,U)$ in $A\Modl$.
 Then we have $\C_A=\Cof_{n+1}(\Hom_R(A,\C))^\oplus$.
\end{cor}

\begin{proof}
 Put $U'_j=\Hom_R(A,U_j)\in A\Modl$ for all $0\le j\le n$.
 Then the class ${}^{\perp_{>0}}U'_j\subset A\Modl$ consists of
all the left $A$\+modules whose underlying left $R$\+modules belong to
the class ${}^{\perp_{>0}}U_j\subset R\Modl$, in view of
Lemma~\ref{homological-formulas-lemma}(b)
(since ${}_RA\in\F\subset\F_j$).
 Similarly, the class ${}^{\perp_{>j}}U'\subset A\Modl$ consists of
all the left $A$\+modules whose underlying left $R$\+modules belong to
the class ${}^{\perp_{>j}}U\subset R\Modl$.
 Hence ${}^{\perp_{>0}}U'_j={}^{\perp_{>j}}U'\subset A\Modl$.

 Notice further that $U_j\in\C$ for all $0\le j\le n$, as
$\F\subset\F_j={}^{\perp_{>0}}U_j\subset{}^{\perp_1}U_j$.
 Now the assertion of the corollary follows from
Theorem~\ref{right-cotilting-class-described} applied to the ring $A$
and the cotilting $A$\+modules $U'=U'_0$, $U_1'$,~\dots,~$U'_n$.
\end{proof}

\subsection{Decreasing filtrations by coinduced modules}
\label{decreasing-filtrations-subsecn}
 Let $R\rarrow A$ be a homomorphism of associative rings, and
let $(\F,\C)$ be an $\Ext^1$\+orthogonal pair of classes of
left $R$\+modules.
 Instead of assuming finiteness of the $\F$\+resolution dimension,
we now assume that the class $\F$ is closed under countable products
in $R\Modl$.

 As usually, we denote by~$\omega$ the first infinite ordinal,
that is the ordinal of nonnegative integers.
 The ``cofiltrations'' appearing in the next proposition are
the usual complete, separated infinite decreasing filtrations
indexed by the natural numbers.

\begin{prop} \label{product-closed-approximation-sequences}
 Assume that the\/ $\Ext^1$\+orthogonal pair of classes of
left $R$\+modules $(\F,\C)$ admits approximation
sequences~\textup{(\ref{sp-precover-seq}\+-\ref{sp-preenvelope-seq})}.
 Assume that the underlying left $R$\+module of $A$ belongs to~$\F$,
and that the condition~\textup{($\dagger\dagger$)} holds.
 Assume further that the class $\F$ is closed under the kernels of
surjective morphisms and countable products in $R\Modl$.
 Then the\/ $\Ext^1$\+orthogonal pair of classes of left $A$\+modules
$\F_A$ and\/ $\Cof_\omega(\Hom_R(A,\C))$ admits approximation
sequences as well.
\end{prop}

\begin{proof}
 The pair of classes $\F_A$ and $\Cof(\Hom_R(A,\C))\subset A\Modl$ is
$\Ext^1$\+orthogonal by Lemma~\ref{orthogonal-to-coinduced-lemma}(c).
 The explicit construction below, showing that the pair of classes
$\F_A$ and $\Cof_\omega(\Hom_R(A,\C))$ admits special precover
sequences, plays a key role.
 It goes back to~\cite[semicontramodule-related assertions in
Lemma~3.3.3]{Psemi}.

 Let $M$ be a left $A$\+module.
 We proceed with the construction from the proof of
Proposition~\ref{finite-resol-dim-approximation-sequences}; but
instead of a finite number of~$k$ iterations, we perform
$\omega$~iterations now.
 So we produce a sequence of surjective morphisms of left
$A$\+modules
\begin{equation} \label{Q-projective-system}
 M\llarrow Q(M)\llarrow Q(Q(M))\llarrow\dotsb\llarrow Q^n(M)
 \llarrow\dotsb,
\end{equation}
where $n$~ranges over the nonnegative integers.
 Clearly, the kernel of the surjective left $A$\+module morphism
$\varprojlim_{n\in\omega}Q^n(M)\rarrow M$ is $\omega$\+cofiltered
by the left $A$\+modules $\Hom_R(A,C'(Q^n(M)))$, \,$n\in\omega$,
which belong to $\Hom_R(A,\C)$ by construction. 
 Now the claim is that the left $A$\+module
$\varprojlim_{n\in\omega} Q^n(M)$ belongs to~$\F_A$.

 Recall that the injective $A$\+module morphism $\nu_M\:M
\rarrow\Hom_R(A,M)$ admits a natural $R$\+linear retraction
$\phi_M\:\Hom_R(A,M)\rarrow M$.
 Looking on the diagram~\eqref{precover-construction-diagram},
one can see that the surjective map $Q(M)\rarrow M$ factorizes
as $Q(M)\rarrow\Hom_R(A,F(M))\rarrow M$.
 Here $Q(M)\rarrow\Hom_R(A,F(M))$ is an $A$\+module morphism, but
$\Hom_R(A,F(M))\rarrow M$ is only an $R$\+module morphism
(between left $A$\+modules).
 Thus the sequence of surjective morphisms of left
$A$\+modules~\eqref{Q-projective-system} is mutually cofinal with
a sequence of left $R$\+module morphisms
\begin{multline} \label{F-Q-projective-system}
 \Hom_R(A,F(M))\llarrow\Hom_R(A,F(Q(M)))\llarrow\dotsb \\
 \llarrow\Hom_R(A,F(Q^n(M)))\llarrow\dotsb
\end{multline}

 The left $R$\+modules $F(Q^n(M))$, \,$n\ge0$, belong to $\F$
by construction.
 According to~($\dagger\dagger$), it follows that the underlying
left $R$\+modules of the left $A$\+modules $\Hom_R(A,F(Q^n(M)))$
belong to $\F$, too.
 The derived projective limits of mutually cofinal projective systems
agree, hence
$$
 \varprojlim\nolimits^1_{n\in\omega}\Hom_R(A,F(Q^n(M)))\simeq
 \varprojlim\nolimits^1_{n\in\omega}Q^n(M)=0,
$$
as the maps $Q^{n+1}(M)\rarrow Q^n(M)$ are surjective.
 Therefore, we have a short exact sequence of left $R$\+modules
\begin{multline} \label{projlimit-telescope}
 0\lrarrow\varprojlim\nolimits_{n\in\omega}\Hom_R(A,F(Q^n(M))) \\
 \lrarrow\prod\nolimits_{n\in\omega}\Hom_R(A,F(Q^n(M)))
 \lrarrow\prod\nolimits_{n\in\omega}\Hom_R(A,F(Q^n(M)))
 \lrarrow0.
\end{multline}
 Since the class $\F\subset R\Modl$ is closed under countable
products and the kernels of surjective morphisms by assumption,
it follows that the left $R$\+module
$\varprojlim_{n\in\omega}\Hom_R(A,F(Q^n(M)))$ belongs to~$\F$.

 Furthermore, the underived projective limits of mutually cofinal
projective systems also agree; so we have an isomorphism of
left $R$\+modules
$$
 \varprojlim\nolimits_{n\in\omega}Q^n(M)\simeq
 \varprojlim\nolimits_{n\in\omega}\Hom_R(A,F(Q^n(M))).
$$
 Since $\varprojlim_{n\in\omega}\Hom_R(A,F(Q^n(M)))\in\F$,
we can conclude that $\varprojlim_{n\in\omega}Q^n(M)\in\F_A$,
as desired.
 This finishes the construction of the special precover sequences
for the pair of classes of left $A$\+modules $\F_A$ and
$\Cof_\omega(\Hom_R(A,\C))$.

 At last, the special preenvelope sequences for the pair of classes
$\F_A$ and $\Cof_\omega(\Hom_R(A,\C))\subset A\Modl$ are produced
from the special precover sequences in the same way as in
the last paragraph of the proof of
Proposition~\ref{finite-resol-dim-approximation-sequences}.
\end{proof}

\begin{thm} \label{product-closed-theorem}
  Let $(\F,\C)$ be a hereditary complete cotorsion pair in $R\Modl$.
 Assume that the left $R$\+module $A$ belongs to~$\F$, and
that the condition~\textup{($\dagger\dagger$)} holds.
 Assume further that the class $\F$ is closed under countable
products in $R\Modl$.
 Then the pair of classes $\F_A$ and
$\C_A=\Cof_\omega(\Hom_R(A,\C))^\oplus$ is a hereditary complete
cotorsion pair in $A\Modl$.
\end{thm}

\begin{proof}
 Follows from Proposition~\ref{product-closed-approximation-sequences}
in view of Lemma~\ref{direct-summand-lemma} (cf.\
the proof of Theorem~\ref{finite-resol-dim-theorem}).
\end{proof}

\begin{cor} \label{product-closed-cor}
 For any associative ring homomorphism $R\rarrow A$ and any
hereditary complete cotorsion pair $(\F,\C)$ in $R\Modl$ satisfying
the assumptions of Theorem~\ref{product-closed-theorem}, one has
$\F_A^{\perp_1}=\Cof_\omega(\Hom_R(A,\C))^\oplus$.
 In particular, it follows that\/ $\Cof(\Hom_R(A,\C))^\oplus=
\Cof_\omega(\Hom_R(A,\C))^\oplus$.
\end{cor}

\begin{proof}
 This is a corollary of Theorem~\ref{product-closed-theorem}
and Lemma~\ref{orthogonal-to-coinduced-lemma}(c)
(cf.\ the proof of Corollary~\ref{finite-resol-dim-cor}).
\end{proof}

\begin{rem} \label{semicontramodules-remark}
 As mentioned in Remark~\ref{corings-and-comodules-remark},
the condition~($\dagger\dagger$) appears to be rather restrictive.
 In fact, the construction of
Proposition~\ref{product-closed-approximation-sequences} originates
from the theory of semicontramodules over semialgebras,
as in~\cite[Lemma~3.3.3]{Psemi}, where the natural analogue of this
condition feels much less restrictive, particularly when $\F$ is
simply the class of all projective objects.
 So one can say that the ring $R$ in this
Section~\ref{decreasing-filtrations-subsecn} really ``wants''
to be a coalgebra $C$ (say, over a field~$k$), and accordingly
the ring $A$ becomes a semialgebra $S$ over~$C$.
 The left $R$\+modules ``want'' to be left $C$\+contramodules, and
the left $A$\+modules ``want'' to be left $S$\+semicontramodules.

 Then the coinduction functor, which was $\Hom_R(A,{-})$ in
the condition~($\dagger\dagger$), takes the form of
the functor $\Cohom_C(S,{-})$.
 This one is much more likely to take projective left $C$\+contramodules
to projective left $C$\+contramodules.
 In fact, all projective $C$\+contramodules are direct summands
of the free contramodules $\Hom_k(C,V)$, where $V$ ranges over
$k$\+vector spaces; and one has $\Cohom_C(S,\Hom_k(C,V))\simeq
\Hom_k(S,V)$.
 This is a projective left $C$\+contramodule for any $V$ whenever
the right $C$\+comodule $S$ is injective.
 Besides, the class of all projective contramodules over
a coalgebra over a field is always closed under infinite products;
so the specific assumption of
Section~\ref{decreasing-filtrations-subsecn} is satisfied in
the contramodule context, too.

 To make a ring $R$ behave rather like a coalgebra, one can assume it
to be ``small'' in some sense.
 The following examples are inspired by the analogy with semialgebras
and semicontramodules.
\end{rem}

\begin{exs} \label{semicontramodules-examples}
 Let $\F=R\Modl_\proj$ be the class of all projective left
$R$\+modules; then $\C=R\Modl$ is the class of all left $R$\+modules
(cf.\ Example~\ref{corings-and-comodules-example}).

\smallskip
 (1)~Assume that the ring $R$ is left perfect and right coherent
(e.~g., it suffices that $R$ be right Artinian).
 Then the class of all projective left $R$\+modules is closed under
infinite products~\cite{Bas,Ch}; so the specific assumption of
Section~\ref{decreasing-filtrations-subsecn} is satisfied.

 Furthermore, all flat left $R$\+modules are projective, and
all left $R$\+modules have projective covers~\cite{Bas}.
 Let $J\subset R$ be the Jacobson radical; then the correspondence
$P\longmapsto P/JP$ is a bijection between the isomorphism classes
of projective left $R$\+modules and arbitrary $R/J$\+modules.
 The quotient ring $R/J$ is classically semisimple, so it is
isomorphic to a finite product of simple Artinian rings
$R_1$,~\dots,~$R_m$.
 Denote by $J_i\subset R$ the kernel of the surjective map
$R\rarrow R_i$, \ $1\le i\le m$.
 Then, choosing~$\kappa$ to be a large enough cardinal, one can make
the (semisimple) $R_i$\+module $R^\kappa/J_i(R^\kappa)$
arbitrarily large.
 Therefore, all the projective left $R$\+modules are direct
summands of products of copies of the free left $R$\+module~$R$.

 Assume further that the left $R$\+module $\Hom_R(A,R)$ is projective.
 Then it follows that the functor $\Hom_R(A,{-})$ preserves the class
$\F$ of all projective left $R$\+modules.
 Thus the condition~($\dagger\dagger$) is satisfied.

\medskip
 (2)~Assume that $R$ is a finite-dimensional algebra over a field~$k$
and $R\rarrow A$ is a morphism of $k$\+algebras.
 This is a particular case of~(1), so the above discussion is
applicable.
 Furthermore, we have $\Hom_R(A,R)\simeq\Hom_R(A,R^{**})\simeq
(R^*\ot_RA)^*$, where $V\longmapsto V^*$ denotes the passage to
the dual $k$\+vector space.

 The functor $N\longmapsto N^*$ takes injective right $R$\+modules
to projective left $R$\+modules.
 Thus the condition~($\dagger\dagger$) holds whenever the underlying
right $R$\+module of the right $A$\+module $R^*\ot_RA$ is injective.

\medskip
 (3)~Assume that $R$ is a quasi-Frobenius ring, i.~e., the classes of
injective and projective left $R$\+modules coincide (and the same
holds for right $R$\+modules).
 All such rings $R$ are left and right Artinian, so the discussion
in~(1) is applicable.

 Furthermore, whenever $R$ is quasi-Frobenius,
the condition~($\dagger\dagger$) can be rephrased by saying that
the functor $\Hom_R(A,{-})$ takes injective left $R$\+modules to
injective left $R$\+modules.
 This holds whenever $A$ is a projective right $R$\+module.

\medskip
 The results of Section~\ref{decreasing-filtrations-subsecn} tell us
that, whenever the left $R$\+module $A$ is projective and any one of
the above sets of conditions~(1\+-3) is satisfied,
the $\Ext^1$\+orthogonal pair of classes of left $A$\+modules
$A\Modl_{R\dproj}$ and $\Cof_\omega(\Hom_R(A,R\Modl))$ admits
approximation sequences.
 Consequently, the pair of classes $\F_A=A\Modl_{R\dproj}$ and
$\C_A=\Cof_\omega(\Hom_R(A,R\Modl))^\oplus$ is a hereditary complete
cotorsion pair in $A\Modl$.
 In particular, we have
$$
 (A\Modl_{R\dproj})^{\perp_1}=\Cof_\omega(\Hom_R(A,R\Modl))^\oplus,
$$
and therefore $\Cof(\Hom_R(A,R\Modl))^\oplus=
\Cof_\omega(\Hom_R(A,R\Modl))^\oplus$.
 So the weakly $A/R$\+injective left $A$\+modules are precisely
the direct summands of the $A$\+modules admitting a complete,
separated $\omega$\+indexed decreasing filtration by $A$\+modules
coinduced from left $R$\+modules.
\end{exs}

 The next theorem is a generalization of
Corollary~\ref{product-closed-cor} in which
the condition~($\dagger\dagger$) is replaced by
the condition~($\widetilde{\dagger\dagger}$).

\begin{thm} \label{product-closed-tilde}
 Let $(\F,\C)$ be a hereditary complete cotorsion pair in $R\Modl$.
 Assume that the left $R$\+module $A$ belongs to $\F$, and that
the condition~\textup{($\widetilde{\dagger\dagger}$)} holds.
 Assume further that the class $\F$ is closed under countable
products in $R\Modl$.
 Then the class $\C_A=\F_A^{\perp_1}\subset A\Modl$ can be described
as $\C_A=\Cof_\omega(\Hom_R(A,\C))^\oplus$.
 In particular, we have\/ $\Cof(\Hom_R(A,\C))^\oplus=
\Cof_\omega(\Hom_R(A,\C))^\oplus$.
\end{thm}

\begin{proof}
 Similar to the proof of Theorem~\ref{finite-resol-dim-tilde}, where
all the essential details have been already worked out.
 One follows the proof of Corollary~\ref{product-closed-cor} step by
step and observes that the assumptions of the present theorem are
sufficient for the validity of the argument.
\end{proof}

\subsection{Combined result on coinduced modules}
\label{combined-coinduced-subsecn}
 In this section we combine the constructions of
Propositions~\ref{finite-resol-dim-approximation-sequences}
and~\ref{product-closed-approximation-sequences} in order to
obtain a more general result under relaxed assumptions.
 Specifically, we assume that all the countable products of modules
from $\F$ have finite $\F$\+resolution dimensions.

\begin{prop} \label{combined-coinduced-approximation-sequences}
 Assume that the\/ $\Ext^1$\+orthogonal pair of classes of
left $R$\+modules $(\F,\C)$ admits approximation
sequences~\textup{(\ref{sp-precover-seq}\+-\ref{sp-preenvelope-seq})}.
 Assume that the left $R$\+module $A$ belongs to~$\F$, and
that the condition~\textup{($\dagger\dagger$)} holds.
 Assume further that the class $\F$ is resolving in $R\Modl$ and
the $\F$\+resolution dimension of any countable product of modules
from $\F$ does not exceed a finite integer~$k\ge0$.
 Then the\/ $\Ext^1$\+orthogonal pair of classes of left $A$\+modules
$\F_A$ and\/ $\Cof_{\omega+k}(\Hom_R(A,\C))$ admits approximation
sequences as well.
 Here $\omega+k$ is the $k$\+th successor ordinal of~$\omega$.
\end{prop}

\begin{proof}
 As in previous proofs, we start with an explicit construction
of special precover sequences for the pair of classes $\F_A$
and $\Cof_{\omega+k}(\Hom_R(A,\C))\subset A\Modl$.

 Let $M$ be a left $A$\+module.
 Proceeding as in the proof of
Proposition~\ref{product-closed-approximation-sequences}, we
construct the $\omega$\+indexed projective system of surjective
morphisms of left $A$\+modules~\eqref{Q-projective-system}.
 The underlying left $R$\+module of the left $A$\+module
$\varprojlim_{n\in\omega}Q^n(M)$ is isomorphic to the projective
limit of the projective system of left
$R$\+modules~\eqref{F-Q-projective-system}, and it can be
described as the leftmost term of the short exact
sequence~\eqref{projlimit-telescope}.

 The left $R$\+modules $\Hom_R(A,F(Q^n(M)))$ belong to $\F$
by~($\dagger\dagger$), so the left $R$\+module
$\prod_{n\in\omega}\Hom_R(A,F(Q^n(M)))$ has $\F$\+resolution
dimension~$\le k$ in our present assumptions.
 By Lemma~\ref{co-resolution-dimension-defined-classes}(a),
it follows that the $\F$\+resolution dimension of (the underlying
left $R$\+module of the left $A$\+module)
$N=\varprojlim_{n\in\omega}Q^n(M)$ does not exceed~$k$.

 Now we apply the construction from the proof of
Proposition~\ref{finite-resol-dim-approximation-sequences}
to the left $A$\+module~$N$, producing the sequence of
surjective morphisms of left $A$\+modules
$$
 N\llarrow Q(N)\llarrow Q(Q(N))\llarrow\dotsb\llarrow Q^k(N).
$$
 Following the argument in the proof of
Proposition~\ref{finite-resol-dim-approximation-sequences}, we have
$Q^k(N)\in\F_A$, since $\rd_\F N\le k$.
 Finally, the kernel of the composition of surjective morphisms
$$
 Q^k(N)\lrarrow N=\varprojlim\nolimits_{n\in\omega} Q^n(M)\lrarrow M
$$
is an extension of the kernels of the morphisms
$Q^k(N)\rarrow N$ and $\varprojlim_{n\in\omega} Q^n(M)\rarrow M$.
 The former kernel belongs to $\Cof_k(\Hom_R(A,\C))$ and
the latter one to $\Cof_\omega(\Hom_R(A,\C))$; thus
the kernel of the morphism $Q^k(N)\rarrow M$ belongs to
$\Cof_{\omega+k}(\Hom_R(A,\C))$.

 We have produced the desired special precover sequences.
 Having these at our disposal, the special preenvelope sequences are
constructed in the same way as in the proofs of
Propositions~\ref{finite-resol-dim-approximation-sequences}
and~\ref{product-closed-approximation-sequences}.
\end{proof}

\begin{thm} \label{combined-coinduced-theorem}
 Let $(\F,\C)$ be a hereditary complete cotorsion pair in $R\Modl$.
 Assume that the left $R$\+module $A$ belongs to~$\F$, and
that the condition~\textup{($\dagger\dagger$)} holds.
 Assume further that the $\F$\+resolution dimension of any
countable product of modules from $\F$ in $R\Modl$ does not
exceed a finite integer $k\ge0$.
 Then the pair of classes $\F_A$ and
$\C_A=\Cof_{\omega+k}(\Hom_R(A,\C))^\oplus$ is a hereditary complete
cotorsion pair in $A\Modl$.
\end{thm}

\begin{proof}
 Follows from
Proposition~\ref{combined-coinduced-approximation-sequences}
in view of Lemma~\ref{direct-summand-lemma}.
\end{proof}

\begin{cor} \label{combined-coinduced-cor}
 For any associative ring homomorphism $R\rarrow A$ and any
hereditary complete cotorsion pair $(\F,\C)$ in $R\Modl$ satisfying
the assumptions of Theorem~\ref{combined-coinduced-theorem}, one has
$\F_A^{\perp_1}=\Cof_{\omega+k}(\Hom_R(A,\C))^\oplus$.
 In particular, it follows that\/ $\Cof(\Hom_R(A,\C))^\oplus=
\Cof_{\omega+k}(\Hom_R(A,\C))^\oplus$.
\end{cor}

\begin{proof}
 This is a corollary of Theorem~\ref{combined-coinduced-theorem}
and Lemma~\ref{orthogonal-to-coinduced-lemma}(c).
\end{proof}

 For a class of examples to Theorem~\ref{combined-coinduced-theorem}
arising from curved DG\+rings, see
Proposition~\ref{doublestar-contraacyclic-prop} below.

 The final theorem of this section is a generalization of
Corollary~\ref{combined-coinduced-cor} in which
the condition~($\dagger\dagger$) is replaced by
the condition~($\widetilde{\dagger\dagger}$).

\begin{thm} \label{combined-coinduced-tilde}
 Let $(\F,\C)$ be a hereditary complete cotorsion pair in $R\Modl$.
 Assume that the left $R$\+module $A$ belongs to $\F$, and that
the condition~\textup{($\widetilde{\dagger\dagger}$)} holds.
 Assume further that the $\F$\+resolution dimension of any
countable product of modules from $\F$ in $R\Modl$ does not
exceed a finite integer $k\ge0$.
 Then the class $\C_A=\F_A^{\perp_1}\subset A\Modl$ can be described
as $\C_A=\Cof_{\omega+k}(\Hom_R(A,\C))^\oplus$.
 In particular, we have\/ $\Cof(\Hom_R(A,\C))^\oplus=
\Cof_{\omega+k}(\Hom_R(A,\C))^\oplus$.
\end{thm}

\begin{proof}
 One follows the proof of Corollary~\ref{combined-coinduced-cor} step by
step and observes that the assumptions of the present theorem are
sufficient for the validity of the argument.
 Almost all the essential details have been worked out already in
the proof of Theorem~\ref{finite-resol-dim-tilde}, and only one
observation remains to be made.

 Let $M$ be a left $A$\+module whose underlying left $R$\+module
belongs to~$\D$.
 Then the underlying left $R$\+module of the left $A$\+module
$\varprojlim_{n\in\omega}Q^n(M)$ also belongs to $\D$, because
the kernel of the surjective $A$\+module morphism
$\varprojlim_{n\in\omega}Q^n(M)\rarrow M$ belongs to
$\Cof_\omega(\Hom_R(A,\C))\subset\C_A$ and the class
$\D\subset R\Modl$ is closed under extensions.
\end{proof}

\Section{Filtrations by Induced Modules} \label{filtrations-secn}

 The setting in this section is dual to that in
Section~\ref{cofiltrations-secn}, and the main results are also dual.
 But the ambient context of the general theory of cotorsion pairs in
module categories, based on the small object argument etc., is
\emph{not} self-dual.
 So we discuss the situation in detail, making both the similarities
and the differences visible.
 
\subsection{Posing the problem}
 Let $R\rarrow A$ be a homomorphism of associative rings, and let
$\C$ be a class of left $R$\+modules.
 Mostly we will assume $\C$ to be the right part of a cotorsion pair
$(\F,\C)$ in $R\Modl$.

 Denote by $\C^A$ the class of all left $A$\+modules whose underlying
$R$\+modules belong to~$\C$.
 Does there exist a cotorsion pair $(\F^A,\C^A)$ in $A\Modl$\,?
 
 Obviously, if the answer to this question is positive, then
the class $\F^A$ can be recovered as $\F^A={}^{\perp_1}\C^A$.
 But can one describe the class $\F^A$ more explicitly?
 
 We start with an easy lemma providing a necessary condition.
 Here, for any ring $S$ and right $S$\+module $E$, we denote
by $E^+$ the left $S$\+module $E^+=\Hom_{\Z}(E,\Q/\Z)$
(which is called the \emph{character module} of~$E$).
 The $S$\+module $E^+$ is always cotorsion in the sense of
Enochs~\cite{En}; in fact, it is even
pure-injective~\cite[Chapter~2]{GT}.
 The left $S$\+module $E^+$ is injective whenever $E$ is a flat
right $S$\+module.
 In particular, the left $S$\+module $S^+$ is injective; in fact,
it is an injective cogenerator of $S\Modl$.

\begin{lem} \label{necessary-A-plus-in-C}
 Assume that $\C^A$ is the right part of a cotorsion pair
$(\F^A,\C^A)$ in $A\Modl$.
 Then the left $R$\+module $A^+$ belongs to~$\C$.
 Consequently, one has $\Tor^R_1(A,F)=0$ for any left $R$\+module
$F\in{}^{\perp_1}\C$.
\end{lem}

\begin{proof}
 For any cotorsion pair $(\F^A,\C^A)$ in $A\Modl$, all injective
left $A$\+modules belong to~$\C^A$.
 So, in the situation at hand, the underlying left $R$\+modules
of all injective left $A$\+modules must belong to~$\C$.
 This proves the first assertion.
 The second one follows from the natural isomorphism of abelian
groups $\Tor^R_1(A,F)^+\simeq\Ext_R^1(F,A^+)=0$.
\end{proof}

 The next lemma shows that this condition is also sufficient to
get a cotorsion pair $(\F^A,\C^A)$.
 Given a class of left $R$\+modules $\S$, we denote by $A\ot_R\S$
the class of all left $A$\+modules of the form $A\ot_RS$ with $S\in\S$.

\begin{lem} \label{orthogonal-to-induced-lemma}
 Let $(\F,\C)$ be a cotorsion pair in $R\Modl$ generated by a class
of left $R$\+modules~$\S$.
 Assume that the left $R$\+module $A^+$ belongs to~$\C$.
 Then we have \par
\textup{(a)} $\C^A=(A\ot_R\F)^{\perp_1}=(A\ot_R\S)^{\perp_1}$; \par
\textup{(b)} $({}^{\perp_1}\C^A,\C^A)$ is a cotorsion pair in
$A\Modl$; \par
\textup{(c)} $\Fil(A\ot_R\S)^\oplus\subset\Fil(A\ot_R\F)^\oplus
\subset{}^{\perp_1}\C^A$.
\end{lem}

\begin{proof}
 Part~(a): by assumptions, we have $\C=\S^{\perp_1}$ and
$\Ext_R^1(S,A^+)=0$ for all $S\in\S$, hence $\Tor^R_1(A,S)=0$.
 By Lemma~\ref{homological-formulas-lemma}(a) (for $n=1$), it follows
that a left $A$\+module $C$ belongs to $(A\ot_R\S)^{\perp_1}$
if and only if the underlying left $R$\+module of $C$ belongs
to~$\S^{\perp_1}$.
 In particular, this is applicable to $\S=\F$.

 Part~(b): in view of part~(a), $({}^{\perp_1}\C^A,\C^A)$ is
the cotorsion pair in $A\Modl$ generated by the class
$A\ot_R\S$ or $A\ot_R\F$.

 Part~(c) follows from part~(a) and Lemma~\ref{eklof-lemma}.
\end{proof}

 So we have answered our first question, but we want to know more.
 Can one guarantee that the cotorsion pair $(\F^A,\C^A)$ is complete?

\begin{prop} \label{generated-by-induced-prop}
 Let $(\F,\C)$ be a (complete) cotorsion pair in $R\Modl$ generated
by a set of left $R$\+modules $\S$, and let $\C^A$ be the class of
all left $A$\+modules whose underlying left $R$\+modules belong to~$\C$.
 Assume that the left $R$\+module $A^+$ belongs to~$\C$.
 Then there is a complete cotorsion pair $(\F^A,\C^A)$ in $A\Modl$
generated by the set of left $A$\+modules $\S^A=A\ot_R\S$.
 Moreover, one has $\F^A=\Fil(\S^A\cup\{A\})^\oplus$.
\end{prop}

\begin{proof}
 By Lemma~\ref{orthogonal-to-induced-lemma}(a\+b), the desired
cotorsion pair $(\F^A,\C^A)$ is generated by the set~$\S^A$.
 Hence both assertions follow from Theorem~\ref{eklof-trlifaj-theorem}.
\end{proof}

 These observations, based on the general theory of cotorsion pairs
in module categories, essentially answer all the questions above.
 We have a complete cotorsion pair $(\F^A,\C^A)$, and we also have
a description of the class~$\F^A$.
 Still we would like to improve upon these answers a little bit.

 In the rest of Section~\ref{filtrations-secn}, our aim is to show
that, under certain specific assumptions, the class $\F^A$ can be
described as $\F^A=\Fil_\beta(A\ot_R\F)^\oplus$ for rather small
ordinals~$\beta$.
 Besides, even though our assumptions are going to be rather
restrictive, we will \emph{not} assume the cotorsion pair $(\F,\C)$
to be generated by a set.

 For a class of examples of cotorsion pairs like in
Proposition~\ref{generated-by-induced-prop} arising in connection
with $n$\+tilting modules, see
Lemma~\ref{tilting-class-change-of-ring}
and Proposition~\ref{dual-bazzoni-prop} below.

\subsection{Finite filtrations by induced modules}
\label{finite-by-induced-subsecn}
 Let $R\rarrow A$ be a ring homomorphism.
 Suppose that we are given an $\Ext^1$\+orthogonal pair of classes
of left $R$\+modules $\F$ and $\C\subset R\Modl$, and denote by
$\C^A\subset A\Modl$ the class of all left $A$\+modules whose
underlying left $R$\+modules belong to~$\C$.

 For any left $R$\+module $M$, one can consider the left $A$\+module
$A\ot_RM$.
 Sometimes we will also consider the underlying left $R$\+module of
the left $A$\+module $A\ot_RM$.
 That is what we do when formulating the following condition, which
will be a key technical assumption in much of the rest of
Section~\ref{filtrations-secn}:
\begin{itemize}
\item[($\dagger$)] for any left $R$\+module $C\in\C$, the left
$R$\+module $A\ot_RC$ also belongs to~$\C$.
\end{itemize}

 The specific assumption on which the results of this
Section~\ref{finite-by-induced-subsecn} are based is that all left
$R$\+modules have finite $\C$\+coresolution dimension.

\begin{lem} \label{induction-preserves-coresolution-dimension}
 Assume that the\/ $\Ext^1$\+orthogonal pair of classes of
left $R$\+modules $(\F,\C)$ admits special preenvelope
sequences~\eqref{sp-preenvelope-seq}.
 Assume further the left $R$\+module $A^+$ belongs to~$\C$,
the condition~\textup{($\dagger$)} holds, and the class $\C$ is
coresolving in $R\Modl$.
 Let $M$ be a left $R$\+module of $\C$\+coresolution dimension~$\le l$.
 Then the $\C$\+coresolution dimension of the $R$\+module $A\ot_RM$
also does not exceed~$l$.
\end{lem}

\begin{proof}
 This is the dual version of
Lemma~\ref{coinduction-preserves-resolution-dimension}.
 Let $0\rarrow M\rarrow C^0\rarrow F^1\rarrow0$ be a special
preenvelope sequence~\eqref{sp-preenvelope-seq} for the left
$R$\+module $M$; so $C^0\in\C$ and $F^1\in\F$.
 Consider a special preenvelope sequence $0\rarrow F^1\rarrow C^1
\rarrow F^2\rarrow0$ for the left $R$\+module $F^1$, etc.
 Proceeding in this way, we construct an exact sequence of left
$R$\+modules $0\rarrow M\rarrow C^0\rarrow C^1\rarrow\dotsb\rarrow
C^{l-2}\rarrow C^{l-1}\rarrow F^l\rarrow0$, in which
$C^i\in\C$ for all $0\le i\le l-1$, \ $F^l\in\F$, and
the image $F^i$ of the morphism $C^{i-1}\rarrow C^i$ belongs to $\F$
for all $1\le i\le l-1$.
 Since the $\C$\+coresolution dimension of $M$ does not exceed~$l$
by assumption, by Lemma~\ref{co-resolution-dimension-lemma}(b)
it follows that $F^l\in\C$.
 Since $\Tor^R_1(A,F)=0$ for all $F\in F$
(see Lemma~\ref{necessary-A-plus-in-C}), our sequence remains exact
after applying the functor $A\ot_R{-}$.
 The resulting exact sequence is the desired coresolution of length~$l$
of the left $R$\+module $A\ot_RM$ by modules from~$\C$.
\end{proof}

\begin{prop} \label{finite-coresol-dim-approximation-sequences}
 Assume that the\/ $\Ext^1$\+orthogonal pair of classes of
left $R$\+modules $(\F,\C)$ admits approximation
sequences~\textup{(\ref{sp-precover-seq}\+-\ref{sp-preenvelope-seq})}.
 Assume that the left $R$\+module $A^+$ belongs to~$\C$, and
that the condition~\textup{($\dagger$)} holds.
 Assume further that the class $\C$ is coresolving in $R\Modl$ and
the $\C$\+coresolution dimension of any left $R$\+module does not
exceed a finite integer~$k\ge0$.
 Then the\/ $\Ext^1$\+orthogonal pair of classes of left $A$\+modules\/
$\Fil_{k+1}(A\ot_R\F)$ and $\C^A$ admits approximation sequences
as well.
 Here the integer~$k+1$ is considered as a finite ordinal.
\end{prop}

\begin{proof}
 The pair of classes $\Fil(A\ot_R\F)$ and $\C^A\subset A\Modl$ is
$\Ext^1$\+orthogonal by Lemma~\ref{orthogonal-to-induced-lemma}(c).
 Let us show by explicit construction that the pair of classes
$\Fil_k(A\ot_R\nobreak\F)$ and $\C^A$ admits special preenvelope
sequences.
 The construction below goes back to~\cite[Lemma~3.1.3(b)]{Psemi}.

 Let $N$ be a left $A$\+module.
 Then there is a natural (adjunction) morphism of left $A$\+modules
$\pi_N\:A\ot_RN\rarrow N$ defined by the formula $\pi_N(a\ot n)=an$
for every $a\in A$ and $n\in N$.
 The map~$\pi_N$ is always surjective.
 Moreover, viewed as a morphism of left $R$\+modules, $\pi_N$~is
a split epimorphism.
 Indeed, the map $\epsilon_N\:N\rarrow A\ot_RN$ taking every element
$n\in N$ to the element $\epsilon(n)=1\ot n\in A\ot_RN$ is a left
$R$\+module morphism for which the composition $\pi_N\circ\epsilon_N$
is the identity map, $\pi_N\circ\epsilon_N=\id_N$.

 Consider the underlying left $R$\+module of $N$, and choose
a special preenvelope sequence $0\rarrow N\rarrow C(N)\rarrow F'(N)
\rarrow0$ in $R\Modl$ with $C(N)\in\C$ and $F'(N)\in\F$.
 Then we have $\Tor^R_1(A,F'(N))=0$
(see Lemma~\ref{necessary-A-plus-in-C}), so the morphism of left
$A$\+modules $A\ot_RN\rarrow A\ot_RC(N)$ induced from the injective
left $R$\+module map $N\rarrow C(N)$ is injective.
 Denote by $W(N)$ the pushout (or in other words, the fibered
coproduct) of the pair of left $A$\+module morphisms
$A\ot_RN\rarrow N$ and $A\ot_RN\rarrow A\ot_RC(N)$.

 We have a commutative diagram of left $A$\+module morphisms, in
which the four short sequences are exact:
\begin{equation} \label{preenvelope-construction-diagram}
\begin{gathered}
 \xymatrix{
  & & 0 \ar[d] & 0 \ar[d] \\
  0 \ar[r] & \ker(\pi_N) \ar[r]\ar@{=}[dd]
  & A\ot_RN \ar[r]^-{\pi_N} \ar[dd] & N \ar[r]\ar[dd] & 0 \\ \\
  0 \ar[r] & \ker(\pi_N) \ar[r] & A\ot_RC(N) \ar[r] \ar[dd]
  & W(N) \ar[r] \ar[dd] & 0 \\ \\
  & & A\ot_RF'(N) \ar@{=}[r] \ar[d] & A\ot_R F'(N) \ar[d] \\
  & & 0 & 0
 }
\end{gathered}
\end{equation}

 Introduce the notation $\cd_\C M$ for the $\C$\+coresolution dimension
of a left $R$\+mod\-ule~$M$.
 We will apply the same notation to left $A$\+modules, presuming that
the $\C$\+coresolution dimension of the underlying $R$\+module is taken.

 Next we observe that, whenever $0<\cd_\C N<\infty$,
the $\C$\+coresolution dimension of the underlying left $R$\+module
of the left $A$\+module $W(N)$ is strictly smaller than
the $\C$\+coresolution dimension of the underlying $R$\+module of
the $A$\+module $N$, i.~e., $\cd_\C W(N)<\cd_\C(N)$.
 Indeed, the short exact sequence of left $A$\+modules $0\rarrow
\ker(\pi_N)\rarrow A\ot_RN\rarrow N\rarrow0$ splits over $R$, or
in other words, the underlying left $R$\+module of $\ker(\pi_N)$ can be
presented as the cokernel of the injective left $R$\+module morphism
$\epsilon_N\:N\rarrow A\ot_RN$.
 By Lemmas~\ref{induction-preserves-coresolution-dimension}
and~\ref{co-resolution-dimension-defined-classes}(b), we have
$\cd_\C\ker(\pi_N)\le\cd_\C N$.
 Since $A\ot_RC(N)\in\C$, it follows from the short exact sequence
$0\rarrow\ker(\pi_N)\rarrow A\ot_RC(N)\rarrow W(N)\rarrow0$ that
$\cd_\C W(N)<\cd_\C N$.

 It remains to iterate our construction, producing a sequence of
injective morphisms of left $A$\+modules
$$
 N\lrarrow W(N)\lrarrow W(W(N))\lrarrow W^3(N)\lrarrow\dotsb
 \lrarrow W^k(N).
$$
Since $\cd_\C(N)\le k$ by assumption, it follows from the above
argument that $\cd_\C W^k(N)\le 0$, that is $W^k(N)\in\C$.

 The cokernel of the injective morphism $N\rarrow W^k(N)$ is filtered
by the cokernels of the injective $A$\+module morphisms
$N\rarrow W(N)$, \ $W(N)\rarrow W^2(N)$, \dots,
$W^{k-1}(N)\rarrow W^k(N)$.
 These are the left $A$\+modules $A\ot_RF'(N)$, \ $A\ot_RF'(W(N))$, \
$A\ot_RF'(W^2(N))$, \dots, $A\ot_RF'(W^{k-1}(N))$.
 We have constructed the desired special preenvelope sequence for
the pair of classes $\Fil_k(A\ot_R\F)$ and~$\C^A$.

 Finally, any left $R$\+module $M$ is a quotient module of
an $R$\+module $F(M)\in\F$, since a special precover sequence with
respect to $(\F,\C)$ exists for $M$ by assumption.
 If $M$ is a left $A$\+module, then the map~$\pi_M$ presents $M$ as
a quotient module of the left $A$\+module $A\ot_RM$, which is
a quotient module of the left $A$\+module $A\ot_RF(M)$.
 Thus $M$ is a quotient $A$\+module of $A\ot_RF(M)$.
 Following the proof of (the ``if'' implication in)
Lemma~\ref{salce-lemma}, we conclude that the pair of classes
$\Fil_{k+1}(A\ot_R\F)$ and $\C^A$ admits special precover sequences.
\end{proof}

\begin{thm} \label{finite-coresol-dim-theorem}
 Let $(\F,\C)$ be a hereditary complete cotorsion pair in $R\Modl$.
 Assume that the left $R$\+module $A^+$ belongs to~$\C$, and
that the condition~\textup{($\dagger$)} holds.
 Assume further that the $\C$\+coresolution dimension of any left
$R$\+module does not exceed a finite integer~$k\ge0$.
 Then the pair of classes $\F^A=\Fil_{k+1}(A\ot_R\F)^\oplus$ and
$\C^A$ is a hereditary complete cotorsion pair in $A\Modl$.
\end{thm}

\begin{proof}
 The class $\C^A$ is closed under direct summands and the cokernels
of injective morphisms, since the class $\C$~is.
 Thus the assertion of the theorem follows from
Proposition~\ref{finite-coresol-dim-approximation-sequences}
in view of Lemma~\ref{direct-summand-lemma}.
\end{proof}

\begin{cor} \label{finite-coresol-dim-cor}
 For any associative ring homomorphism $R\rarrow A$ and any
hereditary complete cotorsion pair $(\F,\C)$ in $R\Modl$ satisfying
the assumptions of Theorem~\ref{finite-coresol-dim-theorem}, one has
${}^{\perp_1}\C^A=\Fil_{k+1}(A\ot_R\F)^\oplus$.
 In particular, it follows that\/ $\Fil(A\ot_R\F)^\oplus=
\Fil_{k+1}(A\ot_R\F)^\oplus$.
\end{cor}

\begin{proof}
 The first assertion is a part of
Theorem~\ref{finite-coresol-dim-theorem}.
 The second assertion follows from the first one together with
Lemma~\ref{orthogonal-to-induced-lemma}(c).
\end{proof}

\begin{rem} \label{corings-and-contramodules-remark}
 The condition~($\dagger$) appears to be rather restrictive.
 In fact, the construction of
Proposition~\ref{finite-coresol-dim-approximation-sequences}
originates from the theory of contramodules over corings,
as in~\cite[Lemma~3.1.3(b)]{Psemi}, where the natural analogue of
this condition feels much less restrictive, particularly when $\C$
is simply the class of all injective left $R$\+modules.
 So one can say that the ring $A$ in this
Section~\ref{finite-by-induced-subsecn} really ``wants'' to be
a coring $C$ over $R$, and the left $A$\+modules ``want'' to be
left $C$\+contramodules.
 Then the induction functor, which was the tensor product $A\ot_R{-}$
in the condition~($\dagger$), takes the form of the Hom functor
$\Hom_R(C,{-})$.
 This one is much more likely to take injective left $R$\+modules to
injective left $R$\+modules (it suffices that $C$ be a flat right
$R$\+module).
 To make a ring $A$ behave rather like a coring, one can assume it
to be ``small'' relative to $R$ in some sense.
 The following example is inspired by the analogy with corings and
contramodules.
\end{rem}

\begin{ex} \label{corings-and-contramodules-example}
 Let $\C=R\Modl_\inj$ be the class of all injective left $R$\+modules.
 Then $\F=R\Modl$ is the class of all left $R$\+modules, and
$\C^A=A\Modl_{R\dinj}$ is the class of all left $A$\+modules whose
underlying $R$\+modules are injective.
 In the terminology of~\cite[Sections~4.1 and~4.3]{BP}
and~\cite[Section~5]{Pfp}, the left $A$\+modules from the related
class $\F^A={}^{\perp_1}\C^A$ would be called \emph{weakly
projective relative to~$R$} or \emph{weakly $A/R$\+projective}.

 For $\C=R\Modl_\inj$, the necessary condition of
Lemma~\ref{necessary-A-plus-in-C} says that $A^+$ must be an injective
left $R$\+module; equivalently, this means that $A$ is a flat right
$R$\+module.
 Assume that $A$ is a finitely generated projective right $R$\+module;
then the functor $A\ot_R{-}$ preserves infinite products.
 Assume further that there exists an injective cogenerator $I$ of
the category of left $R$\+modules such that the left $R$\+module
$A\ot_RI$ is injective.
 Under the above assumption, this is equivalent to the condition that
the right $R$\+module $\Hom_{R^{\mathrm{op}}}(A,R)$ is flat.
 Then it follows that the functor $A\ot_R{-}$ preserves the class
$\C$ of all injective left $R$\+modules.
 Thus the condition~($\dagger$) is satisfied.

 The results of Section~\ref{finite-by-induced-subsecn} tell us that,
whenever the left homological dimension of the ring $R$ is a finite
number~$k$ and the assumptions in the previous paragraph hold,
the $\Ext^1$\+orthogonal pair of classes of left $A$\+modules
$\Fil_{k+1}(A\ot_RR\Modl)$ and $A\Modl_{R\dinj}$ admits
approximation sequences.
 Consequently, the pair of classes
$\F^A=\Fil_{k+1}(A\ot_RR\Modl)^\oplus$ and $\C^A=A\Modl_{R\dinj}$
is a hereditary complete cotorsion pair in $A\Modl$.
 In particular, we have
$$
 {}^{\perp_1} A\Modl_{R\dinj}=\Fil_{k+1}(A\ot_RR\Modl)^\oplus
$$
and $\Fil(A\ot_RR\Modl)^\oplus=\Fil_{k+1}(A\ot_RR\Modl)^\oplus$.
 So the weakly $A/R$\+projective left $A$\+modules are precisely
the direct summands of the $A$\+modules admitting a finite
$(k+1)$\+step filtration by $A$\+modules induced from left $R$\+modules.

 The reader can find a discussion of the related results for corings
and contramodules (of which this example is a particular case)
in~\cite[Lemma~3.11(b)]{Prev}.
\end{ex}

 For a class of examples to Theorem~\ref{finite-coresol-dim-theorem}
arising in connection with $n$\+tilting modules, see
Example~\ref{tilting-examples}(1) below.
 For a class of examples to the same theorem arising from curved
DG\+rings, see Proposition~\ref{finite-homol-dim-coacyclic}.

 One problem with the condition~($\dagger$) is that it mentions
the underived tensor product $A\ot_RC$.
 The groups $\Tor^R_i(A,C)$ with $i>0$ are a potential source of
problems, but they are ignored in the formulation of the condition.
 Yet there is no reason to expect these $\Tor$ groups to vanish for
all modules $C\in\C$.

 Therefore, one may want to restrict~($\dagger$) to some subclass of
the class $\C$, consisting of modules for which the functor
$A\ot_R{-}$ is better behaved.
 One can do so by considering the following condition:
\begin{itemize}
\item[($\widetilde\dagger$)] there exists a resolving class
$\G\subset R\Modl$ such that $\F\subset\G$, the underlying left
$R$\+modules of all the left $A$\+modules from $\F^A={}^{\perp_1}\C^A$
belong to~$\G$, and the left $R$\+module $A\ot_RC$ belongs to $\C$
for every left $R$\+module $C\in\C\cap\G$.
\end{itemize}

 Taking $\G=R\Modl$, one recovers~($\dagger$) as a particular case
of~($\widetilde\dagger$).

\begin{thm} \label{finite-coresol-dim-tilde}
 Let $(\F,\C)$ be a hereditary complete cotorsion pair in $R\Modl$.
 Assume that the left $R$\+module $A^+$ belongs to~$\C$, and
that the condition~\textup{($\widetilde\dagger$)} holds.
 Assume further that the $\C$\+coresolution dimension of any left
$R$\+module does not exceed a finite integer~$k\ge0$.
 Then the class $\F^A={}^{\perp_1}\C^A\subset A\Modl$ can be described
as $\F^A=\Fil_{k+1}(A\ot_R\F)^\oplus$.
 In particular, we have\/ $\Fil(A\ot_R\F)^\oplus=
\Fil_{k+1}(A\ot_R\F)^\oplus$.
\end{thm}

\begin{proof}
 We are following the proof of Corollary~\ref{finite-coresol-dim-cor}
step by step and observing that the assumptions of the present theorem
are sufficient for the validity of the argument.
 Essentially, the point is that the key constructions are performed
within the class $\G\subset R\Modl$ and the class of all
left $A$\+modules whose underlying left $R$\+modules belong to~$\G$.

 The inclusion $\Fil(A\ot_R\F)^\oplus\subset\F^A$ holds by
Lemma~\ref{orthogonal-to-induced-lemma}(c).
 Given a left $A$\+module $M\in\F^A$, we will show that
$M\in\Fil_{k+1}(A\ot_R\F)^\oplus$.

 Arguing as in the last paragraph of the proof of
Proposition~\ref{finite-coresol-dim-approximation-sequences},
the left $R$\+module $M$ is a quotient module of an $R$\+module
$F(M)\in\F$, and therefore the left $A$\+module $M$ is a quotient
$A$\+module of the left $A$\+module $A\ot_RF(M)$.
 Denote by~$N$ the kernel of the surjective $A$\+module morphism
$A\ot_RF(M)\rarrow M$.
 By~($\widetilde\dagger$), we have ${}_RM\in\G$ and $A\ot_RF(M)\in\G$,
hence the underlying left $R$\+module of $N$ also belongs to~$\G$.

 Now we construct the diagram~\eqref{preenvelope-construction-diagram}
for the left $A$\+module~$N$.
 In the special preenvelope sequence $0\rarrow N\rarrow C(N)\rarrow
F'(N)\rarrow0$, we have $F'(N)\in\F\subset\G$ and ${}_RN\in\G$,
hence $C(N)\in\G$.
 According to~($\widetilde\dagger$), it follows that $A\ot_RC(N)\in\C$.
 Also by~($\widetilde\dagger$), we have $A\ot_RF'(M)\in A\ot_R\F
\subset\G$, so it follows from the short exact sequence
$0\rarrow N\rarrow W(N)\rarrow A\ot_RF'(N)\rarrow0$ that
the underlying left $R$\+module of the left $A$\+module $W(N)$
belongs to~$\G$.

 Iterating the construction and following the proof of
Proposition~\ref{finite-coresol-dim-approximation-sequences},
we obtain an injective morphism of left $A$\+modules $N\rarrow W^k(N)$
with $W^k(N)\in\C^A$ and the cokernel belonging to $\Fil_k(A\ot_R\F)$.
 Following the proof of (the ``if'' implication in)
Lemma~\ref{salce-lemma}, we produce a surjective $A$\+module morphism
onto $M$ from an $A$\+module belonging to $\Fil_{k+1}(A\ot_R\F)$ with
the kernel isomorphic to~$W^k(N)$.
 As $\Ext^1_A(M,W^k(N))=0$ by assumption, we can conclude that
$M\in\Fil_{k+1}(A\ot_R\F)^\oplus$.
\end{proof}

 For a class of examples to Theorem~\ref{finite-coresol-dim-tilde}
arising in connection with $n$\+tilting modules,
see Example~\ref{tilting-examples}(2).

\subsection{Tilting cotorsion pairs and Bongartz--Ringel lemma}
\label{tilting-subsecn}
 In this section we discuss an important class of examples in which
a suitable version of the Bongartz--Ringel lemma~\cite[Lemma~2.1]{Bon},
\cite[Lemma~4$'$]{Rin}, \cite[Lemma~6.15]{GT} leads to a better result
than the techniques of Section~\ref{finite-by-induced-subsecn}.

 Let $T$ be a left $R$\+module and $n\ge0$ be an integer.
 The $R$\+module $T$ is said to be \emph{$n$\+tilting}
\cite[Section~2]{AC}, \cite[Definition~13.1]{GT} if the following
three conditions hold:
\begin{enumerate}
\renewcommand{\theenumi}{T\arabic{enumi}}
\item the projective dimension of the left $R$\+module $T$ does not
exceed~$n$;
\item $\Ext_R^i(T,T^{(\kappa)})=0$ for all integers $i>0$ and all
cardinals~$\kappa$;
\item there exists an exact sequence of left $R$\+modules
$0\rarrow R\rarrow T^0\rarrow T^1\rarrow\dotsb\rarrow T^{n-1}
\rarrow T^n\rarrow0$ with $T^i\in\Add_R(T)$ for all $0\le i\le n$.
\end{enumerate}

 The \emph{$n$\+tilting class} induced by $T$ in $R\Modl$ is
the class of left $R$\+modules $\C=T^{\perp_{>0}}=
\{C\in R\Modl\mid\Ext^i_R(T,C)=0\ \forall i>0\}$.
 The cotorsion pair $(\F,\C)$ with $\F={}^{\perp_1}\C\subset R\Modl$
is hereditary and complete by Theorem~\ref{eklof-trlifaj-theorem}
(see~\cite[Theorem~3.1]{AC} for details); it is called
the \emph{$n$\+tilting cotorsion pair} induced by $T$ in $R\Modl$.

\begin{prop} \label{when-n-tilting}
 Let $R\rarrow A$ be a homomorphism of associative rings and $T$ be
an $n$\+tilting left $R$\+module.
 Assume that the underlying left $R$\+module of
$A^+=\Hom_{\Z}(A,\Q/\Z)$ belongs to $\C$, that is, in other words,
$\Tor^R_i(A,T)=0$ for all $i>0$.
 Then \par
\textup{(a)} the left $A$\+module $A\ot_RT$ satisfies
the conditions~\textup{(T1)} and~\textup{(T3)}; \par
\textup{(b)} the left $A$\+module $A\ot_RT$ satisfies~\textup{(T2)}
if and only if its underlying left $R$\+module belongs to~$\C$.
\end{prop}

\begin{proof}
 Part~(a): applying the functor $A\ot_R{-}$ to a projective resolution
$0\rarrow P_n\rarrow\dotsb\rarrow P_0\rarrow T\rarrow0$ of the left
$R$\+module $T$ produces a projective resolution $0\rarrow A\ot_RP_n
\rarrow\dotsb\rarrow A\ot_R P_0\rarrow A\ot_RT\rarrow0$ of the left
$A$\+module $A\ot_RT$.
 Similarly, applying the functor $A\ot_R{-}$ to an exact sequence
in~(T3) produces an exact sequence of left $A$\+modules $0\rarrow A
\rarrow A\ot_RT^0\rarrow\dotsb\rarrow A\ot_RT^n\rarrow0$, in which
$A\ot_RT^i\in\Add_A(A\ot_RT)$ for all $0\le i\le n$.

 Part~(b): put $T'=A\ot_RT$.
 By Lemma~\ref{homological-formulas-lemma}(a), we have
$\Ext^i_A(T',T'{}^{(\kappa)})\simeq\Ext_R^i(T,T'^{(\kappa)})$ for all
$i\ge0$, since $\Tor^R_i(A,T)=0$ for $i>0$.
 It follows that the left $A$\+module $T'$ is $n$\+tilting if and only
if the left $R$\+module $T'{}^{(\kappa)}$ belongs to $\C\subset R\Modl$
for every cardinal~$\kappa$.
 Since the $n$\+tilting class $\C$ is closed under infinite direct
sums in $R\Modl$ \cite[Proposition~13.13(b)]{GT}, it suffices that
${}_RT'\in\C$.
\end{proof}

\begin{lem} \label{tilting-class-change-of-ring}
 Let $R\rarrow A$ be a homomorphism of associative rings and $T$ be
an $n$\+tilting left $R$\+module.
 Let $(\F,\C)$ be the $n$\+tilting cotorsion pair induced by $T$ in
$R\Modl$.
 Assume that the underlying left $R$\+module of $A^+$ belongs to $\C$,
that is ${}_RA^+\in\C$.
 Assume further that the left $A$\+module $A\ot_RT$ is $n$\+tilting.
 Then the $n$\+tilting cotorsion pair induced by $A\ot_RT$ in $A\Modl$
has the form $(\F^A,\C^A)$ in our notation.
 In other words, the $n$\+tilting class induced by $A\ot_RT$ in $A\Modl$
consists precisely of all the left $A$\+modules whose underlying
left $R$\+modules belong to the $n$\+tilting class $\C$ induced by $T$
in $R\Modl$.
\end{lem}

\begin{proof}
 Indeed, for any left $A$\+module $C$ we have $\Ext^i_A(A\ot_RT,\>C)
\simeq\Ext^i_R(T,C)$ for all $i\ge0$ by
Lemma~\ref{homological-formulas-lemma}(a), since
$\Tor^R_i(A,T)=0$ for all $i>0$.
\end{proof}

\begin{prop} \label{dual-bazzoni-prop}
 Let $R$ be a commutative ring and $A$ be an associative $R$\+algebra.
 Let $T$ be an $n$\+tilting $R$\+module and $(\F,\C)$ be
the $n$\+tilting cotorsion pair induced by $T$ in $R\Modl$.
 Assume that the underlying $R$\+module of $A^+$ belongs to~$\C$.
 Then the left $A$\+module $A\ot_RT$ is $n$\+tilting.
\end{prop}

\begin{proof}
 According to Proposition~\ref{when-n-tilting}, it suffices to show that
the $R$\+module $T'=A\ot_RT$ belongs to~$\C$.
  We use the dual argument to the one in
Proposition~\ref{bazzoni-prop}.
 By~\cite[Lemma~3.2]{Baz} or~\cite[Proposition~13.13(b)]{GT},
the tilting class $\C$ can be described as the class of all $R$\+modules
admitting a resolution by direct sums of copies of~$T$.
 Let $\dotsb\rarrow P_2\rarrow P_1\rarrow P_0\rarrow A\rarrow0$ be
a free resolution of the $R$\+module~$A$.
 Then $\dotsb\rarrow P_2\ot_R T\rarrow P_1\ot_RT\rarrow P_0\ot_RT
\rarrow A\ot_RT\rarrow0$ is a resolution of the $R$\+module $A\ot_RT$
by direct sums of copies of~$T$.
 Thus ${}_RT'\in\C$, as desired.
\end{proof}

 A discussion of the particular case of the above proposition and lemma
in which $A$ is a flat commutative $R$\+algebra can be found
in~\cite[Proposition~2.3 and Lemma~2.4]{HST}.

\begin{exs} \label{tilting-examples}
 Let $A$ be an associative algebra over a commutative ring $R$, and let
$T$ be an $n$\+tilting $R$\+module.
 Let $(\F,\C)$ be the $n$\+tilting cotorsion pair induced by $T$ in
$R\Modl$.
 Assume that the $R$\+module $A^+$ belongs to~$\C$.
 Then $A\ot_RT$ is an $n$\+tilting left $A$\+module by
Proposition~\ref{dual-bazzoni-prop}, and the induced $n$\+tilting
cotorsion pair in $A\Modl$ has the form $(\F^A,\C^A)$ in our notation
by Lemma~\ref{tilting-class-change-of-ring}.

 (1)~In the following particular cases
Theorem~\ref{finite-coresol-dim-theorem} is applicable.
 Assume that either $A$ is a flat $R$\+module, or $n\le 2$.
 Then the condition~($\dagger$) holds.

 Indeed, when $A$ is a flat $R$\+module, it suffices to observe
that the $n$\+tilting class $\C$ is closed under
direct limits~\cite[Corollary~13.42]{GT}.
 When $n\le2$, consider an $R$\+module $C\in\C$.
 Choose a projective presentation $P_1\rarrow P_0\rarrow A\rarrow0$ for
the $R$\+module~$A$.
 Then we have a right exact sequence of $R$\+modules $P_1\ot_RC\rarrow
P_0\ot_RC\rarrow A\ot_RC\rarrow0$ with $P_i\ot_RC\in\C$ for $i=0$,~$1$.
 Denoting by $K$ the kernel of the morphism $P_1\ot_RC\rarrow
P_0\ot_RC$, we have $\Ext_R^i(T,\>A\ot_RC)=\Ext_R^{i+2}(T,K)=0$
for all $i>0$, as desired.

 Finally, the $\C$\+coresolution dimension of any $R$\+module does not
exceed~$n$ (since the projective dimension of the  $R$\+module $T$
is~$\le n$).
 According to Corollary~\ref{finite-coresol-dim-cor}, we have
${}^{\perp_1}\C^A=\F^A=\Fil_{n+1}(A\ot_R\F)^\oplus$.

 (2)~This is a generalization of~(1) that can be obtained using
Theorem~\ref{finite-coresol-dim-tilde}.
 We are assuming that $A$ is an associative $R$\+algebra, $T$ is
an $n$\+tilting $R$\+module, and ${}_RA^+\in\C$.
 Assume further that $\Tor^R_i(A,\>A\ot_RT)=0$ for all $i>0$.
 Then we claim that~($\widetilde\dagger$) is satisfied.

 Let $\G$ be the class of all $R$\+modules $G$ such that
$\Tor^R_i(A,G)=0$ for all $i>0$.
 Then we have $\F\subset\G$, since ${}_RA^+\in\C$.
 Furthermore, all the left $A$\+modules in $\F^A$ have finite
coresolutions by direct summands of direct sums of copies of $A\ot_RT$
\cite[Proposition~13.13(a)]{GT}; hence $\Tor^R_i(A,F)=0$ for all
$F\in\F^A$ and $i>0$.

 In order to check the condition~($\widetilde\dagger$), it remains to
show that $A\ot_RC\in\C$ for any $R$\+module $C\in\C\cap\G$.
 Indeed, let us choose a projective resolution $\dotsb\rarrow P_2
\rarrow P_1\rarrow P_0\rarrow A\rarrow0$ for the $R$\+module~$A$.
 Then we have an exact sequence of $R$\+modules $\dotsb\rarrow
P_2\ot_RC\rarrow P_1\ot_RC\rarrow P_0\ot_RC\rarrow A\ot_RC\rarrow0$
with $P_i\ot_RC\in\C$ for all $i\ge0$.
 Denoting by $K$ the image of the morphism $P_n\ot_RC\rarrow P_{n-1}
\ot_RC$, we have $\Ext_R^i(T,\>A\ot_RC)=\Ext_R^{i+n}(T,K)=0$
for all $i>0$, as desired.

 By Theorem~\ref{finite-coresol-dim-tilde}, we can infer that
${}^{\perp_1}\C^A=\F^A=\Fil_{n+1}(A\ot_R\F)^\oplus$.
\end{exs}

 Now we will explain how a stronger and more general version of
the results of Examples~\ref{tilting-examples} can be obtained with
an approach based on a suitable version of the Bongartz--Ringel lemma.

\begin{thm}[Bongartz--Ringel lemma] \label{n-bongartz}
 Let $A$ be an associative ring, $n\ge0$ be an integer, and
$\S=\{S_0,S_1,\dotsc,S_n\}$ be a collection of $n+1$ left $A$\+modules.
 Assume that $S_0$ is a projective generator of $A\Modl$ and\/
$\Ext_A^1(S_i,S_j^{(\kappa)})=0$ for all\/ $0\le i\le j\le n$ and
all cardinals~$\kappa$.
 Let $(\F,\C)=({}^{\perp_1}(\S^{\perp_1}),\S^{\perp_1})$ be
the (complete) cotorsion pair generated by the set $\S$ in $A\Modl$.
 Then the class $\F\subset A\Modl$ can be described as the class of
all direct summands of $(n+1)$\+filtered left $A$\+modules $G$ with
a filtration $0=F_0G\subset F_1G\subset\dotsb\subset F_nG\subset
F_{n+1}G=G$ such that $F_{i+1}G/F_iG\in\Add(S_i)$
for every\/ $0\le i\le n$.
 In particular, we have
$\F=\Fil_{n+1}(\bigcup_{i=0}^n\Add(S_i))^\oplus$.
\end{thm}

\begin{proof}
 This is an $n\ge1$ generalization of the classical Bongartz
lemma~\cite[Lemma~6.15]{GT}, which corresponds to the case $n=1$.
 At the same time, this is an infinitely generated version of
Ringel's~\cite[Lemma~4$'$]{Rin}.
 The argument is dual to the proof of Theorem~\ref{dual-n-bongartz}.
\end{proof}

\begin{rem}
 For any $n$\+tilting $R$\+module $T$, the induced $n$\+tilting
cotorsion pair $(\F,\C)$ is obviously generated by the syzygy modules
$T$, $\Omega^1T$,~\dots, $\Omega^{n-1}T$ of the $R$\+module~$T$.
 However, the sequence of syzygy modules $S_n=T$, \,$S_{n-1}=
\Omega^1T$,~\dots, $S_1=\Omega^{n-1}T$ does \emph{not} satisfy
the assumptions of Theorem~\ref{n-bongartz} (generally speaking),
as one can see from a straightforward $k$\+vector space dualization of
the counterexample of D'Este~\cite[Theorem~5]{Est} mentioned in
Remark~\ref{deste-counterex-remark}.
 Hence the need for a more sophisticated approach based on the next 
lemma.
\end{rem}

\begin{lem} \label{derived-tilting-classes}
 Let $R$ be an associative ring and $T$ be an $n$\+tilting left
$R$\+module.
 Then, for every\/ $0\le j\le n$, there exists an $(n-j)$\+tilting
left $R$\+module $T_j$ such that the class $\C_j=T^{\perp_{>j}}$
is the tilting class induced by $T_j$ in $R\Modl$.
 In particular, one can (and we will) take $T_0=T$, while $P=T_n$ is
a projective generator of $R\Modl$.
\end{lem}

\begin{proof}
 This is~\cite[Lemma~3.5]{BS} or~\cite[Remark~15.14]{GT}.
\end{proof}

\begin{lem} \label{generated-by-tilting-modules}
 In the notation of Lemma~\ref{derived-tilting-classes},
the $n$\+tilting class $\C=T^{\perp_{>0}}$ can be described as
$\C=\{T_0,T_1,\dotsc,T_n\}^{\perp_1}=\{C\in R\Modl\mid
\Ext_R^1(T_j,C)=0\ \forall\,0\le j\le n\}$.
\end{lem}

\begin{proof}
 Dual to the proof of Lemma~\ref{cogenerated-by-cotilting-modules}.
\end{proof}

\begin{thm} \label{left-tilting-class-described}
 Let $R$ be an associative ring and $T$ be an $n$\+tilting left
$R$\+module.
 Let $(\F,\C)$ be the $n$\+tilting cotorsion pair induced by $T$ in
$R\Modl$.
 Then the class $\F$ can be described as the class of all direct
summands of $(n+1)$\+filtered left $R$\+modules $G$ with a filtration
$0=F_0G\subset F_1G\subset\dotsb\subset F_nG\subset F_{n+1}G=G$ such
that, in the notation of Lemma~\ref{derived-tilting-classes},
\,$F_1G\in\Add(P)$, \ $F_{i+1}G/F_iG\in\Add(T_{n-i})$ for every\/
$0\le i\le n$, and $F_{n+1}G/F_nG\in\Add(T)$.
\end{thm}

\begin{proof}
 By Lemma~\ref{generated-by-tilting-modules}, the $n$\+tilting
cotorsion pair $(\F,\C)$ is generated by the set of $n+1$ modules
$S_0=P$, $S_1=T_{n-1}$, $S_2=T_{n-2}$,~\dots, $S_n=T$.
 Furthermore, one has $\Ext^m_R(T_j,T_i^{(\kappa)})=0$ for all
integers $0\le i\le j\le n$, \,$m>0$ and all cardinals~$\kappa$,
since $T_i^{(\kappa)}\in T_i^{\perp_{>0}}\subset T_j^{\perp_{>0}}$.
 Thus Theorem~\ref{n-bongartz} is applicable.
\end{proof}

 The result of the following corollary generalizes those of
Examples~\ref{tilting-examples}.

\begin{cor} \label{finite-filtration-by-induced-via-n-bongartz}
 Let $R\rarrow A$ be a homomorphism of associative rings and $T$ be
an $n$\+tilting left $R$\+module.
 Let $(\F,\C)$ be the $n$\+tilting cotorsion pair induced by $T$ in
$R\Modl$.
 Assume that the underlying left $R$\+module of $A^+$ belongs to $\C$,
that is\/ $\Tor^R_i(A,T)=0$ for all $i>0$.
 Assume further that the left $A$\+module $A\ot_RT_j$ is
$(n-j)$-tilting for every\/ $0\le j\le n$.
 (In particular, by Proposition~\ref{dual-bazzoni-prop}, this holds
whenever $R$ is commutative and $A$ is an $R$\+algebra.)
 Let $(\F^A,\C^A)$ be the $n$\+tilting cotorsion pair induced by
$T'=A\ot_RT$ in $A\Modl$.
 Then we have $\F^A=\Fil_{n+1}(A\ot_R\F)^\oplus$.
\end{cor}

\begin{proof}
 Follows easily from Theorem~\ref{left-tilting-class-described}
and Lemma~\ref{homological-formulas-lemma}(a)
(cf.\ the proof of the dual assertion in
Corollary~\ref{finite-filtration-by-coinduced-via-dual-n-bongartz}).
\end{proof}

\subsection{Increasing filtrations by induced modules}
\label{increasing-filtrations-subsecn}
 Let $R\rarrow A$ be a homomorphism of associative rings, and
let $(\F,\C)$ be an $\Ext^1$\+orthogonal pair of classes of
left $R$\+modules.
 Instead of assuming finiteness of the $\C$\+coresolution dimension,
we now assume that the class $\C$ is closed under countable direct
sums in $R\Modl$.

 As above, we denote by~$\omega$ the ordinal of nonnegative integers.
 The ``filtrations'' appearing in the next proposition are the usual
exhaustive infinite increasing filtrations indexed by
the natural numbers.

\begin{prop} \label{direct-sum-closed-approximation-sequences}
 Assume that the\/ $\Ext^1$\+orthogonal pair of classes of
left $R$\+modules $(\F,\C)$ admits approximation
sequences~\textup{(\ref{sp-precover-seq}\+-\ref{sp-preenvelope-seq})}.
 Assume that the left $R$\+module $A^+$ belongs to~$\C$, and
that the condition~\textup{($\dagger$)} holds.
 Assume further that the class $\C$ is closed under the cokernels of
injective morphisms and countable direct sums in $R\Modl$.
 Then the\/ $\Ext^1$\+orthogonal pair of classes of left $A$\+modules\/
$\Fil_\omega(A\ot_R\F)$ and $\C^A$ admits approximation sequences
as well.
\end{prop}

\begin{proof}
 The pair of classes $\Fil(A\ot_R\F)$ and $\C^A\subset A\Modl$
is $\Ext^1$\+orthogonal by Lemma~\ref{orthogonal-to-induced-lemma}(c).
 The explicit construction below, showing that the pair of classes
$\Fil_\omega(A\ot_R\F)$ and $\C^A\subset A\Modl$ admits special
preenvelope sequences, plays a key role.
 It goes back to~\cite[Lemma~1.3.3]{Psemi}.
 
 Let $N$ be a left $A$\+module.
 We proceed with the construction from the proof of
Proposition~\ref{finite-coresol-dim-approximation-sequences}, but
instead of a finite number of~$k$ iterations, we perform
$\omega$~iterations now.
 So we produce a sequence of injective morphisms of left $A$\+modules
\begin{equation} \label{W-inductive-system}
 N\lrarrow W(N)\lrarrow W(W(N))\lrarrow\dotsb\lrarrow W^m(N)
 \lrarrow\dotsb,
\end{equation}
where $m$~ranges over the nonnegative integers.
 Clearly, the cokernel of the injective left $A$\+module morphism
$N\rarrow\varinjlim_{m\in\omega}W^m(N)$ is $\omega$\+filtered by
the left $A$\+modules $A\ot_RF'(W^m(N))$, \,$m\in\omega$, which
belong to $A\ot_R\F$ by construction.
 Now the claim is that the left $A$\+module $\varinjlim_{m\in\omega}
W^m(N)$ belongs to~$\C^A$.

 Recall that the surjective $A$\+module morphism $\pi_N\:A\ot_RN
\rarrow N$ admits a natural $R$\+linear section
$\epsilon_N\:N\rarrow A\ot_RN$.
 Looking on the diagram~\eqref{preenvelope-construction-diagram},
one can see that the injective map $N\rarrow W(N)$ factorizes as
$N\rarrow A\ot_RC(N)\rarrow W(N)$.
 Here $A\ot_RC(N)\rarrow W(N)$ is an $A$\+module morphism, but
$N\rarrow A\ot_RC(N)$ is only an $R$\+module morphism (between
$A$\+modules).
 Thus the sequence of injective morphisms of left
$A$\+modules~\eqref{W-inductive-system} is mutually cofinal with
a sequence of left $R$\+module morphisms
\begin{equation} \label{C-W-inductive-system}
 A\ot_RC(N)\lrarrow A\ot_RC(W(N))\lrarrow\dotsb\lrarrow
 A\ot_RC(W^m(N))\lrarrow\dotsb
\end{equation}

 We have a short exact sequence of left $R$\+modules
\begin{multline} \label{indlimit-telescope}
 0\lrarrow\bigoplus\nolimits_{m\in\omega} A\ot_RC(W^m(N))\lrarrow
 \bigoplus\nolimits_{m\in\omega} A\ot_RC(W^m(N)) \\ \lrarrow
 \varinjlim\nolimits_{m\in\omega} A\ot_RC(W^m(N))\lrarrow0.
\end{multline}
 The left $R$\+modules $C(W^m(N))$, \,$m\ge0$, belong to $\C$ by
construction.
 According to~($\dagger$), it follows that the underlying left
$R$\+modules of the left $A$\+modules $A\ot_RC(W^m(N))$ belong
to~$\C$, too.
 Since the class $\C\subset R\Modl$ is closed under countable direct
sums and the cokernels of injective morphisms by assumption,
it follows that the left $R$\+module
$\varinjlim_{m\in\omega} A\ot_RC(W^m(N))$ belongs to~$\C$.

 The inductive limits of mutually cofinal inductive systems agree,
so we have an isomorphism of left $R$\+modules
$$
 \varinjlim\nolimits_{m\in\omega}W^m(N)\simeq
 \varinjlim\nolimits_{m\in\omega}A\ot_RC(W^m(N)).
$$
 Since $\varinjlim_{m\in\omega}A\ot_RC(W^m(N))\in\C$, we can conclude
that $\varinjlim_{m\in\omega}W^m(N)\in\C^A$, as desired.
 This finishes the construction of the special preenvelope sequences
for the pair of classes of left $A$\+modules $\Fil_\omega(A\ot_R\F)$
and~$\C^A$.

 At last, the special precover sequences for the pair of classes
$\Fil_\omega(A\ot_R\F)$ and $\C^A\subset A\Modl$ are produced from
the special preenvelope sequences in the same way as in the last
paragraph of the proof of
Proposition~\ref{finite-coresol-dim-approximation-sequences}.
\end{proof}

\begin{thm} \label{direct-sum-closed-theorem}
 Let $(\F,\C)$ be a hereditary complete cotorsion pair in $R\Modl$.
 Assume that the left $R$\+module $A^+$ belongs to~$\C$, and
that the condition~\textup{($\dagger$)} holds.
 Assume further that the class $\C$ is closed under countable direct
sums in $R\Modl$.
 Then the pair of classes $\F^A=\Fil_\omega(A\ot_R\F)^\oplus$ and
$\C^A$ is a hereditary complete cotorsion pair in $A\Modl$.
\end{thm}

\begin{proof}
 Follows from
Proposition~\ref{direct-sum-closed-approximation-sequences}
in view of Lemma~\ref{direct-summand-lemma} (cf.\ the proof of
Theorem~\ref{finite-coresol-dim-theorem}).
\end{proof}

\begin{cor} \label{direct-sum-closed-cor}
 For any associative ring homomorphism $R\rarrow A$ and any
hereditary complete cotorsion pair $(\F,\C)$ in $R\Modl$ satisfying
the assumptions of Theorem~\ref{direct-sum-closed-theorem}, one has
${}^{\perp_1}\C^A=\Fil_\omega(A\ot_R\F)^\oplus$.
 In particular, it follows that\/ $\Fil(A\ot_R\F)^\oplus=
\Fil_\omega(A\ot_R\F)^\oplus$.
\end{cor}

\begin{proof}
 This is a corollary of Theorem~\ref{direct-sum-closed-theorem}
and  Lemma~\ref{orthogonal-to-induced-lemma}(c) (cf.\ the proof of
Corollary~\ref{finite-coresol-dim-cor}).
\end{proof}

\begin{rem} \label{semimodules-remark}
 As mentioned in Remark~\ref{corings-and-contramodules-remark},
the condition~($\dagger$) appears to be rather restrictive.
 In fact, the construction of
Proposition~\ref{direct-sum-closed-approximation-sequences} originates
from the theory of semimodules over semialgebras, as
in~\cite[Lemma~1.3.3]{Psemi}, where the natural analogue of this
condition feels much less restrictive, particularly when $\C$ is
simply the class of all injective objects.
 So one can say that the ring $R$ in this
Section~\ref{increasing-filtrations-subsecn} really ``wants'' to be
a coalgebra $C$ (say, over a field~$k$), and accordingly the ring $A$
becomes a semialgebra $S$ over~$C$.
 The left $R$\+modules ``want'' to be left $C$\+comodules, and
the left $A$\+modules ``want'' to be left $S$\+semimodules.

 Then the induction functor, which was $A\ot_R{-}$ in
the condition~($\dagger$), takes the form of the cotensor product
functor $S\oc_C{-}$.
 This one is much more likely to take injective left $C$\+comodules
to injective left $C$\+comodules (it suffices that $S$ be
an injective left $C$\+comodule).
 Besides, the class of all injective comodules over a coalgebra over
a field is always closed under infinite direct sums; so the specific
assumption of Section~\ref{increasing-filtrations-subsecn} is
satisfied in the comodule context, too.

 To make a ring $R$ behave rather like a coalgebra, one can assume it
to be ``small'' in some sense.
 The following examples are suggested by the analogy with semialgebras
and semimodules.
\end{rem}

\begin{exs} \label{semimodules-examples}
 Let $\C=R\Modl_\inj$ be the class of all injective left $R$\+modules;
then $\F=R\Modl$ is the class of all left $R$\+modules
(cf.\ Example~\ref{corings-and-contramodules-example}).

\smallskip
 (1)~Assume that the ring $R$ is left Noetherian.
 Then the class of all injective left $R$\+modules is closed under
infinite direct sums; so the specific assumption of
Section~\ref{increasing-filtrations-subsecn} is satisfied.

 Let $I$ be an injective left $R$\+module containing every
indecomposable injective left $R$\+module as a direct summand.
 Assume further that the left $R$\+module $A\ot_RI$ is injective.
 Then it follows that the functor $A\ot_R{-}$ preserves the class
of all injective left $R$\+modules.
 Thus the condition~($\dagger$) is satisfied.

\medskip
 (2)~Assume that $R$ is a finite-dimensional algebra over a field~$k$.
 This is a particular case of~(1).
 Furthermore, the injective left $R$\+module $I=R^*=\Hom_k(R,k)$ has
the property that every injective left $R$\+module is a direct
summand of a direct sum of copies of~$R^*$.
 Therefore, the condition~($\dagger$) holds whenever the underlying
left $R$\+module of the left $A$\+module $A\ot_RR^*$ is injective.

\medskip
 (3)~Assume that $R$ is a quasi-Frobenius ring.
 This is also a particular case of~(1)
(cf.\ Example~\ref{semicontramodules-examples}(3)).
 In this case, the condition~($\dagger$) can be rephrased by saying
that the functor $A\ot_R{-}$ takes projective left $R$\+modules to
projective left $R$\+modules.
 This holds whenever $A$ is a projective left $R$\+module.
\end{exs}

\begin{rem}
 The above examples shed some light on the condition~$(\dagger)$,
but they provide no new information from the point of view of
the comparison between the results of
Section~\ref{increasing-filtrations-subsecn} and
those known from the general theory of cotorsion pairs in module
categories.
 In fact, taking $\C$ to be the class of all injective left
$R$\+modules and assuming that the ring $R$ is left Noetherian, one can
drop the condition~($\dagger$) altogether, as the following version
of Proposition~\ref{direct-sum-closed-approximation-sequences},
and consequently also of Theorem~\ref{direct-sum-closed-theorem} and
Corollary~\ref{direct-sum-closed-cor}, is readily provable using
the small object argument.
\end{rem}

\begin{prop} \label{small-object-argument-is-better-for-noetherian}
 Let $\C$ be the class of all injective left $R$\+modules.
 Assume that the ring $R$ is left Noetherian and the left $R$\+module
$A^+$ is injective (equivalently, the right $R$\+module $A$ is flat).
 Then the\/ $\Ext^1$\+orthogonal pair of classes of left $A$\+modules\/
$\Fil_\omega(A\ot_RR\Modl)$ and $\C^A=A\Modl_{R\dinj}$ admits
approximation sequences.

 Consequently, the pair of classes
$\F^A=\Fil_\omega(A\ot_RR\Modl)^\oplus$ and $\C^A$ is a hereditary
complete cotorsion pair in $A\Modl$.
 In particular, ${}^{\perp_1}A\Modl_{R\dinj}=
\Fil_\omega(A\ot_RR\Modl)^\oplus$ and\/
$\Fil(A\ot_RR\Modl)^\oplus=\Fil_\omega(A\ot_RR\Modl)^\oplus$.
\end{prop}

\begin{proof}
 This is a particular case of
Proposition~\ref{small-object-argument-is-better-for-sfp} below.
\end{proof}

 In other words, in the assumptions of
Proposition~\ref{small-object-argument-is-better-for-noetherian},
the weakly $A/R$\+projective left $A$\+modules are precisely the direct
summands of the $A$\+modules admitting an $\omega$\+indexed increasing
filtration by $A$\+modules induced from left $R$\+modules.

 Let us formulate a generalization of 
Proposition~\ref{small-object-argument-is-better-for-noetherian}
to coherent rings.
 For this purpose, we need some definitions.
 Let $R$ be an associative ring.
 A left $R$\+module $J$ is called \emph{fp\+injective}~\cite{Ste}
if $\Ext^1_R(S,J)=0$ for all finitely presented left $R$\+modules~$S$.
 A left $R$\+module $Q$ is said to be \emph{fp\+projective}~\cite{MD}
if $\Ext^1_R(Q,J)=0$ for all fp\+injective left $R$\+modules~$J$;
equivalently, this means that $Q$ is a direct summand of an $R$\+module
filtered by finitely presented $R$\+modules.
 The fp\+projective modules are called ``direct summands of
fp\+filtered modules'' in~\cite[Definition~8.4]{GT}.
 The classes of fp\+projective and fp\+injective $R$\+modules form
a complete cotorsion pair in $R\Modl$ \,\cite[Theorem~8.6(b)]{GT};
over a left coherent ring $R$, this cotorsion pair is hereditary.
 The class of all fp\+injective left $R$\+modules is closed under
direct sums~\cite[Corollary~2.4]{Ste}; over a left coherent ring $R$,
it is also closed under direct limits~\cite[Theorem~3.2]{Ste}.
 We refer to~\cite[Section~1]{Pfp} for some further discussion.

\begin{prop} \label{small-object-argument-for-coherent}
 Let $\F$ be the class of all fp\+projective left $R$\+modules
and $\C$ be the class of all fp\+injective left $R$\+modules.
 Assume that the ring $R$ is left coherent and the left $R$\+module
$A^+$ is fp\+injective (equivalently, the right $R$\+module $A$
is flat).
 Then the\/ $\Ext^1$\+orthogonal pair of classes of left $A$\+modules\/
$\Fil_\omega(A\ot_R\F)$ and $\C^A$ admits approximation sequences.

 Consequently, the pair of classes
$\F^A=\Fil_\omega(A\ot_R\F)^\oplus$ and $\C^A$ is a hereditary
complete cotorsion pair in $A\Modl$.
 In particular, ${}^{\perp_1}\C^A=
\Fil_\omega(A\ot_R\F)^\oplus$ and\/
$\Fil(A\ot_R\nobreak\F)^\oplus=\Fil_\omega(A\ot_R\F)^\oplus$.
\end{prop}

\begin{proof}
 This is also a particular case of
Proposition~\ref{small-object-argument-is-better-for-sfp}.
 Let us only explain why the right $R$\+module $A$ is flat whenever
the left $R$\+module $A^+$ is fp\+injective.
 This can be shown by observing that any fp\+injective pure-injective
module is injective (by~\cite[Lemma~8.5]{GT}), but the simplest
argument is that, given a right $R$\+module $G$ such that
the left $R$\+module $G^+$ is fp\+injective, the isomorphism
$\Tor^R_1(G,S)^+\simeq\Ext_R^1(S,G^+)=0$ implies $\Tor^R_1(S,G)=0$
for all finitely presented left $R$\+modules~$S$.
\end{proof}

\begin{ex}
 This example is an $n=\infty$ version of
Example~\ref{tilting-examples}(1).
 Let $A$ be an associative algebra over a commutative ring $R$ such that
$A$ is a flat $R$\+module, and let $(\F,\C)$ be a hereditary cotorsion
pair in $R\Modl$ generated by a set $\S$ of strongly finitely presented
$R$\+modules (i.~e., every module $S\in\S$ admits a resolution by
finitely generated projective $R$\+modules).
 Then the class $\C$ is closed under direct limits (and in particular,
direct sums) in $R\Modl$, so the condition~($\dagger$) holds for
the reason explained in Example~\ref{tilting-examples}(1), and
Theorem~\ref{direct-sum-closed-theorem} is applicable.
 According to Corollary~\ref{direct-sum-closed-cor}, we can conclude
that ${}^{\perp_1}\C^A=\F^A=\Fil_\omega(A\ot_R\F)^\oplus$.

 Using the small object argument, one can get rid of the assumption of
commutativity of the ring $R$ in this result, and relax the other
conditions as follows.
\end{ex}

\begin{prop} \label{small-object-argument-is-better-for-sfp}
 Let $R\rarrow A$ be a homomorphism of associative rings, and let
$(\F,\C)$ be a cotorsion pair in $R\Modl$ generated by a set $\S$
of left $R$\+modules such that an exact sequence of left $R$\+modules
$P_2\rarrow P_1\rarrow P_0\rarrow S\rarrow0$ with finitely generated
projective $R$\+modules $P_2$, $P_1$, $P_0$ exists for every $S\in\S$.
 Assume that the left $R$\+module $A^+$ belongs to~$\C$.
 Then the\/ $\Ext^1$\+orthogonal pair of classes of left $A$\+modules\/
$\Fil_\omega(A\ot_R\F)$ and $\C^A$ admits approximation sequences.

 Consequently, the pair of classes\/ $\F^A=\Fil_\omega(A\ot_R\F)^\oplus$
and $\C^A$ is a complete cotorsion pair in $A\Modl$.
 In particular, ${}^{\perp_1}\C^A=\Fil_\omega(A\ot_R\F)^\oplus$ and\/
$\Fil(A\ot_R\F)^\oplus=\Fil_\omega(A\ot_R\F)^\oplus$.
\end{prop}

\begin{proof}
 The proof is a simple version of the small object
argument~\cite[Theorem~2]{ET}, \cite[Theorem~6.11]{GT}.
 The claim that $\omega$\+filtrations by induced modules
are sufficient follows from Lemma~\ref{homological-formulas-lemma}(a)
and the fact that the functor $\Ext_R^1(S,{-})$ preserves direct
limits for any left $R$\+module $S$ satisfying the assumption of
the proposition.
 So, in fact, all the $A$\+modules from $\F^A$ are direct summands of
$A$\+modules $\omega$\+filtered by left $A$\+modules induced from
direct sums of copies of left $R$\+modules from~$\S$.
 We leave the details to the reader.
\end{proof}

 The next theorem is a generalization of
Corollary~\ref{direct-sum-closed-cor} in which the condition~($\dagger$)
is replaced by the condition~($\widetilde\dagger$).

\begin{thm} \label{direct-sum-closed-tilde}
 Let $(\F,\C)$ be a hereditary complete cotorsion pair in $R\Modl$.
 Assume that the left $R$\+module $A^+$ belongs to~$\C$, and
that the condition~\textup{($\widetilde\dagger$)} holds.
 Assume further that the class $\C$ is closed under countable direct
sums in $R\Modl$.
 Then the class $\F^A={}^{\perp_1}\C^A\subset A\Modl$ can be described
as $\F^A=\Fil_\omega(A\ot_R\F)^\oplus$.
 In particular, we have\/ $\Fil(A\ot_R\F)^\oplus=
\Fil_\omega(A\ot_R\F)^\oplus$.
\end{thm}

\begin{proof}
 Similar to the proof of Theorem~\ref{finite-coresol-dim-tilde},
which contains all the essential details.
 One follows the proof of Corollary~\ref{direct-sum-closed-cor} step by
step and observes that the assumptions of the present theorem are
sufficient for the validity of the argument.
\end{proof}

\subsection{Combined result on induced modules}
\label{combined-induced-subsecn}
 In this section we combine the results of
Propositions~\ref{finite-coresol-dim-approximation-sequences}
and~\ref{direct-sum-closed-approximation-sequences} in order to
obtain a more general result under relaxed assumptions.
 Specifically, we assume that all the countable direct sums of
modules from $\C$ have finite $\C$\+coresolution dimensions.

\begin{prop} \label{combined-induced-approximation-sequences}
 Assume that the\/ $\Ext^1$\+orthogonal pair of classes of
left $R$\+modules $(\F,\C)$ admits approximation
sequences~\textup{(\ref{sp-precover-seq}\+-\ref{sp-preenvelope-seq})}.
 Assume that the left $R$\+module $A^+$ belongs to~$\C$, and
that the condition~\textup{($\dagger$)} holds.
 Assume further that the class $\C$ is coresolving in $R\Modl$ and
the $\C$\+coresolution dimension of any countable direct sum of modules
from $\C$ does not exceed a finite integer~$k\ge0$.
 Then the\/ $\Ext^1$\+orthogonal pair of classes of left $A$\+modules\/
$\Fil_{\omega+k}(A\ot_R\F)$ and $\C^A$ admits approximation
sequences as well.
 Here $\omega+k$ is the $k$\+th successor ordinal of~$\omega$.
\end{prop}

\begin{proof}
 As in previous proofs, we start with an explicit construction of
special preenvelope sequences for the pair of classes
$\Fil_{\omega+k}(A\ot_R\F)$ and $\C^A\subset A\Modl$.

 Let $N$ be a left $A$\+module.
 Proceeding as in the proof of
Proposition~\ref{direct-sum-closed-approximation-sequences}, we
construct the $\omega$\+indexed inductive system of injective
morphisms of left $A$\+modules~\eqref{W-inductive-system}.
 The underlying left $R$\+module of the left $A$\+module
$\varinjlim_{m\in\omega}W^m(N)$ is isomorphic to the inductive limit
of the inductive system of left
$R$\+modules~\eqref{C-W-inductive-system}, and it can be described
as the rightmost term of the short exact
sequence~\eqref{indlimit-telescope}.

 The left $R$\+modules $A\ot_RC(W^m(N))$ belong to~$\C$
by~($\dagger$), so the left $R$\+module $\bigoplus_{m\in\omega}
A\ot_RC(W^m(N))$ has $\C$\+coresolution dimension~$\le k$ in
our present assumptions.
 By Lemma~\ref{co-resolution-dimension-defined-classes}(b), it
follows that the $\C$\+coresolution dimension of (the underlying
left $R$\+module of the left $A$\+module)
$M=\varinjlim_{m\in\omega}W^m(N)$ does not exceed~$k$.

 Now we apply the construction from the proof of
Proposition~\ref{finite-coresol-dim-approximation-sequences} to
the left $A$\+module $M$, producing the sequence of injective
morphisms of left $A$\+modules
$$
 M\lrarrow W(M)\lrarrow W(W(M))\lrarrow\dotsb\lrarrow W^k(M).
$$
 Following the argument in the proof of
Proposition~\ref{finite-coresol-dim-approximation-sequences},
we have $W^k(M)\in\C^A$, since $\cd_\C M\le k$.
 Finally, the cokernel of the composition of injective morphisms
$$
 N\lrarrow\varinjlim\nolimits_{m\in\omega}W^m(N)=M\lrarrow W^k(M)
$$
is an extension of the cokernels of the morphisms $N\rarrow
\varinjlim_{m\in\omega}W^m(N)$ and $M\rarrow W^k(M)$.
 The former cokernel belongs to $\Fil_\omega(A\ot_R\F)$ and
the latter one to $\Fil_k(A\ot_R\nobreak\F)$; thus the cokernel of
the morphism $N\rarrow W^k(M)$ belongs to $\Fil_{\omega+k}(A\ot_R\F)$.

 We have produced the desired special preenvelope sequences.
 Using these, the special precover sequences are constructed in
the same way as in the proofs of
Propositions~\ref{finite-coresol-dim-approximation-sequences}
and~\ref{direct-sum-closed-approximation-sequences}.
\end{proof}

\begin{thm} \label{combined-induced-theorem}
 Let $(\F,\C)$ be a hereditary complete cotorsion pair in $R\Modl$.
 Assume that the left $R$\+module $A^+$ belongs to~$\C$, and
that the condition~\textup{($\dagger$)} holds.
 Assume further that the $\C$\+coresolution dimension of any
countable direct sum of modules from $\C$ in $R\Modl$ does not
exceed a finite integer $k\ge0$.
 Then the pair of classes $\F^A=\Fil_{\omega+k}(A\ot_R\F)^\oplus$ and
$\C^A$ is a hereditary complete cotorsion pair in $A\Modl$.
\end{thm}

\begin{proof}
 Follows from
Proposition~\ref{combined-induced-approximation-sequences}
in view of Lemma~\ref{direct-summand-lemma}.
\end{proof}

\begin{cor} \label{combined-induced-cor}
 For any associative ring homomorphism $R\rarrow A$ and any
hereditary complete cotorsion pair $(\F,\C)$ in $R\Modl$ satisfying
the assumptions of Theorem~\ref{combined-induced-theorem}, one has
${}^{\perp_1}\C^A=\Fil_{\omega+k}(A\ot_R\F)^\oplus$.
 In particular, it follows that\/ $\Fil(A\ot_R\nobreak\F)^\oplus=
\Fil_{\omega+k}(A\ot_R\F)^\oplus$.
\end{cor}

\begin{proof}
 This is a corollary of Theorem~\ref{combined-induced-theorem}
and Lemma~\ref{orthogonal-to-induced-lemma}(c).
\end{proof}

 For a class of examples to Theorem~\ref{combined-induced-theorem}
arising from curved DG\+rings, see
Proposition~\ref{star-coacyclic-prop} below.

 The final theorem of this section is a generalization of
Corollary~\ref{combined-induced-cor} in which the condition~($\dagger$)
is replaced by the condition~($\widetilde\dagger$).

\begin{thm} \label{combined-induced-tilde}
 Let $(\F,\C)$ be a hereditary complete cotorsion pair in $R\Modl$.
 Assume that the left $R$\+module $A^+$ belongs to~$\C$, and
that the condition~\textup{($\widetilde\dagger$)} holds.
 Assume further that the $\C$\+coresolution dimension of any
countable direct sum of modules from $\C$ in $R\Modl$ does not
exceed a finite integer $k\ge0$.
 Then the class $\F^A={}^{\perp_1}\C^A\subset A\Modl$ can be described
as $\F^A=\Fil_{\omega+k}(A\ot_R\F)^\oplus$.
 In particular, we have\/ $\Fil(A\ot_R\nobreak\F)^\oplus=
\Fil_{\omega+k}(A\ot_R\F)^\oplus$.
\end{thm}

\begin{proof}
 One follows the proof of Corollary~\ref{combined-induced-cor} step by
step and observes that the assumptions of the present theorem are
sufficient for the validity of the argument.
 Almost all the essential details have been presented already in
the proof of Theorem~\ref{finite-coresol-dim-tilde}, and only one
observation remains to be made.

 Let $N$ be a left $A$\+module whose underlying left $R$\+module belongs
to~$\G$.
 Then the underlying left $R$\+module of the left $A$\+module
$\varinjlim_{m\in\omega}W^m(N)$ also belongs to $\G$, because
the cokernel of the injective $A$\+module morphism $N\rarrow
\varinjlim_{m\in\omega}W^m(N)$ belongs to $\Fil_\omega(A\ot_R\F)\subset
\F^A$ and the class $\G\subset R\Modl$ is closed under extensions.
\end{proof}

\Section{Illustration: Contraderived and Coderived Categories}
\label{contra-co-derived-secn}

 The heading above starts with the word ``illustration'' rather than
``application'', because there are few new results in this section
(Theorems~\ref{countable-cofiltrations-contraacyclic}
and~\ref{countable-filtrations-coacyclic} being notable exceptions).
 Still we demonstrate some classes of examples where
Theorems~\ref{finite-resol-dim-theorem},
\ref{combined-coinduced-theorem}, \ref{finite-coresol-dim-theorem},
and~\ref{combined-induced-theorem} are applicable, leading to
nontrivial conclusions, even if previously known to be provable
with different methods.

\subsection{Curved DG-rings and modules} \label{curved-subsecn}
 A \emph{curved DG\+ring} (\emph{CDG\+ring}) $R=(R,d,h)$ is a graded ring
$R=\bigoplus_{n\in\Z}R^n$ endowed with an odd derivation $d\:R\rarrow R$
of degree~$1$ and a curvature element $h\in R^2$.
 These words mean that, for every $n\in\Z$, an abelian group
homomorphism $d_n\:R^n\rarrow R^{n+1}$ is specified such that
the equation $d(rs)=d(r)s+(-1)^{|r|}rd(s)$ holds for all
$r\in R^{|r|}$ and $s\in R^{|s|}$, \ $|r|$, $|s|\in\Z$.
 In addition, the following two equations need to be satisfied:
\begin{enumerate}
\renewcommand{\theenumi}{\roman{enumi}}
\item $d(d(r))=hr-rh$ for all $r\in R$;
\item $d(h)=0$.
\end{enumerate}
 The element $h\in R^2$ is called the \emph{curvature element}.
 A curved DG\+ring with $h=0$ is the same thing as a usual
DG\+ring (differential graded ring) $(R,d)$.

 A \emph{left CDG\+module} $M=(M,d_M)$ over a CDG\+ring $(R,d,h)$ is
a graded left $R$\+module $M=\bigoplus_{n\in\Z}M^n$ endowed with
an odd derivation $d_M\:M\rarrow M$ compatible with the derivation~$d$
on~$R$.
 These words mean that, for every $n\in\Z$, an abelian group
homomorphism $d_{M,n}\:M^n\rarrow M^{n+1}$ is specified such that
the equation $d_M(rm)=d(r)m+(-1)^{|r|}rd_M(m)$ holds for all
$r\in R^{|r|}$ and $m\in M^{|m|}$.
 In addition, the following equation needs to be satisfied:
\begin{enumerate}
\renewcommand{\theenumi}{\roman{enumi}}
\setcounter{enumi}{2}
\item $d_M(d_M(m))=hm$ for all $m\in M$.
\end{enumerate}

 Notice that \emph{CDG\+rings and CDG\+modules are not complexes},
due to the presence of a nontrivial right-hand side in the equations~(i)
and~(iii).
 Nevertheless, for any two left CDG\+modules $L$ and $M$ over $(R,d,h)$,
the \emph{complex of morphisms} $\Hom_R(L,M)$ is well-defined.
 This is a complex of abelian groups whose degree~$i$ component
$\Hom_R^i(L,M)$ is the group of all homomorphisms of graded left
$R$\+modules $L\rarrow M[i]$, where $[i]$~denotes the usual
cohomological grading shift $M[i]^n=M^{i+n}$.
 There is a sign rule involved in the definition of
the left $R$\+module structure on $M[i]$.
 We refer to~\cite[Sections~1.1 and~3.1]{Pkoszul} for the details.
 
 One can assign a graded ring $A$ to a CDG\+ring $(R,d,h)$ by adjoining
a new element $\delta\in A^1$ to the graded ring $R$ and imposing
the relations $\delta r-(-1)^{|r|}r\delta=d(r)$ for all $r\in R^{|r|}$
and $\delta^2=h$.
 The elements of the grading components $A^n$ are the formal expressions
$r+\delta s$ with $r\in R^n$ and $s\in R^{n-1}$, with the multiplication
of such formal expressions defined in the obvious way using the above
relations.
 With an appropriate definition of morphisms of CDG\+rings and
a natural structure (the differential $\partial=\partial/\partial\delta$)
on the graded ring $A$, the correspondence between CDG\+rings $(R,d,h)$
and acyclic DG\+rings $(A,\partial)$ becomes an equivalence of
categories (see~\cite[Section~4.2]{Prel}, where the notation is
$B=R$ and $\widehat B=A$).

 We denote the abelian categories of graded left modules over the graded
rings $R$ and $A$ by $R\Modl^\sgr$ and $A\Modl^\sgr$, respectively.
 As usually in module theory, all the results above in this paper can
be extended easily from the categories of modules to the categories
of graded modules.
 The abelian category $R\Modl^\cdg$ of left CDG\+modules over $(R,d,h)$,
with homogeneous morphisms of degree~$0$ commuting with the action of
$R$ and the differentials on the CDG\+modules, is equivalent to
the abelian category of graded $A$\+modules $A\Modl^\sgr$.
 The group of morphisms $L\rarrow M$ in this category is isomorphic
to the kernel of the differential $\Hom_R^0(L,M)\rarrow\Hom_R^1(L,M)$.

 Notice that the graded ring $A$ is a finitely generated projective
graded left and right $R$\+module.
 In fact, it is a free graded left $R$\+module with two generators~$1$
and~$\delta$, and it is also a free graded right $R$\+module with
the same two generators.
 The functors $A\ot_R\nobreak{-}\,\:\allowbreak R\Modl^\sgr\rarrow
A\Modl^\sgr=R\Modl^\cdg$ and $\Hom_R(A,{-})\:R\Modl^\sgr\rarrow
A\Modl^\sgr=R\Modl^\cdg$ are described
in~\cite[proof of Theorem~3.6]{Pkoszul}, where they are denoted by
$G^+=A\ot_R{-}$ and $G^-=\Hom_R(A,{-})$.
 The two functors only differ by a shift of grading: for every graded
left $R$\+module $S$, there is a natural isomorphism of graded left
$A$\+modules $G^-(S)=G^+(S)[1]$.

 It is an easy but important observation that all the CDG\+modules
in the essential image of the functor $G^+$, or equivalently, $G^-$
are \emph{contractible}.
 In other words, all of them represent zero objects in the homotopy
category of CDG\+modules $\Hot(R\Modl^\cdg)$
(cf.~\cite[Section~3.2]{Pkoszul}).
 The natural action of the differential
$\partial=\partial/\partial\delta$ in $G^+(S)$ and $G^-(S)$, induced
by the action of~$\partial$ in $A$, provides a contracting homotopy.

 The \emph{homotopy category of left CDG\+modules} $\Hot(R\Modl^\cdg)$
is defined by the rule $\Hom_{\Hot(R\Modl^\cdg)}(L,M)=H^0(\Hom_R(L,M))$;
so $\Hot(R\Modl^\cdg)$ is the degree-zero cohomology category of
the DG\+category of left CDG\+modules over $(R,d,h)$, with the complexes
of morphisms $\Hom_R(L,M)$ between CDG\+modules $L$ and~$M$.
 The homotopy category $\Hot(R\Modl^\cdg)$ is a triangulated category
with infinite direct sums and
products~\cite[Sections~1.2 and~3.1]{Pkoszul}.

\subsection{Contraderived category} \label{contraderived-subsecn}
 A left CDG\+module $P$ over $(R,d,h)$ is said to be
\emph{graded projective} if the graded left $R$\+module $P$ is
projective in $R\Modl^\sgr$.
 We denote the full subcategory of graded projective CDG\+modules
by $R\Modl^\cdg_\proj=A\Modl^\sgr_{R\dproj}\subset A\Modl^\sgr=
R\Modl^\cdg$ and the corresponding full subcategory in the homotopy
category by $\Hot(R\Modl^\cdg_\proj)\subset\Hot(R\Modl^\cdg)$.

 A left CDG\+module $X$ over $(R,d,h)$ is said to be \emph{contraacyclic
in the sense of Becker}~\cite{Bec} if the complex $\Hom_R(P,X)$
is acyclic for all graded projective CDG\+modules
$P\in R\Modl^\cdg_\proj$, or equivalently,
$\Hom_{\Hot(R\Modl^\cdg)}(P,X)=0$ for all $P\in\Hot(R\Modl^\cdg_\proj)$.
 We denote the full subcategory of contraacyclic CDG\+modules by
$R\Modl^{\cdg,\ctr}_\acycl\subset R\Modl^\cdg$ and the corresponding
full subcategory in the homotopy category by
$\Hot(R\Modl^{\cdg,\ctr}_\acycl)\subset\Hot(R\Modl^\cdg)$.
 Clearly, $\Hot(R\Modl^{\cdg,\ctr}_\acycl)$ is a triangulated subcategory
closed under infinite products in $\Hot(R\Modl^\cdg)$.

\begin{thm} \label{becker-contra}
 Let $(R,d,h)$ be a CDG\+ring and $A=R[\delta]$ be the corresponding
graded ring.
 Then \par
\textup{(a)} the pair of classes of objects $R\Modl^\cdg_\proj$ and
$R\Modl^{\cdg,\ctr}_\acycl$ is a hereditary complete cotorsion pair in
the abelian category $R\Modl^\cdg=A\Modl^\sgr$; \par
\textup{(b)} the composition of the triangulated inclusion functor\/
$\Hot(R\Modl^\cdg_\proj)\rarrow\Hot(R\Modl^\cdg)$ and the triangulated
Verdier quotient functor\/ $\Hot(R\Modl^\cdg)\rarrow
\Hot(R\Modl^\cdg)/\allowbreak\Hot(R\Modl^{\cdg,\ctr}_\acycl)$
is a triangulated equivalence\/ $\Hot(R\Modl^\cdg_\proj)\simeq
\Hot(R\Modl^\cdg)/\allowbreak\Hot(R\Modl^{\cdg,\ctr}_\acycl)$.
\end{thm}

\begin{proof}
 This is~\cite[Propositions~1.3.6(1) and~1.3.8(1)]{Bec}.
 Parts~(a) and~(b) are two closely related assertions; in fact,
(b)~follows from~(a).
 We skip the details.
\end{proof}

 The quotient category $\sD^\ctr(R\Modl^\cdg)=
\Hot(R\Modl^\cdg)/\allowbreak\Hot(R\Modl^{\cdg,\ctr}_\acycl)$
is called the \emph{contraderived category} of left CDG\+modules
over $(R,d,h)$ \emph{in the sense of Becker}.
 It has to be distinguished from the contraderived category
in the sense of the books and papers~\cite{Psemi,Pkoszul,Pfp,Prel}
(see~\cite[Example~2.6(3)]{Pps} for a discussion).
 It is an open question whether the two definitions of
a contraderived category are equivalent for an arbitrary CDG\+ring.
 In this section we explain how one can show that they are, in fact,
equivalent under certain assumptions.

 To any pair of morphisms with zero composition $K\rarrow L$ and
$L\rarrow M$ in the category $R\Modl^\cdg=A\Modl^\sgr$,
one can assign its totalization $\Tot(K\to L\to M)$, which is
an object of the same category.
 The construction of the CDG\+module $\Tot(K\to L\to M)$ is
a generalization of the construction of the total complex of
a bicomplex with three rows; it can be interpreted as a twisted
direct sum or an iterated cone in the DG\+category of CDG\+modules.
 We refer to~\cite[Section~1.2]{Pkoszul} for a discussion.
 Specifically, we are interested in totalizations of \emph{short
exact sequences} in the abelian category $R\Modl^\cdg=A\Modl^\sgr$.

\begin{prop} \label{totalizations-contraacyclic}
 Let $(R,d,h)$ be a CDG\+ring.
 Then the totalization of any short exact sequence of left
CDG\+modules over $(R,d,h)$ belongs to $R\Modl^{\cdg,\ctr}_\acycl$.
 Hence the minimal full triangulated subcategory of the homotopy
category\/ $\Hot(R\Modl^\cdg)$ containing the totalizations of
short exact sequences of CDG\+modules and closed under products
is a subcategory in\/ $\Hot(R\Modl^{\cdg,\ctr}_\acycl)$.
\end{prop}

\begin{proof}
 This is the result of~\cite[Theorem~3.5(b)]{Pkoszul}.
\end{proof}

 We start with a rather general lemma concerning applicability of
the results of Section~\ref{cofiltrations-secn} to our injective
morphism of graded rings $R\rarrow A$.

\begin{lem} \label{doubledagger-holds-for-curved}
 Let $(R,d,h)$ be a CDG\+ring and $A=R[\delta]$ be the corresponding
graded ring.
 Then the (graded version of) condition~\textup{($\dagger\dagger$)}
from Section~\ref{finite-by-coinduced-subsecn} holds for \emph{any}
cotorsion pair $(\F,\C)$ in $R\Modl^\sgr$ that is invariant under
the degree shift\/~$[1]$.
 In other words, the underlying graded left $R$\+module of the left
CDG\+module $G^-(F)=\Hom_R(A,F)$ belongs to $\F$ for any graded
left $R$\+module $F\in\F$.
\end{lem}

\begin{proof}
 The $R$\+$R$\+bimodule $A/R$ is isomorphic to $R[-1]$ (with
appropriate sign rules).
 Hence for any graded left $R$\+module $F$ there is a short exact
sequence of graded left $R$\+modules $0\rarrow F[1]\rarrow G^-(F)
\rarrow F\rarrow0$.
 Now $F\in\F$ and $F[1]\in\F$ imply $G^-(F)\in\F$, since
the class $\F$ is closed under extensions in $R\Modl^\sgr$.
\end{proof}

 We will apply the results of Sections~\ref{finite-by-coinduced-subsecn}
and~\ref{combined-coinduced-subsecn} to the following (trivial)
cotorsion pair $(\F,\C)$ in the category of graded left $R$\+modules
$R\Modl^\sgr$.
 Take $\F=R\Modl_\proj^\sgr$ to be the class of all projective graded
left $R$\+modules and $\C=R\Modl^\sgr$ to be the class of all graded
left $R$\+modules (as in Examples~\ref{corings-and-comodules-example}
and~\ref{semicontramodules-examples}).
 In the spirit of the notation in Section~\ref{cofiltrations-secn}, we
denote by $G^-(R\Modl^\sgr)=\Hom_R(A,R\Modl^\sgr)$ the class of all
left CDG\+modules over $(R,d,h)$ of the form $G^-(S)$
with $S\in R\Modl^\sgr$.

\begin{prop} \label{finite-homol-dim-contraacyclic}
 Let $(R,d,h)$ be a CDG\+ring.
 Assume that the abelian category of graded left $R$\+modules
$R\Modl^\sgr$ has finite homological dimension~$k$.
 Then one has\/ $R\Modl^{\cdg,\ctr}_\acycl=
\Cof_{k+1}(G^-(R\Modl^\sgr))^\oplus\subset R\Modl^\cdg$.
\end{prop}

\begin{proof}
 In the notation of Section~\ref{cofiltrations-secn}, we have
$\F_A=A\Modl^\sgr_{R\dproj}=R\Modl^\cdg_\proj$.
 Hence, by Theorem~\ref{becker-contra}(a),
\,$\C_A=R\Modl^{\cdg,\ctr}_\acycl$.
 The assumptions of Theorem~\ref{finite-resol-dim-theorem} hold
in view of Lemma~\ref{doubledagger-holds-for-curved}, and
it remains to apply Corollary~\ref{finite-resol-dim-cor}.
\end{proof}

 The following condition from~\cite[Section~3.8]{Pkoszul} ensures
applicability of Theorem~\ref{combined-coinduced-theorem}:
\begin{itemize}
\item[($**$)] any countable product of projective graded left
$R$\+modules, viewed as a graded left $R$\+module, has finite
projective dimension not exceeding a fixed integer~$k$.
\end{itemize}

\begin{prop} \label{doublestar-contraacyclic-prop}
 Let $(R,d,h)$ be a CDG\+ring.
 Assume that the graded ring $R$ satisfies
the condition~\textup{($**$)}.
 Then one has\/ $R\Modl^{\cdg,\ctr}_\acycl=
\Cof_{\omega+k}(G^-(R\Modl^\sgr))^\oplus\subset R\Modl^\cdg$.
\end{prop}

\begin{proof}
 Similar to Proposition~\ref{finite-homol-dim-contraacyclic}.
 The condition~($\dagger\dagger$) holds by
Lemma~\ref{doubledagger-holds-for-curved}, and the desired assertion
is obtained by comparing Theorem~\ref{becker-contra}(a) with
Corollary~\ref{combined-coinduced-cor}.
\end{proof}

\begin{lem} \label{finite-cofiltrations-absacyclic}
 Let $(R,d,h)$ be a CDG\+ring.
 Let $\T$ be a class of objects in $R\Modl^\cdg$ and $k$~be
a finite integer.
 Then any object from\/ $\Cof_{k+1}(\T)\subset R\Modl^\cdg$, viewed
as an object of the homotopy category\/ $\Hot(R\Modl^\cdg)$, belongs
to the minimal full triangulated subcategory of\/ $\Hot(R\Modl^\cdg)$
containing the CDG\+modules from $\T$ and the totalizations of
short exact sequences in $R\Modl^\cdg$.
\end{lem}

\begin{proof}
 It suffices to observe that, for any short exact sequence
$0\rarrow K\rarrow L\rarrow M\rarrow0$ in $R\Modl^\cdg$, the object
$L$ belongs to the minimal triangulated subcategory
of $\Hot(R\Modl^\cdg)$ containing $K$, $M$, and $\Tot(K\to L\to M)$.
\end{proof}

\begin{thm} \label{countable-cofiltrations-contraacyclic}
 Let $\T$ be a class of objects in $R\Modl^\cdg$ and $\alpha$~be
a countable ordinal.
 Then any object from\/ $\Cof_\alpha(\T)\subset R\Modl^\cdg$,
viewed as an object of the homotopy category\/ $\Hot(R\Modl^\cdg)$,
belongs to the minimal full triangulated subcategory of\/
$\Hot(R\Modl^\cdg)$ containing the CDG\+modules from $\T$ and
the totalizations of short exact sequences in $R\Modl^\cdg$,
and closed under countable products.
\end{thm}

\begin{proof}
 Denote by $\X\subset\Hot(R\Modl^\cdg)$ the minimal triangulated
subcategory containing the CDG\+modules from $\T$ and the totalizations
of short exact sequences in $R\Modl^\cdg$, and closed under
countable products.
 Let us first consider the case $\alpha=\omega$.
 Let $M=G_\omega M\twoheadrightarrow \dotsb\twoheadrightarrow G_nM
\twoheadrightarrow\dotsb\twoheadrightarrow G_2M\twoheadrightarrow
G_1M\twoheadrightarrow G_0M=0$, \
$G_\omega M=\varprojlim_{n\in\omega}G_nM$, be an $\omega$\+cofiltration
of a left CDG\+module $M$ over $(R,d,h)$ by CDG\+modules from $\T$
(or even from~$\X$).
 Then we have a short exact sequence
\begin{equation} \label{cofiltration-telescope}
 0\lrarrow M\lrarrow\prod\nolimits_{n\in\omega}G_nM\lrarrow
 \prod\nolimits_{n\in\omega}G_nM\lrarrow0
\end{equation}
in the abelian category of CDG\+modules $R\Modl^\cdg=A\Modl^\sgr$.
 By Lemma~\ref{finite-cofiltrations-absacyclic}, we have
$G_nM\in\X$ for all the integers $n\ge0$.
 Hence $\prod_{n\in\omega}G_nM\in\X$.
 Since the totalization of the short exact
sequence~\eqref{cofiltration-telescope} also belongs to $\X$,
it follows that $M\in\X$.

 In the general case of a countable ordinal~$\alpha$, we proceed by
transfinite induction in~$\alpha$.
 Let $M=G_\alpha M\twoheadrightarrow\dotsb\twoheadrightarrow G_1M
\twoheadrightarrow G_0M=0$ be a CDG\+module $\alpha$\+cofiltered
by CDG\+modules from~$\T$.
 We need to show that $M\in\X$.
 The case $\alpha=0$ is obvious.
 If $\alpha=\beta+1$ is a successor ordinal, then we have a short
exact sequence of CDG\+modules $0\rarrow T\rarrow M\rarrow
G_\beta M\rarrow0$ with $T\in\T$.
 By the induction assumption, $G_\beta M\in\X$.
 Since $\Tot(T\to M\to G_\beta M)\in\X$, it follows that $M\in\X$.

 In the case of a countable limit ordinal~$\alpha$, choose an increasing
sequence of ordinals $(\beta_n)_{n\in\omega}$ with
$0=\beta_0<\beta_1<\beta_2<\dotsb<\alpha$ and $\alpha=
\lim_{n\to\omega}\beta_n$.
 Then there exist ordinals $(\gamma_n>0)_{n\in\omega}$ such that
$\beta_{n+1}=\beta_n+\gamma_n$ for every $n\in\omega$.
 It follows that $\gamma_n\le\beta_{n+1}<\alpha$.
 Define an $\omega$\+cofiltration $G'$ on the CDG\+module $M$ by
the rule $G'_nM=G_{\beta_n}M$.
 Then the kernel $K_n$ of the surjective morphism of CDG\+modules
$G'_{n+1}M\rarrow G'_nM$ is $\gamma_n$\+cofiltered by $\T$ for every
$n\in\omega$.
 By the induction assumption, we have $K_n\in\X$ for every $n\in\omega$.
 According to the above argument for the case of
an $\omega$\+cofiltration, it follows that $M\in\X$.
\end{proof}

\begin{cor} \label{finite-homol-dim-contra=absderived}
 Let $(R,d,h)$ be a CDG\+ring.
 Assume that the abelian category of graded left $R$\+modules
$R\Modl^\sgr$ has finite homological dimension.
 Then\/ $\Hot(R\Modl^{\cdg,\ctr}_\acycl)$ is the minimal thick
subcategory in\/ $\Hot(R\Modl^\cdg)$ containing the totalizations of
short exact sequences of CDG\+modules.
\end{cor}

\begin{proof}
 This is essentially ``the contraderived half''
of~\cite[Theorem~3.6]{Pkoszul}.
 Here is a proof based on the techniques developed in this paper.
 Let $\X_0\subset\Hot(R\Modl^\cdg)$ be the minimal thick subcategory
containing the totalizations of short exact sequences of CDG\+modules.
 Then $\X_0\subset\Hot(R\Modl^{\cdg,\ctr}_\acycl)$
by Proposition~\ref{totalizations-contraacyclic} (this assertion
does not depend on any assumptions on the ring~$R$).

 The (nontrivial) inclusion $\Hot(R\Modl^{\cdg,\ctr}_\acycl)\subset\X_0$
holds by Proposition~\ref{finite-homol-dim-contraacyclic}
and Lemma~\ref{finite-cofiltrations-absacyclic}.
 Here we use the observation, mentioned in Section~\ref{curved-subsecn},
that the CDG\+module $G^-(S)$ is contractible for every graded
$R$\+module~$S$.
\end{proof}

\begin{cor} \label{doublestar-contraderived-cor}
 Let $(R,d,h)$ be a CDG\+ring.
 Assume that the graded ring $R$ satisfies
the condition~\textup{($**$)}.
 Then\/ $\Hot(R\Modl^{\cdg,\ctr}_\acycl)$ is the minimal full
triangulated subcategory in\/ $\Hot(R\Modl^\cdg)$ containing
the totalizations of short exact sequences of CDG\+modules
over $(R,d,h)$ and closed under countable products.
\end{cor}

\begin{proof}
 This is essentially~\cite[Theorem~3.8]{Pkoszul}.
 Here is a proof based on the results of this paper.
 Let $\X\subset\Hot(R\Modl^\cdg)$ be the minimal thick subcategory
containing the totalizations of short exact sequences of CDG\+modules
and closed under countable products.
 Then $\X\subset\Hot(R\Modl^{\cdg,\ctr}_\acycl)$
by Proposition~\ref{totalizations-contraacyclic} (this assertion
does not depend on any assumptions on the ring~$R$).

 The (nontrivial) inclusion $\Hot(R\Modl^{\cdg,\ctr}_\acycl)\subset\X$
holds by Proposition~\ref{doublestar-contraacyclic-prop}
and Theorem~\ref{countable-cofiltrations-contraacyclic}.
 The observation that the CDG\+module $G^-(S)$ is contractible
for every graded $R$\+module $S$ is important here.

 Finally, notice that any full triangulated subcategory having
countable products is a thick subcategory by Rickard's
criterion~\cite[Criterion~1.3]{Neem} and the B\"okstedt--Neeman
theorem~\cite[Proposition~3.2 or Remark~3.3]{BN}.
\end{proof}

 The contraderived category of left CDG\+modules over $(R,d,h)$ in
the sense of the books and papers~\cite{Psemi,Pkoszul,Pfp,Pps,Prel}
is defined as the quotient category of the homotopy category
$\Hot(R\Modl^\cdg)$ by its minimal triangulated subcategory containing
the totalizations of short exact sequences of CDG\+modules and closed
under infinite products.
 Thus Corollary~\ref{doublestar-contraderived-cor} can be rephrased by
saying that, \emph{under the condition}~($**$), \emph{the contraderived
category in the sense of Becker coincides with the contraderived
category in the sense of}~\cite{Psemi,Pkoszul,Pfp,Pps,Prel}.

\subsection{Coderived category} \label{coderived-subsecn}
 A left CDG\+module $J$ over $(R,d,h)$ is said to be \emph{graded
injective} if the graded left $R$\+module $J$ is injective in
$R\Modl^\sgr$.
 We denote the full subcategory of graded injective CDG\+modules by
$R\Modl^\cdg_\inj=A\Modl^\sgr_{R\dinj}\subset A\Modl^\sgr=R\Modl^\cdg$
and the corresponding full subcategory in the homotopy category by
$\Hot(R\Modl^\cdg_\inj)\subset\Hot(R\Modl^\cdg)$.

 A left CDG\+module $X$ over $(R,d,h)$ is said to be \emph{coacyclic
in the sense of Becker}~\cite{Bec} if the complex $\Hom_R(X,J)$ is
acyclic for all graded injective CDG\+modules $J\in R\Modl^\cdg_\inj$,
or equivalently, $\Hom_{\Hot(R\Modl^\cdg)}(X,J)=0$ for all
$J\in\Hot(R\Modl^\cdg_\inj)$.
 We denote the full subcategory of coacyclic CDG\+modules by
$R\Modl^{\cdg,\co}_\acycl\subset R\Modl^\cdg$ and the corresponding
full subcategory in the homotopy category by
$\Hot(R\Modl^{\cdg,\co}_\acycl)\subset\Hot(R\Modl^\cdg)$.
 Clearly, $\Hot(R\Modl^{\cdg,\co}_\acycl)$ is a triangulated subcategory
closed under infinite direct sums in $\Hot(R\Modl^\cdg)$.
{\hbadness=1175\par}

\begin{thm} \label{becker-co}
 Let $(R,d,h)$ be a CDG\+ring and $A=R[\delta]$ be the corresponding
graded ring.
 Then \par
\textup{(a)} the pair of classes of objects $R\Modl^{\cdg,\co}_\acycl$
and $R\Modl^\cdg_\inj$ is a hereditary complete cotorsion pair in
the abelian category $R\Modl^\cdg=A\Modl^\sgr$; \par
\textup{(b)} the composition of the triangulated inclusion functor\/
$\Hot(R\Modl^\cdg_\inj)\rarrow\Hot(R\Modl^\cdg)$ and the triangulated
Verdier quotient functor\/ $\Hot(R\Modl^\cdg)\rarrow
\Hot(R\Modl^\cdg)/\allowbreak\Hot(R\Modl^{\cdg,\co}_\acycl)$
is a triangulated equivalence\/ $\Hot(R\Modl^\cdg_\inj)\simeq
\Hot(R\Modl^\cdg)/\allowbreak\Hot(R\Modl^{\cdg,\co}_\acycl)$.
\end{thm}

\begin{proof}
 This is~\cite[Propositions~1.3.6(2) and~1.3.8(2)]{Bec}.
 Parts~(a) and~(b) are closely related; in fact, (b)~follows from~(a).
 We omit the details.
\end{proof}

 The quotient category $\sD^\co(R\Modl^\cdg)=
\Hot(R\Modl^\cdg)/\allowbreak\Hot(R\Modl^{\cdg,\co}_\acycl)$
is called the \emph{coderived category} of left CDG\+modules
over $(R,d,h)$ \emph{in the sense of Becker}.
 It needs to be distinguished from the coderived category
in the sense of the books and papers~\cite{Psemi,Pkoszul,Pfp,Prel}
(see~\cite[Example~2.5(3)]{Pps} for a discussion).
 It is an open question whether the two definitions of
a coderived category are equivalent for an arbitrary CDG\+ring.
 In this section we explain how one can show that they are, in fact,
equivalent under certain assumptions.

\begin{prop} \label{totalizations-coacyclic}
 Let $(R,d,h)$ be a CDG\+ring.
 Then the totalization of any short exact sequence of left
CDG\+modules over $(R,d,h)$ belongs to $R\Modl^{\cdg,\co}_\acycl$.
 Hence the minimal full triangulated subcategory of the homotopy
category\/ $\Hot(R\Modl^\cdg)$ containing the totalizations of
short exact sequences of CDG\+modules and closed under direct sums
is a subcategory in\/ $\Hot(R\Modl^{\cdg,\co}_\acycl)$.
\end{prop}

\begin{proof}
 This is the result of~\cite[Theorem~3.5(a)]{Pkoszul}.
\end{proof}

 The following lemma is a rather general assertion concerning
applicability of the results of Section~\ref{filtrations-secn}
to our morphism of graded rings $R\rarrow A$.

\begin{lem} \label{dagger-holds-for-curved}
 Let $(R,d,h)$ be a CDG\+ring and $A=R[\delta]$ be the corresponding
graded ring.
 Then the (graded version of) condition~\textup{($\dagger$)}
from Section~\ref{finite-by-induced-subsecn} holds for \emph{any}
cotorsion pair $(\F,\C)$ in $R\Modl^\sgr$ that is invariant under
the degree shift\/~$[1]$.
 In other words, the underlying graded left $R$\+module of the left
CDG\+module $G^+(C)=A\ot_RC$ belongs to $\C$ for any graded
left $R$\+module $C\in\C$.
\end{lem}

\begin{proof}
 Similar to Lemma~\ref{doubledagger-holds-for-curved}.
 For any graded left $R$\+module $C$ there is a short exact sequence
of graded left $R$\+modules $0\rarrow C\rarrow G^+(C)
\rarrow C[-1]\rarrow0$.
 Now $C\in\C$ and $C[-1]\in\C$ imply $G^+(C)\in\C$, since
the class $\C$ is closed under extensions in $R\Modl^\sgr$.
\end{proof}

 We will apply the results of Sections~\ref{finite-by-induced-subsecn}
and~\ref{combined-induced-subsecn} to the following (trivial) cotorsion
pair $(\F,\C)$ in the category of graded left $R$\+modules
$R\Modl^\sgr$.
 Take $\C=R\Modl^\sgr_\inj$ to be the class of all injective graded
left $R$\+modules and $\F=R\Modl^\sgr$ to be the class of all graded
left $R$\+modules (as in
Examples~\ref{corings-and-contramodules-example}
and~\ref{semimodules-examples}).
 Following the notation of Sections~\ref{filtrations-secn}
and~\ref{contraderived-subsecn}, we denote by $G^+(R\Modl^\sgr)=
A\ot_R R\Modl^\sgr$ the class of all CDG\+modules over $(R,d,h)$
of the form $G^+(S)$ with $S\in R\Modl^\sgr$.

\begin{prop} \label{finite-homol-dim-coacyclic}
 Let $(R,d,h)$ be a CDG\+ring.
 Assume that the abelian category of graded left $R$\+modules
$R\Modl^\sgr$ has finite homological dimension~$k$.
 Then one has\/ $R\Modl^{\cdg,\co}_\acycl=
\Fil_{k+1}(G^+(R\Modl^\sgr))^\oplus\subset R\Modl^\cdg$.
\end{prop}

\begin{proof}
 In the notation of Section~\ref{filtrations-secn}, we have
$\C^A=A\Modl^\sgr_{R\dinj}=R\Modl^\cdg_\inj$.
 Hence, by Theorem~\ref{becker-co}(a),
\,$\F^A=R\Modl^{\cdg,\co}_\acycl$.
 The assumptions of Theorem~\ref{finite-coresol-dim-theorem}
hold in view of Lemma~\ref{dagger-holds-for-curved}, and it remains
to apply Corollary~\ref{finite-coresol-dim-cor}.
\end{proof}

\begin{cor} \label{finite-homol-dim-co=contra}
 For any CDG\+ring $(R,d,h)$ such that the abelian category of graded
left $R$\+modules has finite homological dimension, the classes of
contraacyclic and coacyclic left CDG\+modules in the sense of Becker
over $(R,d,h)$ coincide, $R\Modl^{\cdg,\ctr}_\acycl=
R\Modl^{\cdg,\co}_\acycl$.
\end{cor}

\begin{proof}
 This is our version of~\cite[Theorem~3.6(a)]{Pkoszul}.
 It is provable by comparing the results of
Propositions~\ref{finite-homol-dim-contraacyclic}
and~\ref{finite-homol-dim-coacyclic}.
 We have $G^-(R\Modl^\sgr)=G^+(R\Modl^\sgr)$, since $G^-=G^+[1]$;
and, obviously, $\Cof_{k+1}(\T)=\Fil_{k+1}(\T)$ for any class
$\T\subset A\Modl^\sgr$ and any finite integer~$k$.
\end{proof}

\begin{cor} \label{finite-homol-dim-co=absderived}
 Let $(R,d,h)$ be a CDG\+ring.
 Assume that the abelian category of graded left $R$\+modules
$R\Modl^\sgr$ has finite homological dimension.
 Then\/ $\Hot(R\Modl^{\cdg,\co}_\acycl)$ is the minimal thick
subcategory in\/ $\Hot(R\Modl^\cdg)$ containing the totalizations of
short exact sequences of CDG\+modules.
\end{cor}

\begin{proof}
 This is ``the coderived half'' of~\cite[Theorem~3.6]{Pkoszul}.
 It can be deduced, e.~g., by comparing the results of
Corollaries~\ref{finite-homol-dim-contra=absderived}
and~\ref{finite-homol-dim-co=contra}.
\end{proof}

 The following condition from~\cite[Section~3.7]{Pkoszul} ensures
applicability of Theorem~\ref{combined-induced-theorem}:
\begin{itemize}
\item[($*$)] any countable direct sum of injective graded left
$R$\+modules, viewed as a graded left $R$\+module, has finite
injective dimension not exceeding a fixed integer~$k$.
\end{itemize}

\begin{prop} \label{star-coacyclic-prop}
 Let $(R,d,h)$ be a CDG\+ring.
 Assume that the graded ring $R$ satisfies
the condition~\textup{($*$)}.
 Then one has\/ $R\Modl^{\cdg,\co}_\acycl=
\Fil_{\omega+k}(G^+(R\Modl^\sgr))^\oplus\subset R\Modl^\cdg$.
\end{prop}

\begin{proof}
 Similar to Proposition~\ref{finite-homol-dim-coacyclic}.
 The condition~($\dagger$) holds by Lemma~\ref{dagger-holds-for-curved},
and the desired assertion is obtained by comparing
Theorem~\ref{becker-co}(a) with Corollary~\ref{combined-induced-cor}.
\end{proof}

\begin{thm} \label{countable-filtrations-coacyclic}
 Let $\T$ be a class of objects in $R\Modl^\cdg$ and $\alpha$~be
a countable ordinal.
 Then any object from\/ $\Fil_\alpha(\T)\subset R\Modl^\cdg$,
viewed as an object of the homotopy category\/ $\Hot(R\Modl^\cdg)$,
belongs to the minimal full triangulated subcategory of\/
$\Hot(R\Modl^\cdg)$ containing the CDG\+modules from $\T$ and
the totalizations of short exact sequences in $R\Modl^\cdg$,
and closed under countable direct sums.
\end{thm}

\begin{proof}
 This is the dual version of
Theorem~\ref{countable-cofiltrations-contraacyclic}.
 Denote by $\X\subset\Hot(R\Modl^\cdg)$ the minimal triangulated
subcategory containing the CDG\+modules from $\T$ and the totalizations
of short exact sequences in $R\Modl^\cdg$, and closed under
countable direct sums.
 Let us consider the case $\alpha=\omega$.
 Let $0=F_0M\subset F_1M\subset F_2M\subset\dotsb F_nM\subset\dotsb
\subset F_\omega M=M$, \ $F_\omega M=\bigcup_{n\in\omega}F_nM$, be
an $\omega$\+filtration of a left CDG\+module $M$ over $(R,d,h)$ by
CDG\+modules from~$\X$.
 Then we have a short exact sequence
\begin{equation} \label{filtration-telescope}
 0\lrarrow \bigoplus\nolimits_{n\in\omega}F_nM\lrarrow
 \bigoplus\nolimits_{n\in\omega}F_nM\lrarrow M\lrarrow0
\end{equation}
in the abelian category of CDG\+modules $R\Modl^\cdg=A\Modl^\sgr$.
 By Lemma~\ref{finite-cofiltrations-absacyclic}, which is applicable
because $\Fil_{k+1}(\X)=\Cof_{k+1}(\X)$, we have $F_nM\in\X$ for
all the integers $n\ge0$.
 Hence $\bigoplus_{n\in\omega}F_nM\in\X$.
 Since the totalization of the short exact
sequence~\eqref{filtration-telescope} also belongs to $\X$,
it follows that $M\in\X$.
 The argument for an arbitrary countable ordinal~$\alpha$ is
similar to that in Theorem~\ref{countable-cofiltrations-contraacyclic}.
\end{proof}

\begin{cor} \label{star-coderived-cor}
 Let $(R,d,h)$ be a CDG\+ring.
 Assume that the graded ring $R$ satisfies
the condition~\textup{($*$)}.
 Then\/ $\Hot(R\Modl^{\cdg,\co}_\acycl)$ is the minimal full
triangulated subcategory in\/ $\Hot(R\Modl^\cdg)$ containing
the totalizations of short exact sequences of CDG\+modules
over $(R,d,h)$ and closed under countable direct sums.
\end{cor}

\begin{proof}
 This is essentially~\cite[Theorem~3.7]{Pkoszul}.
 Here is a proof based on the results of this paper.
 Let $\X\subset\Hot(R\Modl^\cdg)$ be the minimal thick subcategory
containing the totalizations of short exact sequences of CDG\+modules
and closed under countable direct sums.
 Then $\X\subset\Hot(R\Modl^{\cdg,\co}_\acycl)$
by Proposition~\ref{totalizations-coacyclic} (this assertion
does not depend on any assumptions on the ring~$R$).

 The (nontrivial) inclusion $\Hot(R\Modl^{\cdg,\co}_\acycl)\subset\X$
holds by Proposition~\ref{star-coacyclic-prop}
and Theorem~\ref{countable-filtrations-coacyclic}.
 The observation that the CDG\+module $G^+(S)$ is contractible for
every graded $R$\+module $S$ needs to be used here.

 Finally, any full triangulated subcategory having countable
direct sums is a thick subcategory by Rickard's
criterion~\cite[Criterion~1.3]{Neem} and the B\"okstedt--Neeman
theorem~\cite[Proposition~3.2 or Remark~3.3]{BN}.
\end{proof}

 The coderived category of left CDG\+modules over $(R,d,h)$ in the sense
of the books and papers~\cite{Psemi,Pkoszul,Pfp,Pps,Prel} is defined as
the quotient category of the homotopy category $\Hot(R\Modl^\cdg)$ by
its minimal triangulated subcategory containing the totalizations of
short exact sequences of CDG\+modules and closed under infinite
direct sums.
 Thus Corollary~\ref{star-coderived-cor} can be rephrased by saying
that, \emph{under the condition}~($*$), \emph{the coderived category in
the sense of Becker coincides with the coderived category in the sense
of}~\cite{Psemi,Pkoszul,Pfp,Pps,Prel}.

\end{document}